\documentclass[
    fontsize=11pt, 
	paper=a4,
	chapterprefix=true,
	WordMark=atelier/logo/FAU-NatFak.pdf,
	ExtraLogo=atelier/logo/DepMathLogo.pdf,
	setPDF, 
    DIV=17,
    bibstyle=styleB,
    ]{book}
\usepackage{graphicx} 

\date{\today}

\usepackage{url}
\usepackage{hyperref}       
\usepackage{setspace}
\usepackage[T1]{fontenc}
\usepackage{lmodern}
\usepackage{orcidlink}
\author{Danielle Bednarski\orcidlink{0000-0001-9440-4468} and Tim Roith\orcidlink{0000-0001-8440-2928}}
\semester{------------------------}
\department{Helmholtz Imaging, Deutsches Elektronen-Synchrotron DESY, Notkestr. 85, 22607
Hamburg, Germany}
\description{------------------------}
\content{%
\begin{minipage}{\textwidth}%
\includegraphics[width=.3\textwidth]{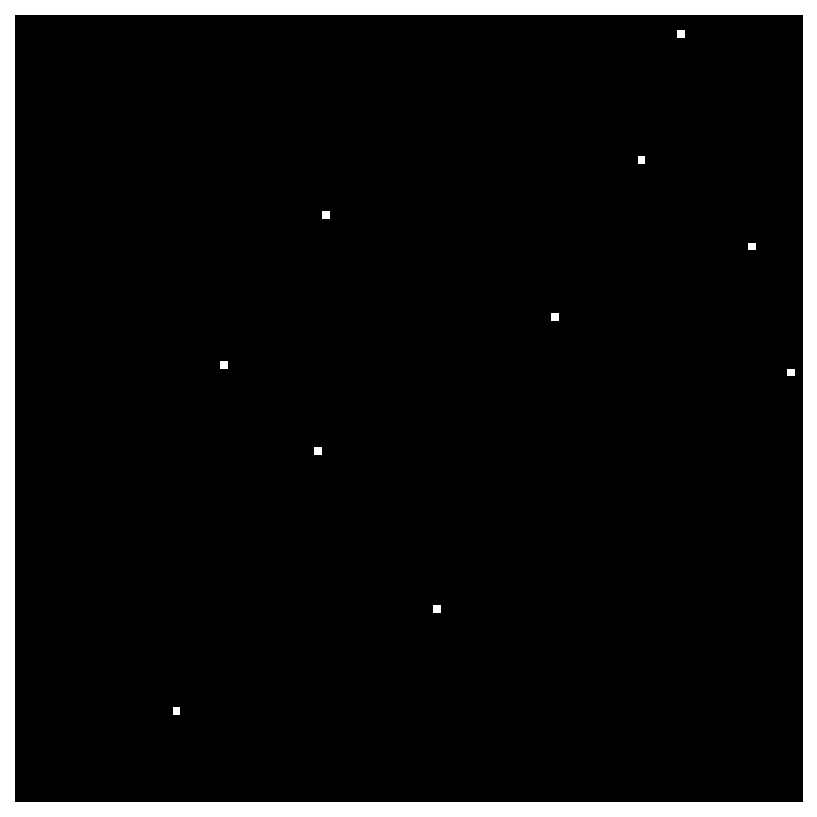}\hfill%
\includegraphics[width=.3\textwidth]{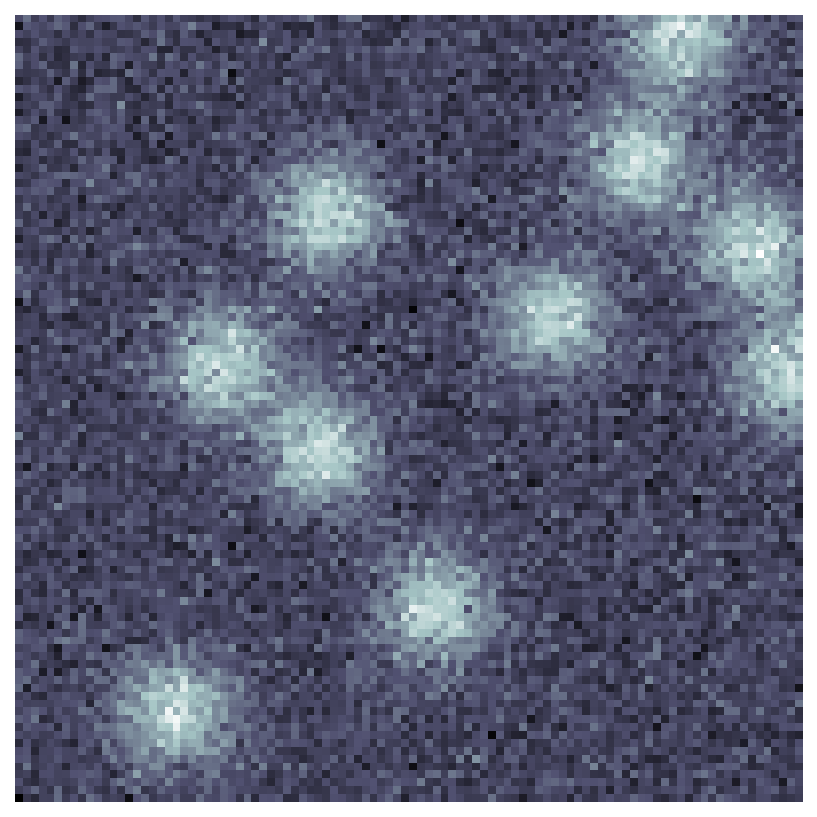}\hfill%
\includegraphics[width=.3\textwidth]{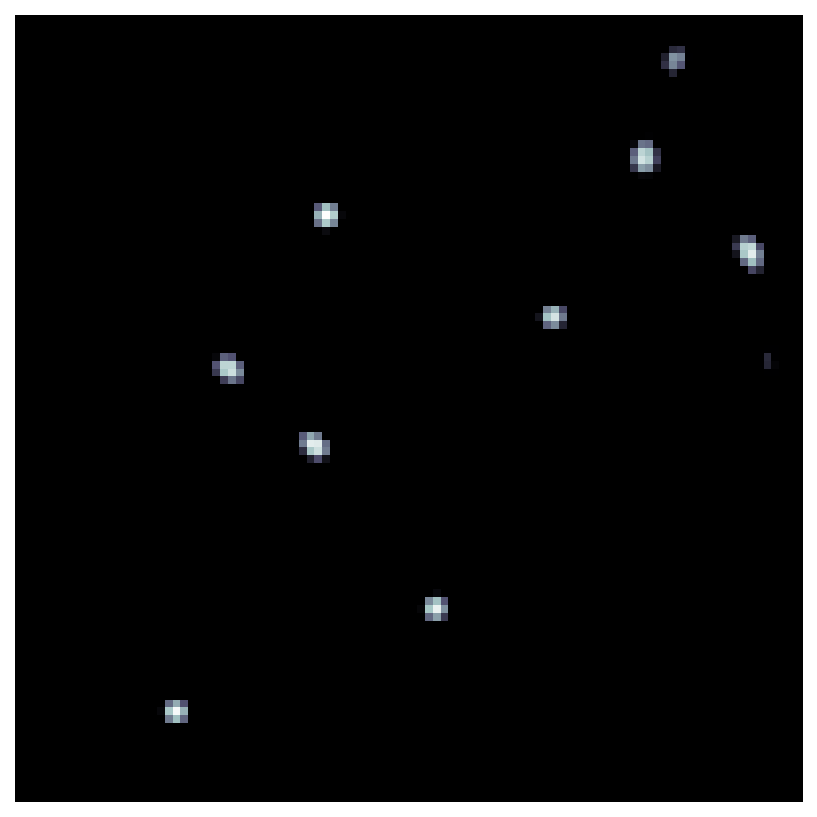}%
\end{minipage}%
\vspace{1em}%

These notes are based on a lecture taught by the authors in the winter semester 2024/2025 at the University of Hamburg.
}
\title{Introduction to Regularization and Learning Methods for Inverse Problems}
\setpdfinfo 
\usepackage[thmboxing=styleA,
			boxingstyle=styleA,
			chapterheader=styleA, 
			footerheader=styleA,
			]{frontmatter/appearence}
\usepackage{frontmatter/commands}
\usepackage{kantlipsum}
\usepackage[totoc]{idxlayout}
\usepackage{algorithm}
\usepackage{algorithmic}
\usepackage{eqparbox}
\newdimen{\algindent}
\usepackage{ulem}
\renewcommand{\emph}[1]{\textit{#1}}

\usepackage{mathtools}
\usepackage{tikz}
\usetikzlibrary{arrows}
\usetikzlibrary{angles, quotes}
\usepackage{pgfplots}
\usepgfplotslibrary{fillbetween}

\usepackage{nicematrix}
\usepackage{subcaption}
\usepackage{csquotes}
\usepackage{enumitem}

\crefname{prob}{Problem}{Probleme}
\crefname{alg}{Algorithmus}{Algorithmen}

\makeatletter
\newcommand{\fixed@sra}{$\vrule height 2\fontdimen22\textfont2 width 0pt\shortrightarrow$}
\newcommand{\shortarrow}[1]{%
  \mathrel{\text{\rotatebox[origin=c]{\numexpr#1*45}{\fixed@sra}}}
}
\makeatother

\usepackage{placeins}

\usepackage{caption}
\usepackage{booktabs}
\usepackage[font=small]{subcaption}
\usepackage{amsmath}
\usepackage{amssymb}

\usepackage{todonotes}
\makeatletter
\define@key{todonotes}{TR}[]{%
    \setkeys{todonotes}{author=TR,color=mediumelectricblue,size=scriptsize}}%
\makeatother 

\makeatletter
\define@key{todonotes}{DB}[]{%
    \setkeys{todonotes}{author=DB,color=green,size=scriptsize}}%
\makeatother

\newif\ifpub
\pubtrue
\begin{document}
\frontmatter
\maketitle

\tableofcontents

\clearpage

These lecture notes evolve around mathematical concepts arising in inverse problems. We start by introducing inverse problems through examples such as differentiation, deconvolution, computed tomography and phase retrieval. This then leads us to the framework of well-posedness and first considerations regarding reconstruction and inversion approaches. Apart from the references mentioned therein, the first chapter mainly uses material from the following lectures:
\begin{itemize}
\item \cite{Lellmann2022bild}: \emph{Mathematische Methoden der Bildverarbeitung} held by Jan Lellmann 2022/23 at the Universität zu Lübeck.
\item \cite{burger2007inverse}: \emph{Inverse problems} held by Martin Burger 2007/2008 at the Westfälische Wilhelms- Universität Münster.
\end{itemize}
The second chapter then first deals with classical regularization theory of inverse problems in Hilbert spaces. After introducing the pseudo-inverse, we review the concept of convergent regularization. While our notes here again take inspiration from \cite{burger2007inverse}, the main point of reference is the work of Engl, Hanke and Neubauer:
\begin{itemize}
\item \cite{Engl2000}: \fullcite{Engl2000}
\end{itemize}
Within this chapter we then proceed to ask the question of how to realize practical reconstruction algorithms. Here, we mainly focus on Tikhonov and sparsity promoting regularization in finite dimensional spaces. This shifts the lecture more towards optimization theory and convex analysis. Apart from standard textbooks, like \textcite{Bauschke2017,rockafellar2015convex}, we mainly take inspiration from the following lecture notes:
\begin{itemize}
\item \cite{Brandt22}: \emph{Convex optimization \& applications} held by Christina Brandt 2022 at the University of Hamburg.
\end{itemize}
In the third chapter, we dive into modern deep-learning methods, which allow solving inverse problems in a data-dependent approach. We first introduce the basic concepts of learning, where we take inspiration from \cite{roith2024consistency}. We then deal with question of how to extend the concept of convergent regularization to the data-driven setting. Here we follow the framework introduced in the following two papers:
\begin{itemize}
\item \cite{kabri2024convergent}: \fullcite{kabri2024convergent}
\item \cite{burger2023learned}: \fullcite{burger2023learned}
\end{itemize}
After that, we proceed by reviewing some practical data-driven approaches. The intersection between inverse problems and machine learning is a rapidly growing field and our exposition here restricts itself to a very limited selection of topics. Among them are learned regularization, fully-learned Bayesian estimation, post-processing strategies and plug-n-play methods. In this case we borrow ideas from the following overview paper, which we also refer to for a more in depth review of methods used in practice:
\begin{itemize}
\item \cite{arridge2019solving}: \fullcite{arridge2019solving}
\end{itemize}
Throughout these notes we highlight the importance of the Bayesian formulation of inverse problem. Unfortunately these notes do not go as far as to make use of this framework to illustrate the recent trend of employing diffusion models for reconstruction tasks, see e.g. \cite{Mardani2024Variational}. We intend to close this gap in future iterations of the lecture.
\\[1em]

Finally, we want to extend our thanks to the many people that shared crucial insights when preparing these notes. Among them, we want to thank Martin Burger, Samira Kabri and Audrey Repetti for the helpful discussions.

\clearpage

\ifpub%
\else%
\section{Organization}

\subsection{Time Table}

\begin{tabular}{c c c p{10cm}}
Date    &  KW & Chapter & Topic \\ \hline
[1] 15.10.  &  42 & Ch 1.1 - 1.2 & Was ist ein Inverses Problem? Was ist ein Bild (kontinuierlich/diskret)? 4 Beispiele mit Code und GUI. Auf Probleme hinweisen.
\\ \hline
[2] 22.10 & 43 & Ch 1.3 - 1.4 & Korrektgestelltheit mit numerischen Diff Beispiel und Code+GUI. Beginn Inversion Approaches. \\ \hline
\hl{[1] 24.10} & 43 &  &  \\ \hline
[3] 29.10 & 44 & Ch 1.4 - 2.1  & Fortsetzung Inversion Approaches und Generalized Inverses \\ \hline
[4] 05.11 & 45 & Ch. 2.2-2.3 & Generalized Inverses for Compact Operators and Regularization \\ \hline
\hline \hl{[2] 07.11} & 45 &  &  \\ \hline
[5] 12.11 & 46 & Ch. 2.3 & Regularization \\ 
[6] 19.11 & 47 & Ch. 2.4 & Convex analysis and Tikhonov regularization \\ \hline
\hl{[3] 21.11} & 47 &  &  \\ \hline
[7] 26.11 & 48 & Ch. 3.2 & Sparsity-based regulariaztion and ISTA \\ \hline
[8] 03.12 & 49 & Ch.  & Total-variation based denoising and iterative methods \\ \hline
\hl{[4] 05.12} & 49 &  &  \\ \hline
[9] 10.12 & 50 & Ch. 3.1 & Neural Networks, Approximation results, Perceptron, CNN, U-Net, Examples of NNs \\ \hline
[10] 17.12 & 51 &   & Backpropagation, Optimization, SGD, convergence SGD, momentum, Adam  \\ \hline
\hl{[] } & 51 & ENTFALLEN &  \\ \hline
[X] 24.12 & 52 && \\ \hline
[X] 31.12 & 52 && \\ \hline
[11] 07.01 & 02 &  Ch. 3.3 & Learning for IP: Post Processing, Unrolling, PnP, Deep Image Prior, RED \\ \hline
\hl{[5] 09.01} & 02 &  &  \\ \hline
[12] 14.01 & 03 & Ch. 3.2-3.3 & Convergent Data-Driven Regularization, Learning in statistical Regularization, \\ \hline
[13] 21.01 & 04 & Ch 3.3.1-3.3.2  & Fully learned Bayes Estimation, Post-Processing, \\ \hline
\hl{[6] 23.01} & 04  &  &  \\ \hline
[14] 28.01 & 05 &&  Unrolling, PnP \\ \hline
\end{tabular}

Vorlesung:

D: [1], [2], [3], [4], [5],  [12], [13]

T: [6], [7], [8], [9], [10], [11]

Übung:

D: [1], [2], [3]

T: [4], [5],

\fi

\mainmatter

\chapter{Introduction to Inverse Problems}

\emph{Inverse problems} describe the process of recovering a \alert{quantity of interest $u$} from observed \alert{data $f$}. The governing relation between these objects is the \alert{forward process $\mcA$}. Mathematically, the forward process is modeled by some \emph{operator} $\mcA:\mcU\to\mcV$ which maps between, e.g., Banach spaces $\mcU$ and $\mcV$. It describes how the observed data is obtained:
\beq
f = \mcA u.    
\eeq
In practice, measurements are often disturbed by some unknown noise $\noise$. In this case, the inverse problem is to find $u$ given the disturbed measurement 
\beq
    \fd = \mcA u + \eps,
\eeq
where we assume that the noise is bounded by a noise level $\noiselvl>0$, e.g., $\|\noise\|\leq\noiselvl$. We call $\|\fd - f\| = \|\noise\|$ the \emph{measurement error}.

\begin{memo}{Inverse Problems}{}
\begin{itemize}
\item[Q]: What is an inverse problem?
\item[A]: Obtain a quantity $u$, where only $f=\mcA u$ is given.
\end{itemize}
\end{memo}
\noindent%
In this course, we will assume that $\mcA$ is a linear operator if not stated differently. Note, however, that non-linear problems and related theory also exist in the literature and in practice. When we are explicitly working in a fully discrete setting, this will be indicated by the use of non-cursive symbols, e.g., $A$ can then always be understood as a matrix.

\section{Quick Reminder of Important Mathematical Definitions}

\begin{definition}{Banach and Hilbert Space}{}
Let $\mcU$ denote a vector space. 
\begin{enumerate}
    \item A mapping $\| \cdot \| :\mcU \to \R^+$ is a called \emph{norm} on $\mcU$ iff 
    \begin{itemize}
        \item[a)] the mapping is definite, i.e., for all $u\in\mcU$ it holds $\| u \| = 0 \Leftrightarrow u=0$,
        \item[b)] the mapping is absolute homogeneous, i.e., for all $u\in\mcU$ and $\lambda\in\R$ it holds $\| \lambda u\| = |\lambda| \| u \|$,
        \item[c)] and the mapping fulfills the triangle inequality, i.e., for all $u,v\in\mcU$ it holds $\| u + v \| \leq \| u \| + \| v \|$.
    \end{itemize}
    \item A normed vector space $\mcU$ is called \emph{complete} iff every Cauchy sequence converges in $\mcU$, i.e., for any sequence $(u_n)\subseteq\mcU$ which fulfills 
    \beq
         \forall \eps > 0, \, \exists k \in \N \,\text{ such that } \|u_m-u_n\|\leq \eps \, \text{ for all } m,n \geq k,
    \eeq
    there exists $u\in\mcU$ such that $\lim_{n\to\infty} \| u_n -u \| = 0$.
    \item A complete normed vector space is called \emph{Banach space}.
    \item A complete normed vector space which has a norm that is induced by an inner product is called a \emph{Hilbert space}.
\end{enumerate}
\end{definition}

\begin{definition}{Linear Operators}{}
Let $\mcU\subseteq\{u:\Omega \to \Gamma \}$ and $\mcV \subseteq\{v: \Omega' \to \Gamma' \}$ denote two Banach spaces.
\begin{enumerate}
    \item A mapping $f:\mcU\to\R$ is called \emph{functional}. 
    \item A mapping $\mcA: \mcU \to \mcV$ is called \emph{operator} between $\mcU$ and $\mcV$. We use the following notation:
    \beq
        \mcA: \mcU \to \mcV, \quad u \in \mcU \mapsto \mcA u := \mcA(u) \in \mcV.
    \eeq
    \item An operator is called \emph{bounded} iff
    \beq
        \| \mcA \| := \| \mcA \|_{\mcU,\mcV} := \sup_{u\in\mcU, \|u\|_\mcU = 1} \| \mcA u \|_\mcV < \infty.
    \eeq
    \item The set of linear bounded operators is denoted by
    \beq
        \mcL (\mcU, \mcV) := \ls \mcA:\mcU\to\mcV \, \middle| \, \mcA \text{ is linear and bounded} \rs.
    \eeq
\end{enumerate}
\end{definition}
Boundedness is an important property as it implies continuity of the operator. This is further elaborated in the following theorem: 
\begin{theorem}{}{thm:bounded_linear_operators}
    Let $\mcU, \mcV$ be normed spaces and let $\mcA:\mcU\to\mcV$ be linear. Then the following are equivalent:
    \begin{itemize}
        \item[a)] $\mcA$ is bounded, i.e., it holds $\mcA\in\mcL(\mcU, \mcV)$.
        \item[b)] $\mcA$ is continuous, i.e., it holds 
        \beq
            \|u^k - u\|_{\mcU} \overset{k\to\infty}{\longrightarrow} 0 \quad \Rightarrow \quad \| \mcA u^k - \mcA u \|_{\mcV} \overset{k\to\infty}{\longrightarrow} 0.
        \eeq
        \item[c)] If $U\subset \mcU$ is bounded in $\mcU$, then $\mcA(U)$ is bounded in $\mcV$.
    \end{itemize} 
\end{theorem}

\begin{exercise}{}{}
    Show Thm.\ref{thm:bounded_linear_operators}.

    \comment{
    \begin{proof}
        We show equivalence in four steps:
        \paragraph{a) $\Rightarrow$ b):} Assume that $A$ is bounded and $u^k$ converges to $u^*$. Then it holds 
        \beq
             \|Au^k-Au^* \|_{\mcV} = \|A(u^k-u^*) \|_{\mcV} \leq \underbrace{\|A\|_{\mcU,\mcV}}_{< \infty} \underbrace{\|u^k-u^*\|_{\mcU}}_{\to 0} \overset{k\to\infty}{\longrightarrow} 0.
        \eeq
        \paragraph{b) $\Rightarrow$ a):} Assume that $A$ is \textbf{not} bounded, i.e., there exists some sequence $(u^k)$ such that $\|Au^k\|_{\mcV} \geq k \|u^k\|_{\mcU}$. W.l.o.g. assume that $\|u^k\|_{\mcU}=1$. Then $\frac{1}{k}u^k \overset{k\to\infty}{\longrightarrow} 0$ but $\|A \frac{1}{k}u^k\|_{\mcV} \geq k \frac{1}{k} \|u^k\|_{\mcU} =1$ and thus $A$ is not continuous.
        \paragraph{a) $\Rightarrow$ c):} Assume that $A$ is bounded and we are given some $u$ such that $\|u\|_{\mcU}\leq C$. Then it directly follows $\|Au\|_{\mcV} \leq C\|A\|$. Thus, $Au$ is bounded in $\mcV$.
        \paragraph{c) $\Rightarrow$ a):} As $\overline{\mcB_1(0)}$ is bounded, there exists a $C$ such that $\|Au\|_{\mcV} \leq C$ for any $\|u\|_{\mcU} = 1$. It directly follows that $\|A\|_{\mcU, \mcV} = \sup_{u\in\mcU, \|u\|_\mcU = 1} \| Au \|_\mcV \leq C < \infty$. 
    \end{proof}
    }
    
\end{exercise}

\section{Some Examples of Inverse Problems}
\label{sec:examples_for_ip}

Before looking at some concrete examples of inverse problems we introduce some terminology and notation related to the imaging setting.

In the \emph{continuous setting} we write $u:\Omega\to\Gamma$, where 
\begin{itemize}
    \item the \emph{domain} $\Omega\subset\R^n$ is assumed to be open,
    \item the \emph{range} $\Gamma\subset\R^d$ is assumed to be compact.
\end{itemize}

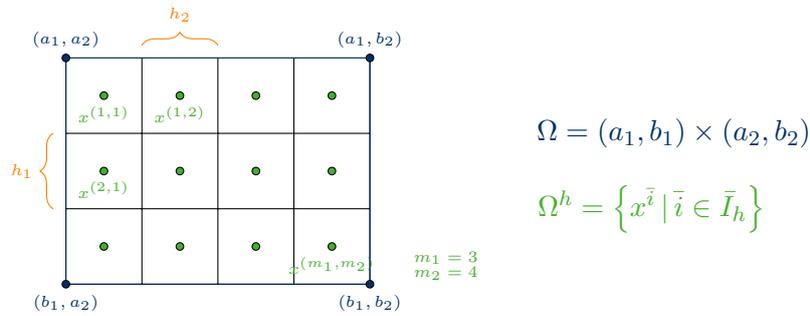
\begin{figure}
\begin{center}
\begin{tikzpicture}
    \draw[step=1cm,black,thin] (0,0) grid (4,3);

    \foreach \x/\y in {0.5/0.5, 0.5/1.5, 0.5/2.5, 1.5/0.5, 1.5/1.5, 1.5/2.5, 2.5/0.5, 2.5/1.5, 2.5/2.5, 3.5/0.5, 3.5/1.5, 3.5/2.5}{\draw [fill=NatGreen, thin] (\x,\y) circle [radius=0.05];}
    
    \node[NatGreen] at (0.5,2.25)  {\tiny $x^{(1,1)}$};
    \node[NatGreen] at (0.5,1.25) {\tiny $x^{(2,1)}$};
    \node[NatGreen] at (1.5,2.25) {\tiny $x^{(1,2)}$};
    \node[NatGreen] at (3.5,0.25) {\tiny $x^{(m_1,m_2)}$};

    \node[NatGreen] at (5,0.35) {\tiny $m_1=3$};
    \node[NatGreen] at (5,0.15) {\tiny $m_2=4$};

    \draw[FAUBlue, thin] (0,0) -- (4,0);
    \draw[FAUBlue, thin] (4,0) -- (4,3);
    \draw[FAUBlue, thin] (4,3) -- (0,3);
    \draw[FAUBlue, thin] (0,3) -- (0,0);
    
    \draw[fill=FAUBlue, thin] (0, 0) circle [radius=0.05];
    \node[FAUBlue] at (0,-0.25) {\tiny $(b_1,a_2)$};

    \draw[fill=FAUBlue, thin] (4, 0) circle [radius=0.05];
    \node[FAUBlue] at (4,-0.25) {\tiny $(b_1,b_2)$};
    
    \draw[fill=FAUBlue, thin] (0, 3) circle [radius=0.05];
    \node[FAUBlue] at (0,3.25) {\tiny $(a_1,a_2)$};

    \draw[fill=FAUBlue, thin] (4, 3) circle [radius=0.05];
    \node[FAUBlue] at (4,3.25) {\tiny $(a_1,b_2)$};

    \draw [orange,decorate,decoration={brace,amplitude=5pt,raise=1ex}] (1,3) -- (2,3) node[midway,yshift=1.5em]{\tiny $h_2$};

    \draw [orange,decorate,decoration={brace,amplitude=5pt,raise=1ex}] (0,1) -- (0,2) node[midway,xshift=-1.5em]{\tiny $h_1$};

    \node[FAUBlue] at (8,2) {$\Omega = (a_1,b_1) \times (a_2,b_2)$};
    \node[NatGreen] at (7.7,1) {$\Omega^h = \ls x^{\bar{i}} \,|\, \bar{i} \in \bar{I}_h\rs$};
    
\end{tikzpicture} 
\end{center}
    \caption{\textbf{Cell-centered grid.} For an open domain $\Omega$, the cell-centered grid with equidistant discretization using the step sizes $h_i$ is given by the green dots.}
    \label{fig:discrete_grid}
\end{figure}

In the \emph{discrete setting} we write $u_h:\Omega_h\to\Gamma$. If not specified differently, we assume that $\Omega_h$ is a \emph{cell-centered} grid. The cells of the grid are called \emph{pixels} in the 2D and \emph{voxels} in the 3D setting. To make it more precise, assume that the continuous domain is given by $\Omega=(a_1,b_1)\times(a_2,b_2)\times\cdots\times(a_n,b_n)$. In the discrete setting this carries over as follows: 
\begin{itemize}
    \item We denote by $m=(m_1,m_2,\ldots,m_n)$ the number of cells in each direction.
    \item The step size of the grid is denoted by $h=(h_1,\ldots,h_n)$, where $h_i=\frac{b_i-a_i}{m_i}$.
    \item The cells can be numerated using a multiindex $\bar{i}\in \bar{I}_h := \{1,\ldots,m_1\}\times\cdots\times\{1,\ldots,m_n\}$ and the cell-centered coordinates are given by $x^{\bar{i}}=x^{(\bar{i}_1,\ldots,\bar{i}_n)}$, $(x^{\bar{i}})_j = a_j + (\bar{i}_j - \frac{1}{2}) h_j$.
    \item The set  $\Omega_h := \{x^{\bar{i}}|\bar{i}\in \bar{I}_h\}$ then is called the cell-centered grid.
    \item The discrete image  $u_h:\Omega_h \to \Gamma$ assigns each cell a value and we can write it as a matrix
    \beq
    u_h =
    \begin{pmatrix}u_h^{(1,1)} & \cdots & u_h^{(1,m_2)} \\
    \vdots & \ddots & \vdots \\
    u_h^{(m_1,1)} & \cdots & u_h^{(m_1,m_2)} \\
    \end{pmatrix}
    \eeq
    where $u_h^{\bar{i}} \in \Gamma$.    
\end{itemize}

The domain $\Omega$ might be 2D in case of an image (photography, X-ray images), 3D in case of a volumetric measurements (CT, MRT) or even 4D in case of time-dependent volumetric measurements. The range $\Gamma$ might be 1D in case of gray-scale images, 3D in case of RGB images, 4D in case of CMYK images or even higher-dimensional in case of hyperspectral images. Furthermore, the range is often restricted to, e.g., $\Gamma\subset[0,1]$ to encode intensities from black (= 0) to white (= 1), positive numbers $\Gamma\subset(\R^+_0)^d$ for intensities which rely on photon counting or absorption or to a quantized range $\Gamma\subset\ls0,1,2,...,255\rs$ which is often used for RGB images. 

\subsection{Differentiation}\label{subsec:ex_diff}

Differentiation is mathematically easy to describe and offers a detailed insight into typical features of inverse problems. While it certainly does not motivate the practical importance of inverse problems in imaging tasks, it is often used as a toy example. We refer to \cite{engl2015regularization,hanke2001inverse} for a detailed study of this problem.

In \emph{differentiation}, the measurement $f\in L^2([0,1])$ is assumed to be the integral of the sought signal $u\in L^2([0,1])$. The operator $\mcA:L^2([0,1])\to L^2([0,1])$ is defined as
\beq
\mcA u := \left[t\mapsto \int_0^t u(x) \dint x\right]
\label{eq:diffop}
\eeq
and its inverse $\mcA^{-1}:L^2([0,1])\to L^2([0,1])$ can be written as
\beq
\mcA^{-1} f = \frac{\dint f}{\dint x}.
\eeq

In the discrete setting, the domain $[0,1]$ is discretized using equidistant grid points $x_i = \frac{i-1}{N-1}$ for $i=1,\ldots,N$. The goal is to recover the values $u_h^i = u(x_i)$ from the measurements
\beq
f_i := \frac{1}{N-1}\sum_{j=1}^{i} u(x_j).
\eeq
This corresponds to the discrete measurement operator $A\in\R^{N,N}$
\beq
A := \frac{1}{N-1}
\begin{pmatrix}
1 & 0 & \hdots & & 0\\
\vdots & \ddots & \ddots  &  & \vdots \\
& &  & & 0\\
1 &  &  \cdots & & 1
\end{pmatrix},
\eeq
where the factor $N-1$ accounts for the length of the discretized intervals. Its inverse is 
\beq
A^{-1} = (N-1) 
\begin{pmatrix}
1  & 0  & \cdots &  &  & 0 \\
-1  & 1 & 0 &  &  & \vdots \\
0  & -1 & 1 & 0  &  &  \\
\vdots  & \ddots & \ddots & \ddots & \ddots &  \\
  &  &  &  &  & 0 \\
0  & \hdots &  & 0 & -1 & 1 \\
\end{pmatrix},
\label{eq:num_diff_Ainv}
\eeq
which corresponds to the backward finite-difference approximation
\beq
\frac{\dint f}{\dint x}(x_{i}) \approx \frac{f(x_{i}) - f(x_{i-1})}{\abs{x_{i} - x_{i-1}}}.
\eeq

\subsection{Deconvolution}\label{subsec:ex_decon}

In \emph{deconvolution}, the measurement is assumed to be a convolved version of the true image; The operator is thus defined as the convolution $\mcA : u \mapsto g*u$ with some kernel $g$. In the easiest setting, the convolution kernel $g$ is assumed to be known. One of the most prominent examples is the one of \emph{deblurring} where~$g$ is the so-called Gaussian kernel, see Fig.~\ref{fig:blur}.

\begin{definition}{Convolution}{def:convolution} Let $u: \R^n\to\R$ and let $g:\R^n\to\R$ denote a kernel. The \emph{continuous convolution} $g\ast u$ is defined as
\beq
    (g\ast u)(x) := \int_{\R^n} g(y) u(x-y) d y \quad \forall x\in\R^n.
\eeq
For a fixed kernel $g$, the convolution is an operator $\mcA: u \mapsto g\ast u$ which maps functions to functions.

Let $\Omega_h:=\{1,\ldots,m_1\}\times\cdots\times\{1,\ldots,m_n\}$, $\Omega'_h := \{-r,\ldots,r\}^n$ for $r\in\N_0$ denote (discretized) domains and assume $u_h:\Omega_h\to\R$ and a kernel $g_h : \Omega'_h \to \R$ (both can be thought of as matrices) are given. The \emph{discrete convolution} $g_h \ast u_h$ is defined as 
\beq
(g_h \ast u_h)(\bar i) := \sum_{\bar j \in \Omega'_h} g_h(\bar j) u_h(\bar i - \bar j),
\eeq
where missing values of $u_h$ and $g_h$ outside of $\Omega_h$ and $\Gamma_h$ need to be specified by boundary conditions.
\end{definition}

\begin{figure}
    \centering
    \begin{subfigure}[b]{0.35\textwidth}
         \centering
         \includegraphics[width=\textwidth]{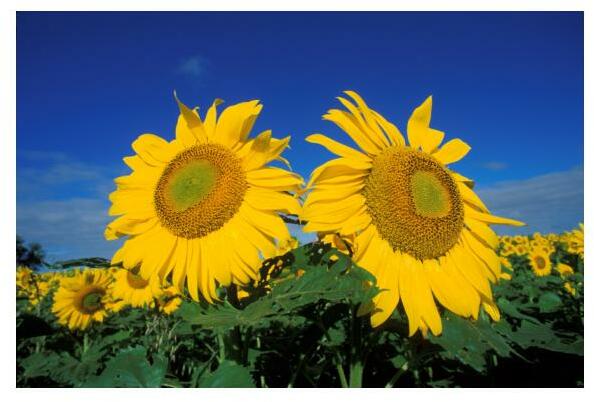}
        \caption{Input image}
    \end{subfigure}
    \hspace{2cm}
    \begin{subfigure}[b]{0.35\textwidth}
         \centering
         \includegraphics[width=\textwidth]{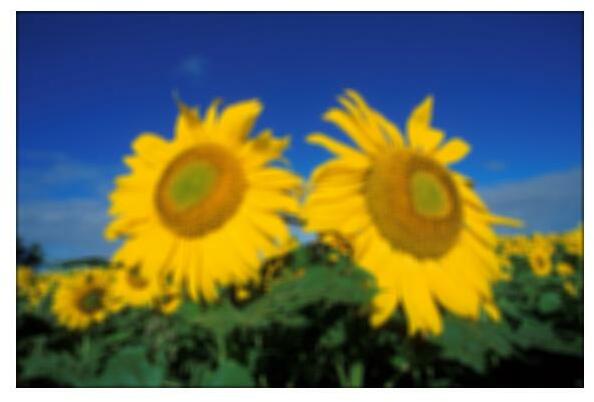}
         \caption{Blurred image}
    \end{subfigure}
    \caption{In deblurring the task is to find a sharp input image which is associated to a given blurred image via a convolution with the Gaussian kernel.}
    \label{fig:blur}
\end{figure}
\noindent%
A prominent example is the one of \emph{deblurring} where $g$ denotes the Gaussian kernel, e.g., in a 2D setting with variance $\sigma^2$ it is defined as
\beq
    g(x) = \frac{1}{2\pi\sigma^2}\exp{-\frac{x^\top x}{2\sigma^2}}.
\eeq

\begin{exercise}{}{}
    Show that the convolution is a linear operator.
\end{exercise}

\subsection{Computed Tomography}\label{subsec:ex_ct}

\emph{Computed tomography} (CT) scans are a popular way to obtain internal images of the body, for a detailed introduction we refer to \cite{buzug2011computed}. Roughly speaking, X-rays are sent through the body from different directions. While traveling through the body, the X-rays are damped depending on the densities of the materials they pass through. This intensity decay is then measured on the opposite side of the body. The collection of all measurements is called \emph{sinogram}. 

The linear operator that is used to describe the scan is called the \emph{Radon transform} $\mcR$ and was studied by Johann Radon way before it was used in practice, \cite{radon20051}.
\begin{center}
\begin{tikzpicture}
    \draw[white] (2, 0) coordinate (A) -- (0, 0) coordinate (B) -- (1.932, 1) coordinate (C)
        pic ["$\varphi$", NatGreen, draw, ->, angle eccentricity=2, thick] {angle};

    \draw[black, ->] (0, 0) -- (2, 0);
    \draw[black, ->] (0, 0) -- (0, 2);
    \node[] at (0.2, 2) {y};
    \node[] at (2.2, 0) {x};
    
    \draw[FAUBlue, thick, ->] (0, 0) -- (1.732, 1);
    \draw[FAUBlue, thick,  ->] (0, 0) -- (-1, 1.732);
    \node[FAUBlue] at (1.932, 1) {s};
    \node[FAUBlue] at (-0.8, 1.732) {t};

    \foreach \a in {-0.3, -0.1, ..., 1.3}{
    \draw[gray, -] (-1+\a*1.732, 1.732+\a)  -- (0.2+\a*1.732, -0.3464+\a);
    }

    \node[opacity=0.1, anchor=south west] at (0,0) {\includegraphics[scale=0.23]{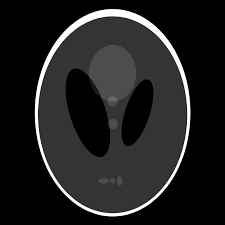}};
\end{tikzpicture} 
\end{center}
The Radon transform is an integral of the original image $u:\Omega\to\Gamma$, $\Gamma\subset\R$ over lines (gray) which are defined over their angle $\phi$ with respect to the $x$-axis and their distance $s$ from the origin; the lines can be parameterized as 
\beq 
    \begin{pmatrix} x(t) \\ y(t) \end{pmatrix} = \begin{pmatrix}
        \cos(\phi) & \sin(\phi) \\ \sin(\phi) & -\cos(\phi)
    \end{pmatrix} \begin{pmatrix} s \\ t \end{pmatrix}
\eeq
and the Radon transform of $u$ is then defined as
\beq
    \mcR u(s,\phi) = \intii u(s\cos(\phi)+t\sin(\phi), s\sin(\phi)-t\cos(\phi)) \dint t.
\eeq
This is further illustrated in Fig.~\ref{fig:ct}. The left column illustrates the rotation of the X-ray source (the angle $\phi$), the middle column the Radon transform for fixed $\phi$ (the dense object absorbs X-rays while the X-rays are allowed to pass around the object resulting in higher values of the Radon transform) and the right column the evolution of the sinogram where the position is $s$ and the angle $\phi$.

\begin{figure}
    \begin{subfigure}{\textwidth}
        \includegraphics[width=\textwidth]{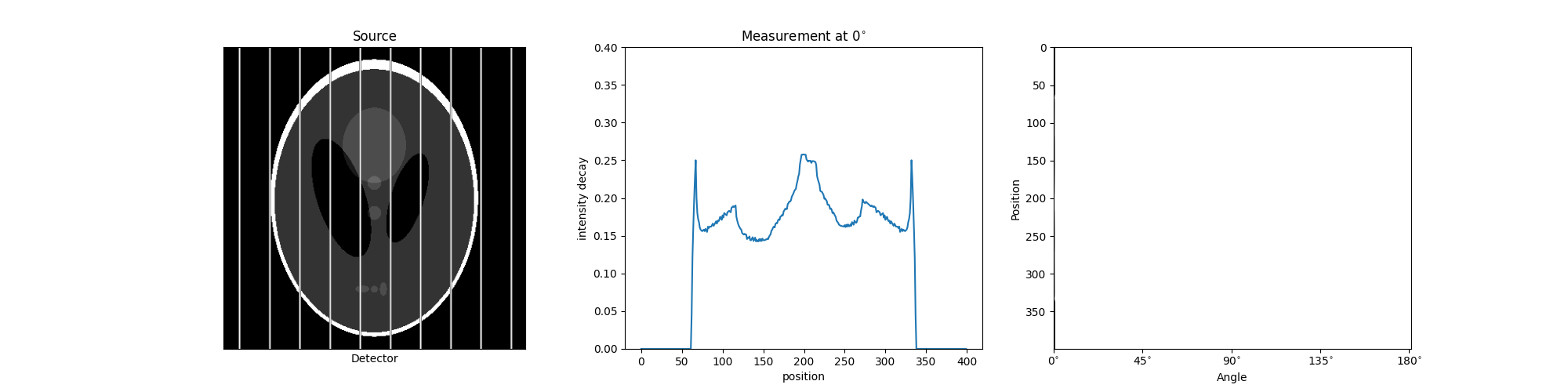}
    \end{subfigure}
    \begin{subfigure}{\textwidth}
        \includegraphics[width=\textwidth]{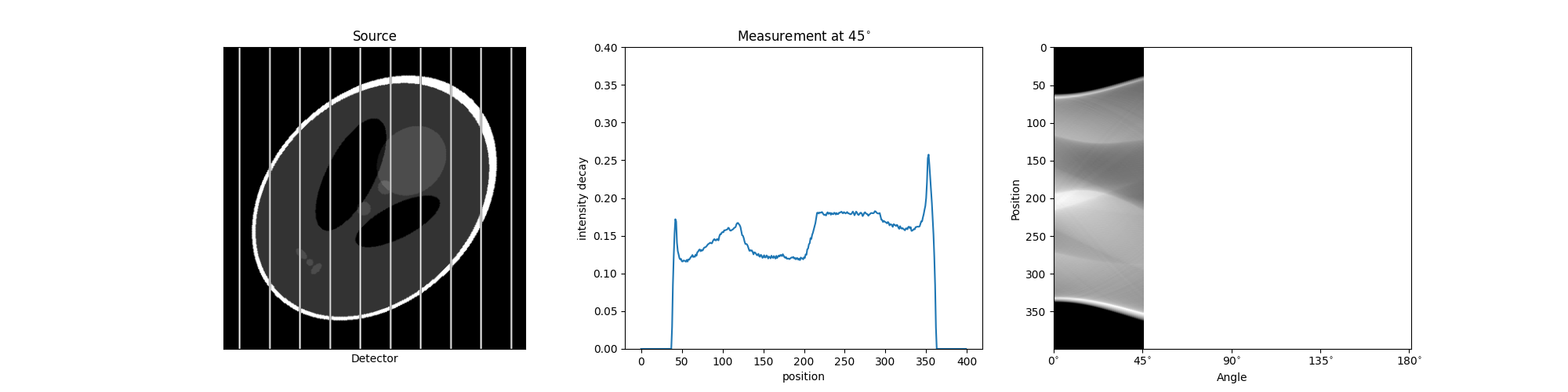}
    \end{subfigure}
    \begin{subfigure}{\textwidth}
        \includegraphics[width=\textwidth]{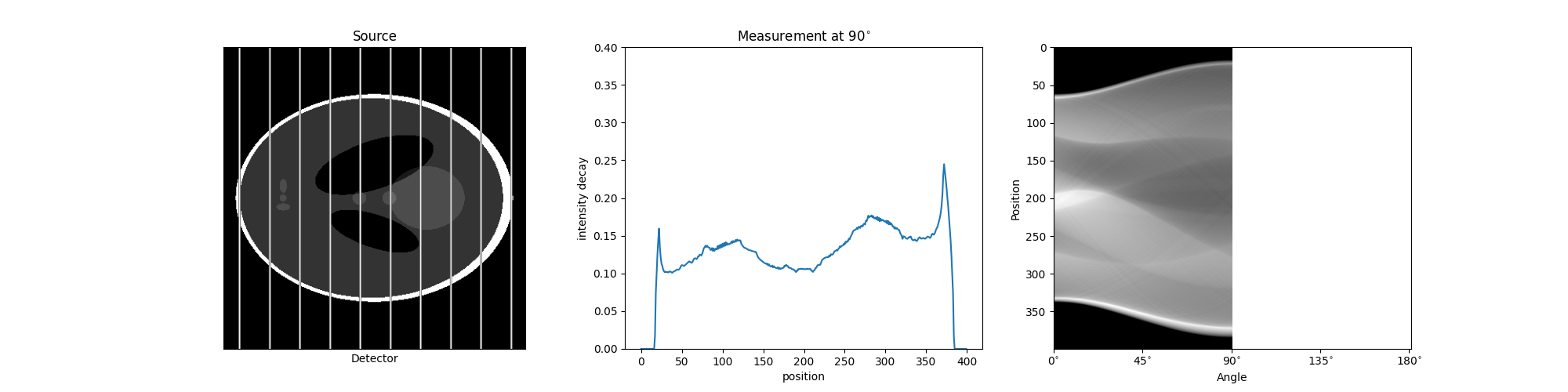}
    \end{subfigure}
    \begin{subfigure}{\textwidth}
        \includegraphics[width=\textwidth]{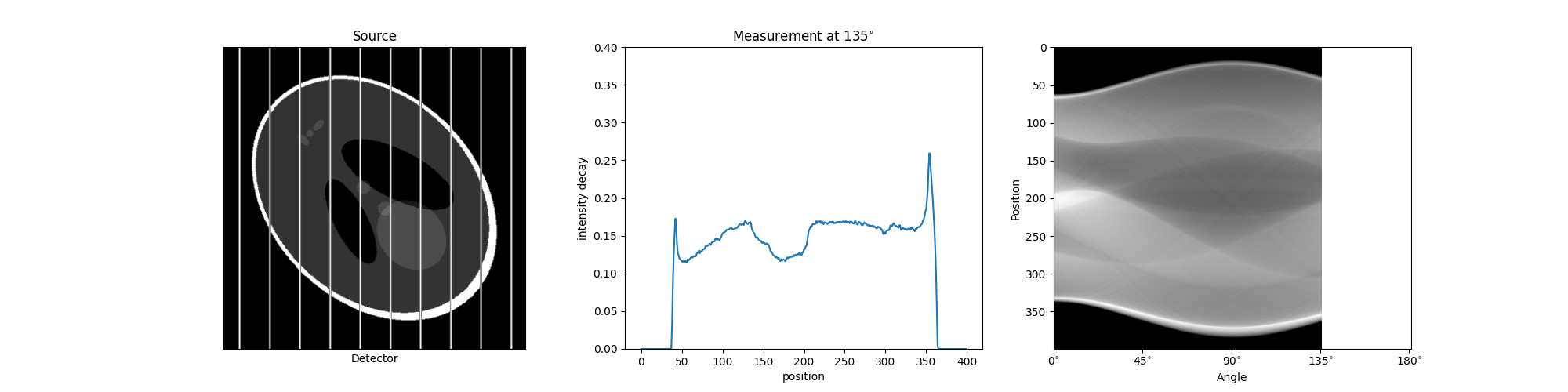}
    \end{subfigure}
    \begin{subfigure}{\textwidth}
        \includegraphics[width=\textwidth]{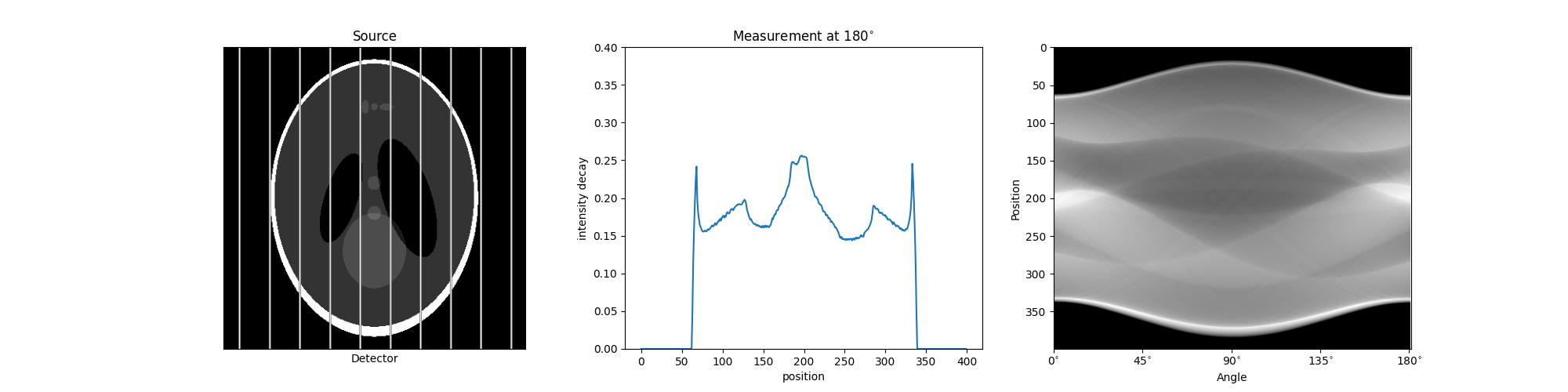}
    \end{subfigure}
    \caption{Computed tomography (CT) example for the Shepp Logan model. The left column shows the rotation of the radiation, the middle column the measurements from the detector for the specific angle and the right column the evolution of the sinogram.}
    \label{fig:ct}
\end{figure}

\subsection{Phase Retrieval}\label{subsec:ex_phase}

\emph{Coherent diffraction imaging} (CDI) is one of the many applications of \emph{phase retrieval}. It is used to image very small objects. We refer to \cite{paganin2006coherent} for a detailed introduction into the topic. For a 2D setting, the experimental setup is depicted in Fig.\ref{fig:cdi}: A sample is placed into a parallel beam with wavelength $\lambda$ and the scattered wave is recorded in the \emph{far field}. The scattered wave in the far field is described by the \emph{Fraunhofer diffraction} regime, i.e., it is given by a scaled Fourier transform of the wavefront distribution behind the sample $u_0(x, y)$, which is also called the \emph{exit wave}. The recorded intensity of the diffraction pattern at distance $z$ from the object is given by magnitude $f(x',y') = |u_z(x',y')|^2$, where according to the Fraunhofer diffraction regime it holds that
\begin{align}
    u_z(x',y')
    &\approx \frac{\exp{ikz}\exp{\frac{ik(x'x' + y'y')}{2z}}}{i\lambda z} \mcF u_0\lb\frac{x'}{\lambda z}, \frac{y'}{\lambda z} \rb.
    \label{eq:cdi_intro}
\end{align}
Given a measurement $f:\Omega'\to\R$ in the detector plane, the goal is to reconstruct the image $u_0:\Omega\to\C$ in the object plane. The operator is thus described by
\beq
    \mcA : u_0 \mapsto \left|\frac{\exp{ikz}\exp{\frac{ik(x'x' + y'y')}{2z}}}{i\lambda z} \mcF u_0\lb\frac{x'}{\lambda z}, \frac{y'}{\lambda z} \rb \right| .
\eeq
It is well-established that the solution to CDI is not unique since a \emph{global phase shift}, \emph{conjugate inversion} and \emph{spatial translation} preserve the Fourier magnitude $|\mcF u_0 |^2$.

\begin{exercise}{}{}
Show that the solution to coherent diffraction imaging is not unique by demonstrating that global phase shift, conjugate inversion and spatial translation preserve the Fourier magnitude.
\end{exercise}


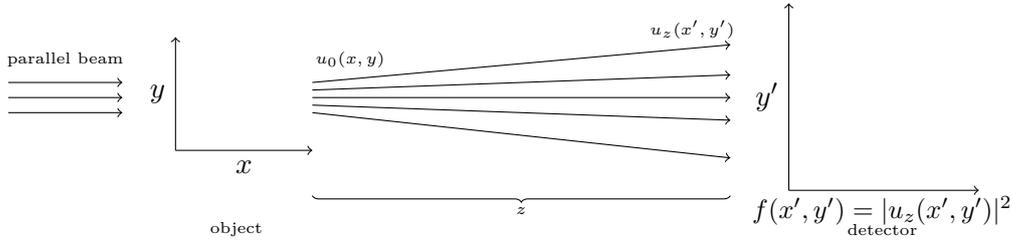
\begin{figure}[t]
    \scalebox{1}{
    \begin{tikzpicture}
        \node (align) at (-1, 0) {}; 
        
        \node (source) at (0, 0) {};    
        \node (object) at (4, 0) {};
        \node (detector) at (11, 0) {};
        
    
        \node at (3.5, 0) {
        };
        \path[->] (3.5-0.8, -0.7) edge node[midway, below] {$x$} (3.5+1, -0.7);
        \path[->] (3.5-0.8, -0.7) edge node[midway, left] {$y$} (3.5-0.8, +0.8);
        
        \node at (3.5, -1.75) {\tiny object};

        \path[->] (0.5,0) edge (2,0); 
        \path[->] (0.5,0.2) edge (2,0.2); 
        \path[->] (0.5,-0.2) edge (2,-0.2); 
    
        \path[->] (4.5,0) edge (10,0); 
        \path[->] (4.5,0.1) edge (10,0.3); 
        \path[->] (4.5,-0.1) edge (10,-0.3); 
        \path[->] (4.5,0.2) edge (10,0.7); 
        \path[->] (4.5,-0.2) edge (10,-0.8); 
    
        \node (pb) at (1.25, 0.5) {\tiny parallel beam};
        \node (uo) at (5, 0.5) {\tiny $u_0(x, y)$};
        \node (uz) at (9.5, 0.9) {\tiny $u_z(x', y')$};
        \draw[decoration={brace, mirror, raise=0.5cm}, decorate] (4.5, -0.8) -- (10, -0.8); 
        \node (z) at (7.25, -1.5) {\tiny $z$};
    
        \node[] at (12, 0) {
        };
        \path[->] (12-1.23, -1.23) edge node[midway, below] {} (12+1.27, -1.23);
        \path[->] (12-1.23, -1.23) edge node[midway, left] {$y'$} (12-1.23, +1.25);
        
        \node at (12, -1.5) {\small$f(x',y') = |u_z(x',y')|^2$};
        \node at (12, -1.75) {\tiny detector};
        
    \end{tikzpicture}
    }
    \caption{\textbf{Coherent Diffraction Imaging.} A sample is placed into a parallel beam and the intensity~$f$ of the scattered wave behind the object is recorded in the \emph{far field}. In the latter, the scattered wave $u_z$ can be approximated using the \emph{Fraunhofer diffraction} regime and thus is given by a scaled Fourier transform of the \emph{exit wave} $u_0(x, y)$ directly behind the sample.
    }
    \label{fig:cdi}
\end{figure}


\section{Theoretical Framework for Well-Posed Inverse Problems}

When solving inverse problems, we have to face certain problems: 
\begin{itemize}
    \item The first problem arises if no solution to the inverse problem \textbf{exists}. This can happen if the measurement is noisy and $\fd$ does not lie within the assumed data range. The problem of non-existence can often be overcome by an appropriate modeling concept.
    \item The second problem arises if the solution to the inverse problem is not \textbf{unique}, i.e., if there exist multiple inputs $u$ which generate the same measurement $\fd$.
    \item The third and most difficult problem arises if solving the inverse problem is not \textbf{stable}, i.e., if the solution's behaviour does not change continuously w.r.t. the measurement $\fd$. If the problem is unstable, even small noise disturbances $\noise$ in the measurement might lead to severe artifacts in the solution.
\end{itemize}

\begin{exercise}{}{}
    Which problems arise in the examples introduced in \cref{sec:examples_for_ip}?
\end{exercise}

The above-mentioned problems have led to the formulation of well-posedness criteria, which date back to Jacques Hadamard\footnote{Jacques Hadamard (1865--1963) was a French mathematician.}, \cite{hadamard1902problemes}.

\begin{definition}{Well- and Ill-Posedness (Hadamard \cite{hadamard1902problemes})}{}
We say an inverse problem with associated forward model $\mcA$ is \alert{well-posed} if the following criteria are fulfilled:
\begin{enumerate}[label=(H\arabic*)]
\item\label{it:H1} The problem has a solution for each data point $f\in\mcV$.
\item\label{it:H2} The solution is unique.
\item\label{it:H3} The solution's behavior changes continuously with the data 
(\emph{stability}).
\end{enumerate}
We say a problem is \alert{ill-posed} if one of the above conditions is not met.
\end{definition}

\begin{example}{Differentiation is Ill-Posed}{ex:numdiff}
It is well-known that differentiation as introduced in \cref{subsec:ex_diff} is an unstable inverse problem. Clearly, $\mcA^{-1}$ is linear but we can show that it is unbounded and thus not continuous: Consider $v^*(x):= \sqrt{2}\sin(2\pi k x)$ for which $\|v^*(x)\|_{L^2}= 1$ holds. We can  estimate 
\beq
    \|\mcA^{-1}\| = \sup_{v\in L^2([0,1]), \|v\|= 1} \|\mcA^{-1}v\| \geq \|\mcA^{-1}v^*\| = 2\pi k
\eeq
and thus $\lim_{k\to\infty} \| \mcA^{-1}\| \to \infty$. Implications: Assume that we have the following noisy measurements 
\beq 
\fd_k = \mcA u + \noise, \quad \noise(x) := \noiselvl \sin(2\pi k x), \quad k\in\N.
\eeq
We observe that both the $L^2$- and the $L^\infty$-\emph{reconstruction error} $\norm{u - \mcA^{-1}\fd_k}_{L^2}, \norm{u - \mcA^{-1}\fd_k}_{L^\infty}$ can become arbitrarily large for large $k$, although the noise level $\noiselvl=\|\noise\|_\infty$ stays fixed. 
\end{example}

\begin{exercise}{}{}
    Calculate the following $L^2$- and $L^\infty$-errors 
\beq
\|f - f^\delta_k \|_2, \quad \|f - f^\delta_k\|_\infty, \quad \|u - \mcA^{-1} f^\delta_k \|_2, \quad \|u - \mcA^{-1} f^\delta_k\|_{\infty}
\eeq
for the differentiation operator \cref{eq:diffop} and interpret the results.
\end{exercise}

Note that instability is a subtle but severe problem. Once a problem has been discretized, any linear (finite) operator is continuous. Still, we can encounter cases were minor measurement errors result in relative big reconstruction errors. 

\begin{memo}{}{mem:bounded}
Linear maps $A:U \to V$ on finite-dimensional spaces $U, V$ are continuous.
\end{memo}

\begin{figure}
    \centering
    \includegraphics[width=0.75\linewidth]{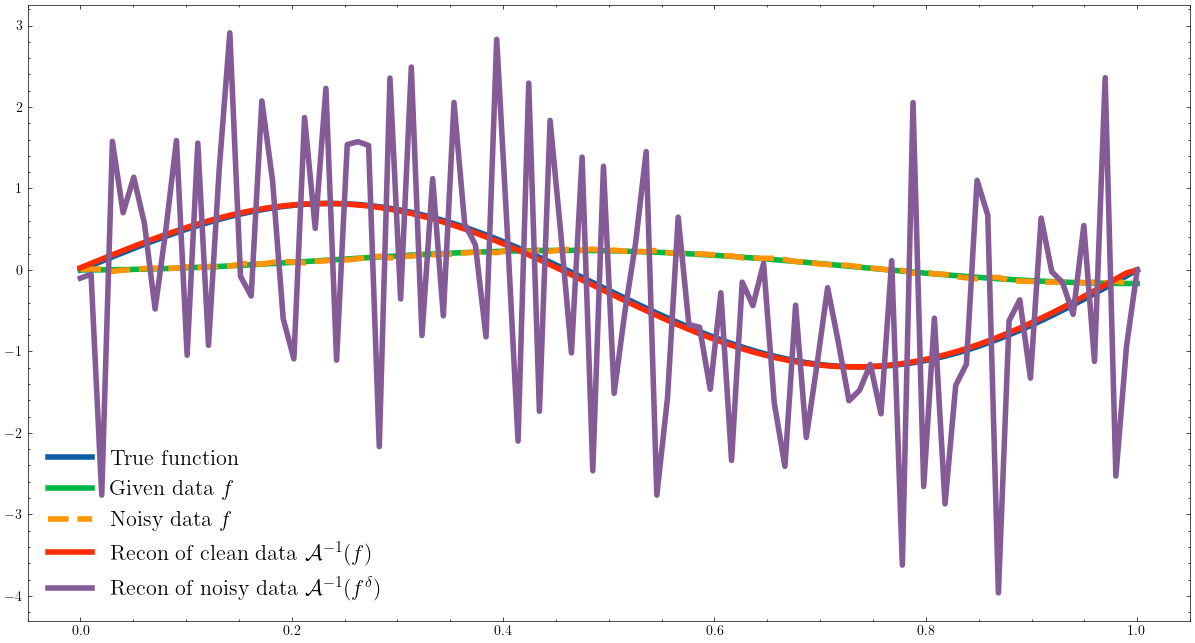}
    \caption{Numerical differentiation is ill-posed \cref{ex:numdiff} / ill-conditioned \cref{ex:numdiff_cont}. Here, we disturbed the true function $u(t) := \sin(2\pi t) + (t-0.5)^2 - 0.25$ by adding Gaussian noise of variance $0.01$. Comparing the clean and noisy data, we can barely see any differences. The reconstructions using either the clean or the noisy data, however, show large differences!}
    \label{fig:numdiff}
\end{figure}

\begin{example}{Differentiation is Ill-Posed -- Continued}{ex:numdiff_cont}
Recalling the definition of the discrete inverse operator in \cref{eq:num_diff_Ainv} as well as the definition of the spectral norm we have for $A\in\R^{N,N}$

\beq
\norm{A^{-1}}_2 =  \sqrt{\lambda_{\text{max}} \left( A^{-T}A^{-1} \right)} > N-1.
\eeq
In order to see the above estimate we define 
\beq
M:= \frac{1}{(N-1)^2} A^{-T} A^{-1} = 
\begin{pmatrix}
2 & -1 &  0 & & \hdots &   & 0 \\
-1 & \ddots & \ddots & \ddots & & & \\
0 & \ddots&  & &  \ddots &  & \vdots  \\
 & \ddots &&  &  &  \ddots & \\
\vdots & & \ddots &  & \ddots & \ddots & 0 \\
& &&  \ddots &  \ddots &  2 & -1 \\
0   &  & \hdots & & 0 & -1 & 1 
\end{pmatrix}.
\eeq
We can directly verify that $\lim_{k\to\infty} M^k \neq 0$, which yields that there exists an eigenvalue bigger than $1$ (see, e.g., \cite{oldenburger1940infinite}) and thus $1 < \lambda_{\max}(M)$. This implies $(N-1)^2 < \lambda_{\max}((N-1)^2 M) = \lambda_{\max}(A^{-T}A^{-1})$. While being finite, the norm can become arbitrarily large for large $N$ as $\norm{A^{-1}}_2 > N-1 \xrightarrow[]{N\to\infty} \infty$. 
Even though the discrete problem is stable by definition, one can use this and similar considerations for $\|A\|_2$ to show that $\text{cond}(A) \gg N$ and thus that the problem is ill-conditioned for $N>1$. The implications are demonstrated numerically in \cref{fig:numdiff}. 
\end{example}

\section{Inversion Approaches}

When solving inverse problems, the goal is to retrieve an unknown quantity $u$ from given data $\fd$. As previous examples have shown, this can be a challenging task. Even if the operator $\mcA$ has a well-defined inverse on its range, i.e., $\mcA^{-1}:\mcR(\mcA)\to\mcU$ exists, there is no guarantee that the perturbed data still lies in the range of our operator. One may also only  obtain partial measurements, which leads to underdetermined systems and renders the direct inversion impossible also on the clean data. Furthermore, many interesting problems also violate the third Hadamard criterion \labelcref{it:H3}.

\paragraph{Regularization.}
One approach to obtain meaningful solutions in the described scenarios is called \alert{regularization}. A regularization $\mfRa:\mcV\to\mcU$ maps an arbitrary data point to a favorable solution. Intuitively, one hopes that the regularization approximately extends the notion of the inverse to the noisy and possibly ill-posed setting, namely
\beq
\mfRa(\mcA u + \noise) \approx u.
\eeq
A typical strategy is so-called \alert{Tikhonov-type} or \alert{variational regularization}, where the output of the regularization map is defined as the solution of the following problem
\begin{align}\label{eq:tik}
\argmin_{u\in\mcU} \underbrace{\norm{\mcA u -\fd}^2_{L^2}}_{\text{data fidelity}} + \regp \underbrace{\mcJ(u)}_{\text{regularizer}}.
\end{align}
The terms have the following interpretation:
\begin{itemize}
\item \alert{Data fidelity}: The term tells us how well our guess $u$ fits to the observed data $\fd$. While we want to minimize the data fidelity, it is not always meaningful to have it equal to zero in the noisy case, since we look for $u^*$ such that $\mcA u^* = \fd - \noise $, where potentially $\mcA u^* \neq \fd$.
\item \alert{Regularizer}: The regularizer allows us to incorporate additional information about our sought solution. In the classical situation, we already know that $u$ should be close to some point $\mu$, i.e., we want to penalize the distance between the solution $u$ and $\mu$. Furthermore, one often wants to penalize certain directions more than others, for which we consider a unitary operator $\mcQ:\mcU\to\mcU$ and then choose
\beq
\mcJ(u) = \norm{u - \mu}^2_{L^2, \mcQ} := \ll u- \mu, \mcQ (u -\mu)\rr,
\eeq
which is referred to as \alert{Tikhonov regularization}\footnote{Andrey Nikolayevich Tikhonov (1906--1993) was a Soviet mathematician.}.
\item \alert{Regularization parameter $\regp$}: The parameter $\regp>0$ controls the strength of the regularizer. In some formulations, this parameter is included into the definition of the regularizer.
\end{itemize}

We explore the theory and practice of regularizers in \cref{ch:reg}. Let us already highlight that classical regularization follows the paradigm 
\begin{center}
\textit{Given a data point $\fd\in\mcV$, find 
a favorable quantity $u\in\mcU$.}
\end{center}
\paragraph{Bayesian inversion.} Another inversion approach is the probabilistic or \alert{Bayesian} one. We refer to \cite{stuart2010inverse} for a detailed introduction into this topic. Here, $\fd, u$ and $\noise$ are not considered as points in a vector space, but rather as random variables that follow a certain probability distribution. This approach follows the paradigm 
\begin{center}
\textit{Given a data point $\fd\in\mcV$, find the probability that it was caused by the quantity $u\in\mcU$.}
\end{center}
This question is usually expressed in the so-called \alert{posterior distribution}, which is the conditional probability $p(u|\fd)$. Employing Bayes' rule\footnote{Thomas Bayes (1701--1761) was an English statistician.}, the posterior can be expressed as
\beq
p(u| \fd) = \frac{p(\fd| u)\, p(u)}{p(\fd)}.
\label{eq:posterior}
\eeq
The quantities involved in the above equation have the following interpretation:
\begin{itemize}
\item \alert{The likelihood $p(\fd| u)$}: The likelihood describes how well the given data fits to the quantity $u$. In our case we know that $\fd=\mcA u + \noise$ and thus
\beq
p(\fd| u) = p(\fd = \mcA u + \noise | u)
= p_\noise(\fd - \mcA u).
\eeq
Here, $p_\noise$ denotes the distribution of our noise. This means, that in order to compute the likelihood, we need to choose a noise model. In finite dimensions, a typical example is a zero-mean Gaussian noise model with diagonal covariance matrix $\Sigma= \sigma^2 \Id$, i.e., $\noise\sim\mcN(0, \sigma^2 \Id)$ and 
\beq
p_\noise(v) = \frac{1}{(2\pi\sigma)^{n/2}}\exp{-\frac{\norm{v}_2^2}{2\sigma^2}}.
\eeq
In this case, the likelihood is given as 
\begin{align}\label{eq:likelihood}
p(\fd|u) = 
\frac{1}{(2\pi\sigma)^{n/2}}\exp{-\frac{\norm{\fd-A u}_2^2}{2\sigma^2}}.
\end{align}
The likelihood directly relates to the data fidelity term in the classical regularization case.
\item \alert{The prior $p(u)$}: The prior models the belief in our desired quantity that we have before any measurement is taken. Intuitively, $p(u)$ tells us how likely we think the quantity $u$ is among all possible quantities in $\mcU$.
\item \alert{The normalizing constant $p(\fd)$}: This quantity can be expressed as 
\beq
p(\fd) = \int p(\fd| u) p(u) \dint u
\eeq
and is often hard to compute and thus not directly available.
\end{itemize}

\paragraph{MAP Estimation.}
From the Bayesian view-point, the posterior distribution is the solution to the inverse problem. At first, this seems completely different to the previously mentioned regularization approach. 
However, the classical/deterministic inversion approach can directly be related to the probabilistic one via the \alert{maximum-a-posteriori (MAP) estimator}.

\begin{definition}{MAP}{}
The \emph{maximum-a-posteriori estimate} is defined as
\beq
u^{\text{MAP}} \in \argmax_{u\in \mcU} \ls p(u| \fd) \rs.
\eeq
\end{definition}\noindent%
The MAP approach follows the paradigm
\begin{center}
\textit{Given a data point $\fd\in\mcV$, find $u\in\mcU$ which most probably lead to the measurement.}
\end{center}
Compared to the Bayesian approach, the MAP approach offers the computational advantage that the normalization constant does not need to be considered, since
\beq
\argmax_{u\in\mcU} \ls p(u|\fd) \rs = \argmax_{u\in\mcU} \ls \frac{p(\fd|u)\, p(u)}{p(\fd)} \rs = \argmax_{u\in\mcU} \ls p(\fd|u)\, p(u)\rs.
\eeq
Usually, one can further simplify the computations by considering the (negative) logarithm of the involved terms. Using the monotonicity of the logarithm, one obtains
\beq
u^{\text{MAP}} \in 
\argmin_{u\in\mcU} \ls - \log(p(\fd|u)) - \log(p(u))\rs ,
\eeq
where $- \log(p(\fd|u))$ is referred to as the \alert{negative log-likelihood}. Consider again the finite dimensional setting with a zero-mean Gaussian noise model. With \cref{eq:likelihood} we directly concur %
\begin{align}\label{eq:MAPGauss}
u^{\text{MAP}} \in 
\argmin_{u\in U} \ls \frac{1}{2\sigma^2}\norm{\fd - A u}^2_2 - \log(p(u))\rs.
\end{align}
Looking at the Tikhonov-type regularization in \cref{eq:tik}, this highlights the connection between the regularizer in the classic and the prior in the Bayesian setting. We will further explore the usage of Bayesian inversion and related estimates in the context of machine learning approaches in \cref{sec:mlforip}.
\begin{memo}{}{}
Assuming a zero-mean Gaussian noise model for the measurement error, the MAP corresponds to Tikhonov-type regularization with an appropriately chosen prior.
\end{memo}
Let us now consider two different priors, namely the uniform and Gaussian distribution. 

\begin{example}{Maximum Likelihood Estimator}{ex:mle}
Often, the probability distribution of $u$ is unknown and the easiest approach appears to assume a uniform distribution. In the finite dimensional case $U=\R^d$, this means that we have a compact set of finite measure $U' \subset U$ with $p(u)=\frac{1}{\abs{U'}}\mathbf{1}_{U'}(u)$. The MAP then reduces to
\bfl
    \argmax_{u\in U} \ls p(u | \fd) \rs  = \argmax_{u\in U} \ls p(\fd | u) p(u) \rs =
    \argmax_{u\in U'} \ls p_\noise(\fd - Au) \rs = \argmin_{u\in U'} \ls \| \fd- A u \|^2_2 \rs.
\efl
Seeing that only the likelihood term remains, this is referred to as the \alert{maximum likelihood estimator (MLE)}.
 This least-squares problem will further be explored in \cref{ch:reg}. For now, we would like to point out that the MLE can be obtained as $u^\dagger = A^\dagger \fd$, where $A^\dagger := (A^\top A)^{-1} A^\top$ is the pseudo inverse of $A$.
\end{example}

\begin{example}{Recovering the Tikhonov Regularization}{ex:gauss_prior_tik_reg}
We can also recover the classic Tikhonov regularization by choosing a Gaussian prior for $u$, i.e.,
\beq
u\sim \mcN(\mu, \sigma_u^2 \Id).
\eeq
Using the computation in \cref{eq:MAPGauss} we have 
\begin{align}
\argmax_{u\in U} \ls p(\fd|u)\, p(u) \rs &= 
\argmin_{u\in U} \ls \frac{1}{2\sigma^2}\norm{\fd - A u}^2_2 - \log(p(u)) \rs \\ 
&= 
\argmin_{u\in U} \ls \frac{1}{2\sigma^2}\norm{\fd - A u}^2_2 + \frac{1}{2\sigma_u^2} \norm{u - \mu}^2_2 \rs\\
&= 
\argmin_{u\in U} \ls \norm{\fd - A u}^2_2 + \frac{\sigma^2}{\sigma_u^2} \norm{u - \mu}^2_2\rs.
\end{align}
The MAP estimate in this setting can be obtained as $u^\dagger = \lb A^\top A + \frac{\sigma^2}{\sigma_u^2} I \rb^{-1}\lb A^\top \fd + \frac{\sigma^2}{\sigma_u^2} \mu\rb$.
\end{example}

While finite-dimensional linear problems never directly violate the stability condition, they can certainly approximate instable behavior in terms of numerical instabilities, as we have seen in \cref{ex:numdiff_cont}. We can now ask ourselves if and how the choice of the prior --- or in other words, the use of regularization --- can affect these numerical instabilities. 
Let us thus consider the condition number of the term $A^\top A$ appearing in the MLE of \cref{ex:mle} (uniform prior, no regularization) as well as the condition number of the term $A^\top A + \frac{\sigma^2}{\sigma_u^2}I$ appearing in the MAP of  \cref{ex:gauss_prior_tik_reg} (Gaussian prior, Tikhonov regularization). Using the singular value decomposition $A = UDV^\top$, we have
\beq
    \cond (A^\top A) = \cond ((UDV^\top)^\top(UDV^\top)) = 
    \cond (V D^\top D V^\top) = \frac{ \sigma_{\max}^2 }{ \sigma_{\min}^2 },
    \label{eq:cond_mle}
\eeq
where $\sigma_{\max}$ and $\sigma_{\min}$ denote the biggest and smallest singular values of $A$. Thus, the problem becomes numerically ill-conditioned if the smallest singular value goes to zero.
On the other hand,
\beq
    \cond \lb A^\top A+ \frac{\sigma^2}{\sigma_u^2} I \rb = \cond \lb V \lb D^\top D+ \frac{\sigma^2}{\sigma_u^2} I \rb V^\top \rb = \frac{ \sigma_{\max}^2 +\frac{\sigma^2}{\sigma_u^2} }{ \sigma_{\min}^2+\frac{\sigma^2}{\sigma_u^2} }.
    \label{eq:cond_map}
\eeq
For small $\sigma_{\min}$, the condition number of the regularized problem is clearly better, especially if $\frac{\sigma^2}{\sigma_u^2}$ is not too small! This means that small changes in the data do not affect the reconstruction as much.

\begin{exercise}{Ill-Conditioned Matrix Equations}{}
Consider the following linear matrix equation
\beq
	Au = f,
\eeq
where $u, f\in\R^n$ and $A\in\R^{n,n}$. Use that the matrix has a representation
\beq
	A = \sum_{i=1}^n \lambda_i x_i x_i^\top
\eeq
where $0\leq\lambda_1\leq...\leq\lambda_n$ are eigenvalues  and $x_i\in\R^n$ with $\|x_i\|=1$ are eigenvectors to show the following:
\begin{itemize}
    \item Noise which is orthogonal to all eigenvectors $x_i$ for which $\lambda_i \neq 0$ does not lead to reconstruction artifacts.
    \item High-frequency noise leads to bigger reconstruction artifacts than low-frequency noise of the same magnitude.
\end{itemize}
\end{exercise}

\begin{exercise}{Parameter Choice Rule}{}
    For some matrix $A\in\R^{n,n}$ and unknown quantity $u\in\R^n$ we are given the noisy measurement 
    \beq
    	f^\noiselvl = Au + \noise,
    	\label{eq:ip}
    \eeq
    where we assume that the noise is bounded, i.e., $\|\noise\| \leq \noiselvl$. In order to improve the condition of the matrix $A$, we shift its eigenvalues away from zero by defining
    \beq
    	A_\alpha := A + \alpha I.
    \eeq
    For $u_\alpha^\delta := A_\alpha^{-1}f^\delta$:
    \begin{itemize}
    \item[a)] Calculate the \alert{reconstruction error} between the exact solution $u^0_0$ and the solution $u_0^\delta$ for noisy data which is introduced by the noise in the measurement, i.e.,
    \beq
    	E_0^\delta = \| u_0^0 - u_0^\delta \|,
    \eeq
    \item[b)] Calculate the \alert{approximation error} between the exact solution $u_0^0$ and the regularized solution $u_\alpha^0$ which is introduced by the approximation (regularization) of the matrix, i.e.,
    \beq
    	E_\alpha^0 = \| u_0^0 - u_\alpha^0 \|.
    \eeq
    \item[c)] Calculate the error between the exact solution $u_0^0$ and the regularized solution $u_\alpha^\delta$ for noisy data, i.e.,
    \beq
    	E_\alpha^\delta = \| u_0^0 - u_\alpha^\delta \|.
    \eeq
    \end{itemize}
    Discuss the role of the approximation (regularization) of the matrix in the presence of noise. How should $\alpha$ be chosen?
    
    \paragraph{Tip:} Use the triangle inequality of norms and the representation 
    \beq
    	A = \sum_{i=1}^n \lambda_i x_i x_i^\top
    \eeq
    where $0\leq\lambda_1\leq...\leq\lambda_n$ are eigenvalues  and $x_i\in\R^n$ with $\|x_i\|=1$ are eigenvectors.
\end{exercise}

\newpage
\chapter{Generalized Solutions and Regularization}\label{ch:reg}

In this chapter we introduce concepts from the theory of regularization for inverse problems. The exposition is loosely based on the lecture notes \cite{burger2007inverse} as well as \cite{korolev2020inverse,hanke2001inverse,engl2015regularization,Engl2000,benning2018modern}. In the following, we consider \alert{bounded} linear operators $\mcA:\mcU\to\mcV$ for \alert{Hilbert spaces} $\mcU$ and $\mcV$ and associated \alert{general linear operator equations} (GLOE) of the form 
\beq
    \mcA u = f.
    \label{eq:gloe}\tag{GLOE}
\eeq
We denote the respective scalar products associated to $\mcU$ and $\mcV$ by $\ll\cdot,\cdot\rr_\mcU$ and $\ll\cdot,\cdot\rr_\mcV$. The (unique) \alert{adjoint operator} of $\mcA$, denoted with $\mcA^*:\mcV\to\mcU$, is defined as follows:
\beq
    \ll \mcA u, v \rr_\mcV = \ll  u, \mcA^* v \rr_\mcU, \quad \forall u\in\mcU, v\in\mcV.
\eeq
We denote by 
\begin{itemize}
    \item $\domain(\mcA)$ the \alert{domain} of $\mcA$, i.e. $\domain(\mcA) = \mcU$ if not noted differently, but $\domain(\mcA) =  U \subset \mcU$ also possible,
    \item $\kernel(\mcA)$ the \alert{null space} of $\mcA$, i.e., $\kernel(\mcA) := \ls u\in\domain(\mcA): \mcA u = 0\rs$,
    \item $\range(\mcA)$ the \alert{range} of $\mcA$, i.e., $ \range(\mcA) := \ls v\in \mcV: \exists u\in\domain(\mcA)\text{ such that } \mcA u = v \rs$.
\end{itemize}
Note, that $\range(\mcA)$ is not closed in general (it is never closed for a compact
operator, unless the range is finite-dimensional). We, furthermore, denote by 
\begin{itemize}
    \item $\overline{Z}$ the \alert{closure} of some space $Z$,
    \item $(Z')^\perp$ the \alert{orthogonal complement} of some subspace $Z'\subset Z$ in $Z$, i.e., \beq (Z')^\perp = \{w\in Z: \langle w, z\rangle = 0\quad \forall z \in Z'\}.\eeq 
\end{itemize} 
One can show that $(Z')^\perp$ is a closed subspace and that $Z^\perp = \ls 0 \rs$. If $Z'$ is closed, it furthermore holds that $Z' = \lb(Z')^\perp\rb^\perp$ and $Z=Z' \oplus (Z')^\perp$. The latter means that for any $z\in Z$, there exists a unique pair $(z', z^\perp) \in Z' \times (Z')^\perp$ such that $z=z'+z^\perp$.
If $Z'$ is not closed, then the weaker version $\overline{Z'} = \lb(Z')^\perp\rb^\perp$ holds. 
With respect to our bounded linear operators this translates to $\kernel(\mcA) = \range(\mcA^*)^\perp$ and thus $\kernel(\mcA)^\perp = \overline{\range(\mcA^*)}$, as well as $\kernel(\mcA^*) = \range(\mcA)^\perp$ and thus $\kernel(\mcA^*)^\perp = \overline{\range(\mcA)}$. Consequently, we have the following decomposition of the Hilbert spaces:
\beq
\mcU = \kernel(A) \oplus \overline{\range(A^*)},\qquad 
\mcV = \kernel(A^*)\oplus \overline{\range(A)}.
\eeq
Lastly, we denote by $\mcP: \mcU \to \kernel(\mcA)$ and $\mcQ: \mcV \to \overline{\range(\mcA)}$ the orthogonal projections onto $\kernel(\mcA)$ and $\overline{\range(\mcA)}$.

\begin{exercise}{}{}
Show that $\kernel(\mcA^*) = \range(\mcA)^\perp$.
\end{exercise}

\comment{
\section{Ill-Posed Matrix Equations}

We start this section with some introductory discussions in the finite dimensional case with $U=V = \R^d$. For simplicity, let us first assume that $A$ is symmetric and positive definite and that, therefore, there exists an orthonormal basis $(v_1,\ldots, v_d)$ of eigenvectors with associated eigenvalues $0<\lambda_1 \leq\ldots\leq\lambda_d$. For clarity, we also assume $\lambda_n = 1$, since typically the largest eigenvalue is not the problematic one, such that $\cond(A) = \lambda_1^{-1}$. Given measurement data $f=A u$ and the noisy version $\fd$, we see that the reconstruction error can be estimated as
\begin{align}\label{eq:orthprop}
\norm{u^\noiselvl - u}_2 \leq 
\lambda_1^{-1} \noiselvl,
\end{align}
and, therefore, we see that the smallest eigenvalue can amplify the reconstruction error. We also observe that the estimate is sharp, since for $\noise = \delta v_1$ we obtain
\begin{align}
\norm{u^\noiselvl - u}_2 = \delta \norm{A^{-1} v_1} = \delta \lambda_1^{-1}.
\end{align}
We briefly discuss the situation, where the operator $A$ has a non-trivial nullspace, i.e., $\kernel(A)\neq \{0\}$. We still assume $f\in \range(A)$, however, the noise might have nullspace components, such that $\fd\notin \range(A)$. Here, and in the following, we utilize the Banach's \alert{closed range theorem}, which yields
\begin{align}
\mcU = \kernel(A) \oplus \overline{\range(A^*)},\qquad 
\mcV = \kernel(A^*)\oplus \overline{\range(A)}.
\end{align}
In the finite dimensional case, the range of a linear operator is always closed and, therefore, the projection $P:\mcV\to\range(A)$ is continuous. In the spectral representation, we can assume that $\lambda_i > 0$ for all $i>m$. Following the computation in \cref{eq:orthprop} we obtain
\begin{align}\label{eq:orthpropnull}
\norm{\ud - u} \leq \delta \lambda_m^{-1},
\end{align}
which means that nullspace components do not amplify the reconstruction error.

\begin{exercise}{}{}
Prove the estimates in \cref{eq:orthprop,eq:orthpropnull}.
\end{exercise}
}

\section{Generalized Inverses}

Let us first observe for which operators $\mcA$ the inverse problem  is ill-posed. 
\begin{itemize}
\item If $\range(\mcA) \neq \mcV$, i.e., the range of $\mcA$ is not the full image space, then \eqref{eq:gloe} is not solvable for $f\in\mcV\backslash(\range(\mcA))$. Thus \labelcref{it:H1} is violated! 
    
    In this case it seems reasonable to look for $u$ such that $\mcA u$ has minimal distance to $f$.
    \item On the other hand, if $\mcA$ has a nontrivial null space, \eqref{eq:gloe} may have multiple solutions. Thus (H2) is violated!

    A popular remedy to obtain uniqueness is to select the solution with minimal norm.
\end{itemize} 

The above considerations lead to the following definition:
\begin{definition}{Least-Squares and Minimal-Norm Solution}{}
An element $u^*\in\mcU$ is called
\begin{enumerate}
\item \strut \alert{least-squares solution} (LSS) of \eqref{eq:gloe} if
\beq
\| \mcA u^* - f\|_\mcV = \inf_{u\in\mcU} \ls \| \mcA u - f\|_\mcV \rs,
\eeq
\item \alert{minimal-norm solution} (MNS) of \eqref{eq:gloe} if
\beq
\|u^*\| = \inf \ls \|u\|_\mcU : u\in\mcU \text{ is a least-squares solution of } \eqref{eq:gloe} \rs.
\eeq
\end{enumerate}
\end{definition} 

\begin{remark}{Similarity to MLE}{}
Note that the least-squares solution corresponds to the maximum-likelihood estimator in \cref{ex:mle}.
\end{remark}
\begin{remark}{Existence of least-squares solutions}{}
In general, there might not exist a least squares solution. In fact, the existence for arbitrary data $f$ is equivalent to the closedness of the range. If we assume that $\range(\mcA)$ is closed, we know that for every $f\in\mcV$ the space $Z= \range(\mcA) - f$ is closed and convex. In Hilbert spaces for closed and convex subsets, there always exist elements of minimal norm, i.e., there exists $z^*\in Z$ such that %
\begin{align}
\norm{z^*} = \inf_{z\in Z} \norm{z}_\mcV = \inf_{u\in\mcU}\norm{\mcA u - f}_\mcV.
\end{align}
On the other hand, if the range is not closed, there exists $f\in\overline{\range(\mcA)}\setminus\range(\mcA)$ and a sequence $\{u_n\}_{n\in\N}\subset \mcU$ with $\lim_{n\to\infty}\norm{\mcA u_n - f}_\mcV = 0$. Therefore, $\inf_u \norm{\mcA u -f}_\mcV = 0$, but since $\norm{\mcA u -f}_\mcV>0$ for all $u\in\mcU$ this yields
\begin{align}
\argmin_{u\in\mcU} \norm{\mcA u - f}_\mcV = \emptyset.
\end{align}
Again, we note that in the finite dimensional case, the range is always closed and, therefore, a least-squares solution always exists.
\end{remark}
\comment{%
\begin{remark}{Uniqueness of minimum-norm solution}{}
If we assume the existence of least-squares solutions, then the minimal-norm solution is unique. This is due to the fact that $Z = \argmin_{u\in\mcU} \norm{f - \mcA u}$ is a convex and closed set. This can be shown as follows: Denote by $c=\min_{u} \norm{f - \mcA u}$, then
\begin{align}
u,v \in Z, t\in [0,1] &\Rightarrow \norm{f - \mcA (t u + (1-t) v)} = \norm{f - t \mcA u - (1-t) \mcA v}\\ &= \norm{t(f-\mcA u) + (1-t) (f-\mcA v)}\leq 
t \norm{f - \mcA u} + (1-t) \norm{f - \mcA v}\\& = 
c \Rightarrow t u + (1-t) v \in Z.
\end{align}
Assuming that $u_n\in Z$ with $u_n\to u\in\mcU$ then by the continuity of the norm and $\mcA$ (we assumed boundedness) we have that
\begin{align}
\norm{\mcA u - f} = \lim_{n\to \infty} \norm{\mcA u_n - f} = c.
\end{align}

and 
\begin{align}
u \mapsto \norm{\mcA u - f}^2
\end{align}
is a strictly convex function. Again, a [....]
\end{remark}}

It is easy to see that if a least-squares solution exists, then the minimal-norm solution is unique, because it is the minimizer of a strictly convex (quadratic) functional on a linear subspace.
For those $f$, where it exists, the minimal-norm solution can (at least in theory) be computed via the Moore--Penrose generalized inverse, which is defined as follows:
\begin{definition}{Moore--Penrose Generalized Inverse}{}
    For $\mcA\in\mcL(\mcU, \mcV)$ let $\tilde{\mcA} : \kernel(\mcA)^\perp \to \range(\mcA)$ denote its \emph{restriction}. Then the \alert{Moore-Penrose generalized inverse} $\mcAdag$ is defined as the unique \emph{linear extension} of $\tilde{\mcA}^{-1}$ to
    \beq
        \domain(\mcAdag) := \range(\mcA) + \range(\mcA)^\perp
    \eeq
    with $\kernel(\mcAdag) = \range(\mcA)^\perp$.
\end{definition}

Note that the Moore--Penrose inverse $\mcAdag$
is well-defined: First of all, due to the restriction
to $\kernel(\mcA)^\perp$ and $\range(\mcA)$, the operator $\tilde{\mcA}$ is injective and surjective, and hence, $\tilde{\mcA}^{-1}$ exists. As a consequence, $\mcAdag$ is well-defined on $\range(\mcA)$. For arbitrary $f\in \domain(\mcAdag)$ we can find unique $f_1 \in \range(\mcA)$ and $f_2\in\range(\mcA)^\perp$ such that $f=f_1 + f_2$. From the linearity of $\mcA$ and from $\kernel(\mcAdag
) = \range(\mcA)^\perp$ we finally obtain
\beq
    \mcAdag f = \mcAdag f_1 + \mcAdag f_2 = \tilde{\mcA}^{-1} f_1.
\eeq

It can be shown that $\mcAdag$
is characterized by the \alert{Moore--Penrose equations}
\begin{align}
    \mcA\mcAdag\mcA &= \mcA \\
    \mcAdag\mcA\mcAdag &= \mcAdag \\    
    \mcAdag\mcA &= I - \mcP\\
    \mcA\mcAdag &= \mcQ|_{\domain(\mcAdag)}  
\end{align}
where $\mcP: \mcU \to \kernel(\mcA)$ and $\mcQ: \mcV \to \overline{\range(\mcA)}$ are the orthogonal projectors onto $\kernel(\mcA)$ and $\overline{\range(\mcA)}$, respectively.
We have announced above that minimal-norm solutions can be computed using the Moore--Penrose generalized inverse, which we make precise by the following result:
\begin{theorem}{}{thm:existence}
    The following can be said about existence, uniqueness and explicit form of least-squares and minimal-norm solutions:
    \begin{itemize}
        \item The set of least-squares solutions associated to the general linear operator equation \eqref{eq:gloe} is non-empty if and only if $f\in\domain(\mcAdag)$. 
        \item The general linear operator equation \eqref{eq:gloe} has a unique minimal-norm solution if and only if $f\in\domain(\mcAdag)$. It is given by
        \beq
            u^\dagger := \mcAdag f
        \eeq
        and the set of all least-squares solutions is given by $\ls u^\dagger \rs + \kernel(\mcA)$.
    \end{itemize}
\end{theorem}

\begin{exercise}{}{}
    Proof Theorem~\ref{thm:existence}.
\end{exercise}

For non-symmetric matrices, it is well-known from linear algebra, that the Gaussian normal equation can be considered to obtain least-squares solutions. We now verify that this assertion is true in the general case:
\begin{theorem}{}{thm:gne}
    For given $f\in \domain(\mcAdag)$, $u\in\mcU$ is a least-squares solution of \eqref{eq:gloe} if and only if $u$ satisfies the Gaussian normal equation
    \beq
        \mcA^*\mcA u = \mcA^* f.
        \label{eq:gne}\tag{GNE}
    \eeq
\end{theorem}
\begin{proof}
    An element $u\in\mcU$ is a least-squares solution if and only if $\mcA u$ is the projection of $f$ onto $\range(\mcA)$, which is equivalent to $\mcA u - f \in \range(\mcA)^\perp$. Since $\range(\mcA)^\perp = \kernel(\mcA^*)$, this is equivalent \eqref{eq:gne}.
\end{proof}
Since $\mcAdag f$ is the least-squares solution of minimal-norm, we obtain from Theorem~\ref{thm:gne} that $\mcAdag f$ is a solution of \eqref{eq:gne} with minimal norm, i.e.,
\beq
    \mcAdag f = (\mcA^* \mcA)^{-1} \mcA^* f.
\eeq
This means that in order to approximate $\mcAdag f$ we may as well compute an approximation to the minimal-norm solution in \eqref{eq:gne}, a fact we will heavily use in the construction of regularization methods.

\begin{memo}{}{mem:generalized_inverse}
    In cases where \ref{it:H1} or \ref{it:H2} are violated, one can enforce existence and uniqueness of the solution by considering the generalized inverse $\mcAdag$.
\end{memo}

\section{Compact Operators}

Let us now take a look at \ref{it:H3}, i.e., the continuity of $A^\dagger$. In particular, we consider the special case of \emph{compact linear operators}, which play an important role in imaging applications:

\begin{definition}{Compact Linear Operators}{}
    Let $\mcA:\mcU\to\mcV$ be a continuous linear operator between the Banach spaces $\mcU$ and $\mcV$. Then $\mcA$ is said to be \emph{compact} if for any bounded set $U\subset\mcU$ the image $\mcA(U) \subset \mcV$ is pre-compact, i.e., has a compact closure  $\overline{\mcA(U)}$.
\end{definition}

\begin{exercise}{}{}
    Which of the operators discussed in section \ref{sec:examples_for_ip} are compact linear operators?
\end{exercise}

While compactness is quite common for operators considered in imaging applications it is also a source of ill-posedness
for the general linear operator equation \eqref{eq:gloe}:
\begin{theorem}{Ill-Posedness in Case of Compact Operators}{}
    Let $\mcA:\mcU\to\mcV$ be a compact linear operator between the infinite-dimensional Hilbert spaces $\mcU$ and $\mcV$, such that the dimension of $\range(\mcA)$ is infinite. Then the problem \eqref{eq:gloe} is ill-posed, as $\mcAdag$ is discontinuous.
\end{theorem}
\begin{proof}
Since $\mcU$ and $\range(\mcA)$ are infinite-dimensional and the dimension of $\range(\mcA)$ is always smaller than or equal to the dimension of $\kernel(\mcA)^\perp$, it follows that $\kernel(\mcA)^\perp$ is also infinite-dimensional. Hence, we can find a sequence $(u_n)$ with $u_n\in\kernel(\mcA)^\perp$, $\|u_n\|_\mcU = 1$ and $\ll u_n, u_k \rr_\mcU = 0$ for $n\neq k$. Since $\mcA$ is a compact operator, the sequence $(v_n) := (\mcA u_n)$ is compact and hence for each $\delta > 0$ we can find $k, l$ such that $\| v_k - v_l\|_\mcV < \delta$, but
\beq
    \| \mcAdag v_k -  \mcAdag  v_l\|_\mcU^2 = \| u_k - u_l \|_\mcU^2 = \|u_k\|_\mcU^2 - 2 \ll u_k, u_l \rr_\mcU + \|u_l\|_\mcU^2 = 2.
\eeq
It follows that $\mcAdag$ is unbounded.
\end{proof}

\begin{memo}{}{mem:compact}
    The generalized inverse $\mcAdag$ of a compact linear operator $\mcA$ is discontinuous.
\end{memo}

Another approach to discuss instability uses the singular value decomposition of compact linear operators. For a deeper understanding of the instability problem and the latter introduced regularization approach, we note and define:
\begin{theorem}{}{thm:svd}
Let $\mcA:\mcU\to\mcV$ be a compact linear operator. Then there exists
\begin{itemize}
    \item a not-necessarily infinite null sequence $\ls \sigma_i\rs_{i\in\N}$ with $\sigma_1 \geq  \sigma_2 \geq  ... > 0$,
    \item an orthonormal basis $\ls u_i \rs_{i\in \N}\subset \mcU$ of $\kernel(\mcA)^\perp$,
    \item an orthonormal basis $\ls v_i\rs_{i\in\N}\subset \mcV$ of $\overline{\range(\mcA)}$
\end{itemize}
with
\beq
    \mcA u_i = \sigma_i v_i, \quad \mcA^*v_i = \sigma_i u_i, \quad \forall i\in\N.
\eeq
Moreover, for all $u\in\mcU$ we have the representation
\beq
    \mcA u = \sum_{i=1}^\infty \sigma_i \ll u, u_i \rr_\mcU v_i.
\eeq
\end{theorem}

\begin{definition}{Singular Value Decomposition}{}
    A sequence $\ls(\sigma^2_i , u_i , v_i )\rs$ as in \cref{thm:svd} is called \alert{singular system} or \alert{singular value decomposition} (SVD) of the compact linear operator $\mcA$.
\end{definition}
Building on the SVD of the operator $\mcA$ one can show that for any $f\in\domain(\mcAdag)$ the following decomposition of the Moore--Penrose generalized inverse holds
\beq
    \mcAdag f = \sum_{i=1}^\infty \frac{1}{\sigma_i} \ll f, v_i \rr_\mcV u_i.
\eeq
With this decomposition we can show the following bound
\beq
    \|\mcAdag\|_{\mcU, \mcV} \geq \frac{\|\mcAdag v_j \|_\mcU}{\|v_j \|_\mcV} = \sqrt{\sum_{i=1}^\infty \frac{|\ll v_j , v_i \rr_\mcV|^2}{\sigma_i^2}} = \frac{1}{\sigma_j}.
    \label{eq:bound_norm_adag}
\eeq
Thus $\|\mcAdag\|_{\mcU, \mcV}$ may become arbitrarily large if the singular values of $\mcA$ tend to 0. 

\begin{exercise}{SVD of $\mcAdag$}{}
    Let $\mcA:\mcU\to\mcV$ be a compact linear operator and let $\ls\sigma_i, u_i, v_i\rs_{i\in\N}$ denote its SVD. Show that the following representation holds for any $f\in\domain(\mcAdag)$:
    \beq
        \mcAdag f = \sum_{i=1}^\infty \frac{1}{\sigma_i} \ll f, v_i \rr_\mcV u_i.
    \eeq
\end{exercise}

\begin{exercise}{Picard Criterion}{}
    Let $\mcA:\mcU\to\mcV$ be a compact linear operator, let $\ls\sigma_i, u_i, v_i\rs_{i\in\N}$ denote its SVD and let $f\in\overline{\range(\mcA)}$. Show that $f\in\range(\mcA)$ if and only if the so-called \emph{Picard criterion} 
\beq
    \| \mcAdag f \|^2 = \sum_{i=1}^\infty \frac{|\ll f, v_i \rr_\mcV|^2}{\sigma_i^2} < \infty 
\eeq
is satisfied. What does this imply?
\end{exercise}

\section{Regularization Methods}
\label{sec:reg}



In this section, we introduce the notion of regularization in a rigorous way for the infinite dimensional setting and discuss some basic properties. In general, we describe a \emph{linear regularization method} by a \emph{family of continuous linear operators} $\mfRa: \mcV \to \mcU$ for $\alpha \in I \subset (0, a)$ where the index set $I$ includes at least one sequence $(\alpha_n)_{n\in\N}$ such that $\alpha_n \to 0$. Of course, the regularization operator should converge to the generalized inverse in some sense as soon as $\alpha \to 0$, i.e., we expect for $f\in\domain(\mcAdag)$ that $\mfRa f \overset{\alpha\to 0}{\longrightarrow} \mcAdag f$. If, however, $f\in \mcV\backslash \domain(\mcAdag)$, we have to expect that $\|\mfRa f\|_\mcU \overset{\alpha\to 0}{\longrightarrow} +\infty$ due to the unboundedness of the generalized inverse. Concerning the \emph{regularization parameter}  $\alpha$, we expect that it should be possible to make a choice which depends on the noise level $\noiselvl$ and input data $\fd$ such that $\mfR_{\bs{\alpha}(\noiselvl, \fd)} f^\delta \overset{\noiselvl\to 0}{\longrightarrow} \mcAdag f$ for $f\in\domain(\mcAdag)$. This definition is based on \cite{Engl2000}.

\begin{definition}{Regularization}{def:regularization}
    A family $\ls\mfRa\rs_{\alpha\in I}$ of continuous linear operators is called \alert{regularization} (or regularization operator) for $\mcAdag$, if for all $f\in\domain(\mcAdag)$ there exists a \alert{parameter choice rule} $\bs{\alpha}:\R^+ \times \mcV \to I$ such that
    \beq
        \lim_{\noiselvl\to 0}\left(\sup \ls \| \mfR_{\bs{\alpha}(\noiselvl, \fd)}\fd - \mcAdag f \|_\mcU  \quad \middle| \quad \fd\in\mcV, \quad \| f-\fd\|_\mcV \leq \noiselvl \rs\right) = 0
        \label{eq:def_reg_cond_a}\tag{R1}
    \eeq
    and
    \beq
        \lim_{\noiselvl\to 0}\left(\sup \ls \bs{\alpha}(\noiselvl, \fd) \quad \middle| \quad \fd \in \mcV, \quad \|f-\fd\|_\mcV \leq \noiselvl \rs \right)= 0.
        \label{eq:def_reg_cond_b}\tag{R2}
    \eeq
    For a specific $f\in\domain(\mcAdag)$, the pair $(\mfRa, \bs{\alpha})$ is called \alert{(convergent) regularization method} of \labelcref{eq:gloe} if \labelcref{eq:def_reg_cond_a} and \labelcref{eq:def_reg_cond_b} hold.
\end{definition}
\labelcref{eq:def_reg_cond_a} intuitively means that even for the worst kind of noise the regularization method should be able to approximate the true solution when the noise level goes to zero.

Sometimes, the parameter choice rule might only depend on the noise level:
\begin{itemize}
    \item If the choice only depends on the noise level $\noiselvl$ it is called an \alert{a priori parameter choice rule} and we write it as $\bs{\alpha}(\noiselvl)$.
    \item If the choice also depends on the given noisy data $\fd$ it is called an \alert{a posteriori parameter choice rule} and we write it as $\bs{\alpha}(\noiselvl, \fd)$.
\end{itemize}
In practice, it might be tempting to choose $\alpha$ in dependence of the known noisy data $\fd$ disregarding the noise level $\noiselvl$ which might be difficult to determine. The following result shows that such an approach \emph{cannot} result in a convergent regularization method for ill-posed problems, or, in other words, such a strategy can only work for well-posed problems,
which could also be solved without regularization:
\begin{theorem}{{
\href{https://doi.org/10.1016/0041-5553(84)90253-2}{\color{white}{Bakushinskii}} \cite{bakushinskii1984remarks}}
}{thm:choice_rule}
Let $\mcA:\mcU\to\mcV$ be a bounded linear operator and let $\ls\mfRa\rs$ be a regularization for $\mcAdag$, such that the regularization method converges for every $f\in\domain(\mcAdag)$ and such that the parameter choice $\bs{\alpha}$ depends on $\fd$ only (and not on $\noiselvl$). Then $\mcAdag$ can be extended to a continuous operator on $\mcV$.
\end{theorem}
\begin{exercise}{}{}
    Proof \cref{thm:choice_rule}.
\end{exercise}

\comment{This result rules out error-free parameter choices $\alpha=\bs{\alpha}(\noiselvl)$ as convergent regularization
methods. Of course, it does not mean that such strategies (which are actually used as heuristic
approaches in practice) do not behave well for finite $\noiselvl$, but at least it indicates that the results
have to be considered with care in such cases.}

Let us  derive
some basic properties to be satisfied by regularization methods.

\begin{proposition}{}{}
    Let $\mcA:\mcU\to\mcV$ be a bounded linear operator and let $\mfRa:\mcV\to\mcU$ for $\alpha\in\R^+$ be a family of continuous operators. The family $\ls\mfRa\rs$ is a regularization of $\mcAdag$ if
    \beq
        \mfRa \to \mcAdag \quad \text{pointwise on } \domain(\mcAdag) \quad \text{as } \alpha\to 0.
    \eeq
    In particular, in this case there exists an a-priori parameter choice rule $\bs{\alpha}$ such that $(\mfRa, \bs{\alpha})$ is a convergent regularization method for \labelcref{eq:gloe}.
\end{proposition}
\begin{proof}
    Let $f\in\domain(\mcAdag)$ be arbitrary, but fixed. Due to the pointwise convergence, we can find a monotone function $\sigma:\R^+\to\R^+$ such that for every $\epsilon>0$
    \beq
        \| \mfR_{\sigma(\epsilon)}f - \mcAdag f\|_\mcU \leq \frac{\epsilon}{2}.
    \eeq
    The operator $\mfR_{\sigma(\epsilon)}$ is continuous for fixed $\epsilon$ and hence, there exists $\rho(\epsilon)\in I$ such that
    \beq
        \| \mfR_{\sigma(\epsilon)}g - \mfR_{\sigma(\epsilon)}f \|_\mcU \leq \frac{\epsilon}{2} \quad \text{if } \| g- f\|_\mcV \leq \rho(\epsilon).
    \eeq
    Without loss of generality, we can assume that $\rho$ is monotone increasing, continuous, and 
    \beq
        \lim_{\epsilon\to 0}\rho(\epsilon) = 0.
    \eeq 
    Hence, there exists a well-defined inverse function $\rho^{-1}$ on the range of $\rho$ with the same properties. Now we extend $\rho^{-1}$ to a continuous, strictly monotone function on $\R^+$ and define the parameter choice rule as
    \beq
        \bs{\alpha}:\R^+ \to \R^+, \quad \noiselvl \mapsto \sigma(\rho^{-1}(\noiselvl)).
    \eeq
    By our construction, we have for $\noiselvl := \rho(\epsilon)$ that
    \beq
        \| \mfR_{\bs{\alpha}(\noiselvl)} \fd - \mcAdag f \|_\mcU \leq \| \mfR_{\bs{\alpha}(\noiselvl)} \fd - \mfR_{\bs{\alpha}(\noiselvl)} f \|_\mcU + \| \mfR_{\bs{\alpha}(\noiselvl)} f - \mcAdag f \|_\mcU \leq \epsilon,
    \eeq
    if $\| f -\fd\|_\mcV\leq\noiselvl$, since we have $\bs{\alpha}(\noiselvl)=\sigma(\epsilon)$. Hence, $(\mfRa, \bs{\alpha})$ is a convergent regularization method for \labelcref{eq:gloe}.
\end{proof}

We now know that any family of continuous operators that converges pointwise to the generalized inverse defines a regularization method. Vice versa, we can conclude from \eqref{eq:def_reg_cond_b} that
\beq
    \lim_{\noiselvl\to 0} \mfR_{\bs{\alpha}(\noiselvl,\fd)}f = \mcAdag f, \quad \forall f\in\domain(\mcAdag),
\eeq
and thus, if $\bs{\alpha}$ is continuous in $\noiselvl$, this implies
\beq
    \lim_{\sigma\to 0} \mfR_\sigma f = \mcAdag f,
\eeq
i.e., a convergent regularization method with continuous parameter choice rule implies pointwise convergence of the regularization operators. Now we turn our attention to the behavior of the regularization operators on $\mcV \backslash \domain(\mcAdag)$. Since the generalized inverse is not defined on this set, we cannot expect that $\mfRa$ remains bounded on this set as $\alpha\to 0$. This is indeed confirmed by the next result:
\begin{proposition}{}{}
    Let $\mcA:\mcU\to\mcV$ be a continuous linear operator and $\mfRa:\mcV\to\mcU$ be a family of continuous linear regularization operators. Then, $u_\alpha := \mfRa f$ converges to $\mcAdag f$ as $\alpha \to 0$ for $f\in\domain(\mcAdag)$. Moreover, if $\sup_{\alpha > 0} \| \mcA \mfRa\|_{\mcU, \mcU} < \infty$, then $\|u_\alpha\|_\mcU\to\infty $ for $f\notin \domain(\mcAdag)$.
\end{proposition}
\begin{proof}
    The convergence of $u_\alpha$ for $f\in\domain(\mcAdag)$ has been verified above. Now, let $f\notin\domain(\mcAdag)$ and assume there exists a sequence $\alpha_n \to 0$ such that $\|u_{\alpha_n}\|_\mcU$ is uniformly bounded. Then there exists a weakly convergent subsequence (denoted by $u_{\alpha_m}$) with some limit $u\in\mcU$, and since continuous linear operators are weakly continuous, too, we have $\mcA u_{\alpha_m} \to \mcA u$. On the other hand, it holds for any $f\in\domain(\mcAdag)$ that $\mcA \mfR_{\alpha_m}f \to \mcA\mcAdag f = \mcQ f.$ Since $\mcA\mfRa$ are uniformly bounded operators, this also holds for any $f\in\mcV$ (see. \href{https://books.google.de/books?hl=de&lr=&id=abvkBwAAQBAJ&oi=fnd&pg=PA1&dq=Applied+Functional+Analysis:+Main+Principles+and+Their+Applications&ots=PUg90qi7W1&sig=Z7nq2MUGZ88iBnVxGuANytr0mW8&redir_esc=y#v=onepage&q=Applied%20Functional%20Analysis%3A%20Main%20Principles%20and%20Their%20Applications&f=false}{Zeidler p.~174 Prop.~2}), i.e., in particular also for $\tilde{f}\notin\domain(\mcAdag)$ it holds $\mcA u_{\alpha_m} = \mcA\mfR_{\alpha_m} \tilde{f} \to Q \tilde{f}$.
    Together with the first part, we get $\mcA u = Q\tilde{f}$ -- this however can only ever hold for $\tilde{f}\in\domain(\mcAdag)$, thus we reach a contradiction.
\end{proof}
Next, we consider the properties of a-priori parameter choice rules:
\begin{proposition}{}{}
    Let $\mcA:\mcU\to\mcV$  be a continuous linear operator and $\mfRa:\mcV\to\mcU$ be a family of continuous linear regularization operators, with a-priori parameter choice rule $\alpha=\bs{\alpha}(\noiselvl)$. Then, $(\mfRa, \bs{\alpha})$ is a convergent regularization method if and only if
    \beq
        \lim_{\noiselvl\to 0} \bs{\alpha}(\noiselvl) = 0, \quad \lim_{\noiselvl\to 0} \noiselvl\| \mfR_{\bs{\alpha}(\noiselvl)}\|_{\mcV, \mcU} = 0
        \label{eq:prop_a_lim}
    \eeq
    hold.
\end{proposition}
\begin{proof}
    If \cref{eq:prop_a_lim} hold, then for all $\fd\in\mcV$ with $\|f-\fd\|\leq\noiselvl$,
    \begin{align}
    \|\mfR_{\bs{\alpha}(\noiselvl)} \fd -\mcAdag f\|_\mcU &\leq \|u_{\bs{\alpha}(\noiselvl)} - \mcAdag f\|_\mcU + \|u_{\bs{\alpha}(\noiselvl)} - \mfR_{\bs{\alpha}(\noiselvl)} \fd \|_\mcU \\
    &\leq \|u_{\bs{\alpha}(\noiselvl)} - \mcAdag f\|_\mcU + \noiselvl \| \mfR_{\bs{\alpha}(\noiselvl)}\|_{\mcV, \mcU} 
    \end{align}
Because of \cref{eq:prop_a_lim} and since regularization operators converge pointwise, the right-hand side tends to zero as $\noiselvl\to 0$, i.e., $(\mfRa, \bs{\alpha})$ is a convergent regularization method. Assume vice versa that $(\mfRa, \bs{\alpha})$ is a convergent regularization method. Assume that there
exists a sequence $\noiselvl_n\to 0$ such that $\noiselvl_n \|\mfR_{\bs{\alpha}(\noiselvl_n)}\|_{\mcV, \mcU} \geq C > 0$ for some constant $C$. Then we can find a sequence  $g_n$ with $\|g_n\| = 1$ such that $\noiselvl_n \|\mfR_{\bs{\alpha}(\noiselvl_n)}g_n\|_{\mcV, \mcU} \geq \frac{C}{2}$. Then, for any $f\in\domain(\mcAdag)$ and
$f_n := f + \noiselvl_n g_n$ we obtain $\|f-f_n\|_\mcV \leq \noiselvl_n$, but
\beq
    (\mfR_{\bs{\alpha}(\noiselvl_n)} g_n - \mcAdag f) = (\mfR_{\bs{\alpha}(\noiselvl_n)} g - \mcAdag f) + \noiselvl_n \mfR_{\bs{\alpha}(\noiselvl_n)}g_n
\eeq
does not converge, since the second term is unbounded. Hence, for $\noiselvl_n$ sufficiently small, \labelcref{eq:def_reg_cond_b}
is not satisfied.
\end{proof}

\begin{memo}{}{}
    Regularization is a method for approximating a possibly discontinuous generalized inverse $\mcAdag$ by a family of continuous linear operators. 
\end{memo}

A violation of \ref{it:H3}, i.e. instability, arises if the spectrum of the operator $\mcA$ is not bounded away from zero -- see \cref{eq:bound_norm_adag}. Thus, it seems natural to construct regularizing approximations such that the smallest singular values are modified. Using the singular value decomposition of the generalized inverse, this suggests the use of regularization operators of the form
\beq
    \mfRa f = \sum_{i=1}^\infty r_\alpha(\sigma_i) \ll f, v_i \rr_\mcV u_i
    \label{eq:mfR_svd}
\eeq
with some function $r_\alpha:\R^+\to\R^+$ such that $r_\alpha(\sigma) \overset{\alpha\to 0}{\longrightarrow} \frac{1}{\sigma}$. Such an operator $\mfRa$ is a regularization operator if
\beq
    r_\alpha(\sigma) \leq C_\alpha < \infty, \quad \forall \sigma\in\R^+.
\eeq
If this bound holds, we directly get
\beq
    \|\mfRa f\|_\mcU^2 
    = \sum_{i=1}^\infty(r_\alpha(\sigma_i))^2 |\ll f, v_n\rr_\mcV|^2 
    \leq C_\alpha^2 \sum_{i=1}^\infty  |\ll f, v_n\rr_\mcV|^2 
    \leq C_\alpha^2 \| f \|_\mcV^2,
    \eeq
i.e., $C_\alpha$ is a bound for the norm of $\mfRa$. From the pointwise convergence of $r_\alpha$ we immediately conclude the pointwise convergence of $\mfRa$ to $\mcAdag$. Moreover, condition \cref{eq:prop_a_lim} on the choice of the regularization parameter can be replaced by
\beq
    \lim_{\delta\to 0} \delta C_{\bs{\alpha}(\delta)} = 0.
    \label{eq:prop_b_lim}
\eeq
since $\|\mfRa \|_{\mcV, \mcU} \leq C_\alpha$. Hence, if \cref{eq:prop_b_lim} is satisfied, we know that $(\mfRa, \bs{\alpha})$ with $\mfRa$ defined by \cref{eq:mfR_svd} is a convergent regularization method.

\begin{example}{Tikhonov Regularization}{}
    Defining $r_\alpha(\sigma) = \frac{\sigma}{\sigma^2+\alpha}$ we can recover once more Tikhonov regularization. The regularized solution for some $f\in\mcV$ and $\alpha\in\R^+$ is 
    \beq
        u_\alpha := \mfRa f = \sum_{i=1}^\infty \frac{\sigma_i}{\sigma_i^2 + \alpha} \ll f, v_i \rr_\mcV u_i.
    \eeq    
    Note, that we can compute $u_\alpha$ without knowledge of the singular system by solving
    \beq
        (\mcA^* \mcA + \alpha I) u_\alpha = \mcA^* f.
    \eeq
    This can be shown using the representations 
    \beq 
        \mcA^* f = \sum_{j=1}^\infty \sigma_j \ll f, v_j \rr_\mcV u_j \quad \text{ and } \quad u = \sum_{j=1}^\infty \ll u, u_j\rr u_j.
    \eeq 
    \comment{
    and calculating
    \begin{align}
    (\mcA^* \mcA + \alpha I) u_\alpha 
    &= \mcA^* \sum_{i=1}^\infty \sigma_i \ll  u_\alpha, u_i \rr_\mcU v_i + \alpha I u_\alpha \\
    &= \sum_{j=1}^\infty \sigma_j \ll v_j, \sum_{i=1}^\infty \sigma_i \ll  u_\alpha, u_i \rr_\mcU v_i \rr_\mcV u_j + \alpha I u_\alpha \\ 
    &= \sum_{j=1}^\infty \sigma_j^2 u_j u_j^\top u_alpha + \alpha  u_\alpha \\ 
    &= \sum_{j=1}^\infty (\sigma_j^2 +\alpha) u_\alpha.
    \end{align}
    }

    Recalling the binomial theorem, we can estimate $\sigma^2 + \alpha \geq 2\sigma\sqrt{\alpha}$ and hence, 
    \beq
    r_\alpha(\sigma) = \frac{\sigma}{\sigma^2+\alpha} \leq \frac{\sigma}{2\sigma\sqrt{\alpha}} = \frac{1}{2\sqrt{\alpha}}.
    \eeq
    Thus, the condition for a convergent regularization method in this case becomes 
    \beq
        \lim_{\delta\to 0} \delta C_\alpha = \lim_{\delta\to 0} \frac{\delta}{2\sqrt{\alpha}} = 0.
    \eeq
\end{example}

\section{Optimization and Sparsity-Based Regularization}\label{sec:opt}

In the last section, we have considered Tikhonov regularization, which can either be expressed via the solution of the linear system
\beq
(A^* A + \alpha I) u_\alpha = A^* \fd.
\label{eq:GNE_Tik}\tag{Tikhonov}
\eeq
or via a filtering of the singular values with the function
\beq
r_\alpha(\sigma) = \frac{\sigma}{\sigma^2+\alpha}.
\eeq
In practice, when we deal with very large matrices $A$, computing the SVD directly can become infeasible. Moreover, even building the matrix $A$ or its conjugate and performing matrix-matrix multiplications might be either too slow or even impossible due to memory constraints. Therefore, we now want to explore algorithms that can efficiently solve linear systems of the form \eqref{eq:GNE_Tik}. Most of the theory and concepts presented in the following can be applied to the case where $\mcU, \mcV$ are infinite dimensional. However, since we now deal with a concrete application, we will be content to consider $U=\R^n$. The exposition in the following is loosely based on the lecture notes \cite{Brandt22}. We give the following references for further introductions to the topic of convex optimization, \cite{barbu2012convexity,Bauschke2017,ekeland1999convex,rockafellar2015convex}.
\subsection{Preliminaries on Convex Analysis}
We first notice that the linear system \labelcref{eq:GNE_Tik} is the first order optimality condition of a quadratic optimization problem, which is further specified in the following lemma:
\begin{lemma}{}{}
The functional $\obj_\alpha:\mcU\to[-\infty,+\infty]$
\begin{align}
\obj_\alpha(u) = \frac{1}{2}\| A u - \fd \|^2 + \frac{\alpha}{2} \|u \|^2,
\end{align}
has a unique global minimum, which fulfills \labelcref{eq:GNE_Tik}.
\end{lemma}
\begin{proof}
 Exercise.
\end{proof}
This means that in order to find a solution to \labelcref{eq:GNE_Tik}, we have to solve an optimization problem. In the following, we will thus consider optimization problems  
\beq\label{eq:opt}\tag{Opt}
\text{find}\qquad u^* = \argmin_{u\in C} \obj(u),
\eeq
where $\obj:C \to \Rc$ is assumed to be a convex functional mapping to the extended real line $\Rc:=(-\infty,+\infty]$ and $C$ is assumed to be a convex set. Whenever $C=U$, \labelcref{eq:opt} is called \alert{global} optimization. Whenever $C$ is a convex subset of $U$ \labelcref{eq:opt} is called \alert{constrained} optimization. Most of our applications deal with global optimization.
A well known result about the existence of minimizers is the following: If $C$ is compact and $\obj$ is continuous, the existence of a minimizers can be inferred by the Weierstrass theorem\footnote{Karl Weierstraß (1815--1897) was a German mathematician.}. As we are interested in global optimization for possibly unbounded $C$ and possibly discontinuous $\obj$, we recall some notations and facts from convex analysis. 
\begin{definition}{}{}
Let $\obj:C\to \Rc$ be a function, defined over the convex set $C$.
\begin{enumerate}
\item The effective domain of a function is defined as 
\beq
\dom(\obj) =\{u\in C: \obj(u) < +\infty\}.
\eeq
A function is called \alert{proper} if $\dom(\obj)\neq \emptyset$.
\item A function is called \alert{(strictly) convex} if for every $u, \tilde{u} \in C$ and every $t\in(0,1)$ we have
\beq
\obj(t u + (1-t) \tilde{u}) \overset{(<)}{\leq} t\obj(u) + (1-t) \obj(\tilde{u}).
\eeq
For $\nu>0$, a function $\obj$ is called  \alert{$\nu$-strongly convex} , if $\obj - \frac{\nu}{2} \norm{\cdot}^2$ is convex.

\item A point $u^*\in C$ is called \alert{local minimizer}, if there exists $\epsilon>0$ such that
\beq
\obj(u^*) \leq \obj(u)\qquad \text{for all } u\in C \text{ with } \quad\norm{u - u^*}\leq \epsilon. 
\eeq
\item A point $u^*\in C$ is called \alert{global minimizer} if
\beq
\obj(u^*) \leq \obj(u)\qquad \text{for all}\quad u\in C. 
\eeq
\end{enumerate}
\end{definition}

\begin{lemma}{}{}
A differentiable function is convex iff %
for all $u,\tilde{u}\in C$, we have
\begin{align}
\obj(u) \geq \obj(\tilde{u}) + \langle \nabla \obj(\tilde{u}), u - \tilde{u}\rangle.
\end{align}
Furthermore, it is $\nu$-strongly convex iff
\begin{align}
\obj(u) \geq \obj(\tilde{u}) + \langle \nabla \obj(\tilde{u}), u - \tilde{u}\rangle + \frac{\nu}{2} \norm{u-\tilde{u}}^2
\quad\Leftrightarrow\quad
\langle 
\nabla \obj(u) - \nabla \obj(\tilde{u}), u - \tilde{u} \rangle \geq \nu \norm{u - \tilde{u}}^2
.
\end{align}
\end{lemma}
\begin{lemma}{}{}
If $\obj:C \to \Rc$ is a convex function, then we have the following results:
\begin{enumerate}
\item If $u^*\in C$ is a local minimizer, then it is a global minimizer.
\item If $\obj$ is strictly convex, then there exists at most one global minimizer.
\end{enumerate}
\end{lemma}
%
%
\begin{definition}{}{}
\begin{enumerate}
\item A function $\obj:C\to\Rc$ is called \alert{lower-semicontinuous}, if for every sequence $u^{(k)} \to u,$ with $ u^{(k)},u\in C$ we have that
\begin{align}
\obj(u) \leq \liminf_{k\to\infty}\obj(u^{(k)}).
\end{align}
\item A function is called \alert{coercive} if for every sequence $(u^{(k)})_{k\in\N}$ with $\lim_{k\to\infty} \norm{u^{(k)}}=+\infty$ we have that
\begin{align}
\lim_{k\to\infty} \obj(u^{(k)}) = +\infty.
\end{align}
\end{enumerate}
\end{definition}
We can now formulate the following lemma about the existence of solutions:
\begin{lemma}{Existence of Solutions}{}
Let $C\neq\emptyset$ be a closed set and $\obj$ be a proper, coercive, lower-semicontinuous function, then there exists a global minimizer $u^*\in C$ of $\obj$.
\end{lemma}
\begin{proof}
Exercise.
\end{proof}
\noindent%
If the function $\obj$ is convex and differentiable, there are so-called \alert{optimality conditions}, which allow us to test, whether a point $u$ is a global minimizer.
\begin{lemma}{}{}
\begin{enumerate}
\item Let $C\neq\emptyset$ be a convex set and let $\obj$ be a convex and differentiable function. Then $u^*$ is a global minimizer over $C$ iff
\begin{align}
\langle \nabla \obj(u^*), u - u^* \rangle = 0 \qquad\forall u\in C.
\end{align}
\item In the case of global optimization, this reduces to $\nabla \obj(u^*) = 0$.
\end{enumerate}
\end{lemma}
\begin{proof}
Exercise.
\end{proof}

\subsection{Gradient Descent}

For general functions $\obj$ it is difficult to find the global minimum, therefore, we want to derive an iterative algorithm to approximate it. We first introduce the so-called \alert{gradient descent algorithm}, which can be traced back to Cauchy\footnote{Augustin-Louis Cauchy (1789--1857) was a French mathematician.} \cite{cauchy1847methode}. In the following, we want to offer two illustrative motivations for this algorithm:
\begin{enumerate}[label=(M\arabic*)]
\item\label{it:M1} Assume you start hiking on a mountain at a point $u^{(0)}$ and you want to get to the valley. Unfortunately, it's completely dark so you can't see anything, and you have no map available. How do you get down to the valley? A possible approach is to examine your surroundings and take a step into a direction, that leads to a point deeper than your current spot. This is referred to as a \alert{descent direction}, where for a differentiable function $\obj$ the steepest descent at $u$ is given by $-\nabla \obj(u)$. Assuming you have a step size of $\tau>0$ this gives you the update rule
\begin{align}
 u^{(k+1)} = u^{(k)} - \tau\, \nabla \obj(u^{(k)}).
\end{align}
\item\label{it:M2} The second motivation uses the idea or replacing the function $\obj$ with something easier. If we are at a point $u^{(k)}$, we can use a first order Taylor approximation, i.e., a \alert{linearization} around $u^{(k)}$, to infer (neglecting the necessary properties on $\obj$ for this to hold)
\begin{align}
\obj(u) \approx \obj(u^{(k)}) + \langle \nabla\obj(u^{k}), u - u^{(k)} \rangle + \frac{1}{2\tau}\norm{u - u^{(k)}}^2 =: \tilde{\obj}(u; u^{(k} \tau)
\end{align}
for some $\tau>0$. We can then define the new iterate as the global minimizer of the surrogate, i.e.,
\begin{align}
u^{(k+1)} = \argmin_{u} \tilde{\obj}(u; u^{(k)} \tau).
\end{align}
Since $\tilde{\obj}(u; u^{(k)} \tau)$ is convex, we can use the optimality condition to infer that
\begin{align}
\nabla \tilde{\obj}(u^{(k+1)}; u^{(k)} \tau) = 0 
\qquad &\Leftrightarrow\qquad 
\nabla\obj(u^{(k)}) + \frac{1}{\tau}( u^{(k+1)} - u^{(k)})= 0\\
& \Leftrightarrow\qquad
u^{(k+1)} = u^{(k)} - \tau\, \nabla\obj(u^{(k)}).
\end{align}
\end{enumerate}
This motivates the definition of gradient descent:
\begin{definition}{Gradient Descent}{}
Let $\obj$ be a differentiable function and let the initial point $u^{(0)}_\GD$ and step size $\tau$ be given, then the iterates of the gradient descent algorithm are given as
\begin{align}
u^\kk_\GD := u^\k_\GD - \tau\, \nabla \obj(u^\k_\GD).
\end{align}
\end{definition}
A typical question for a given optimization algorithm is if and how fast the iterates converge to a global minimum. To answer this, we need another assumption on the functional $\obj$, namely \alert{Lipschitz continuity} of its gradient with constant $L$, i.e.,
\begin{align}
\norm{\nabla\obj(u) - \nabla\obj(\tilde{u})} \leq L \norm{u - \tilde{u}}\quad \forall u, \tilde{u}.
\end{align}
Then we have the following result.
\begin{theorem}{Convergence of GD}{}
Let $\obj$ be a convex function with global minima $S\neq \emptyset$, and $L$-Lipschitz continuous gradient. For a step size $\tau= 1/L$ we have that the iterates of gradient descent fulfill
\begin{align}
\obj(u^\k) - \obj(u^*) \leq \frac{\norm{u^{(0)} - u^*}^2}{2 \tau\, k}\qquad\text{ for any } u^*\in S.
\end{align}
If additionally $\obj$ is $\nu$-strongly convex, then for any step size $\tau < 1/L$ we have
\beq
\norm{u^\k - u^*}^2 \leq (1- \nu \tau)^k \norm{u^{(0)} - u^*}^2.
\eeq
\end{theorem}
\begin{proof}
The first part is a special case of \cref{thm:convpgd}. The second part is an exercise.
\end{proof}
\begin{remark}{}{}
In practice, one often chooses step sizes that are much larger than the bound given above. Furthermore, usually one also chooses an adaptive, non-constant step size $\tau^\k$. Consider for example, $\obj(u):=u^4$ for $U=\R$. Then given any step size, if we start at the point $u^{(0)} = \frac{1}{\sqrt{2\tau}}$ we see that the algorithm does not converge, thus indicating, that we cannot prove a convergence statement for every initial condition.
\end{remark}
\begin{figure}[t]
\begin{subfigure}{.32\textwidth}
\includegraphics[width=\textwidth]{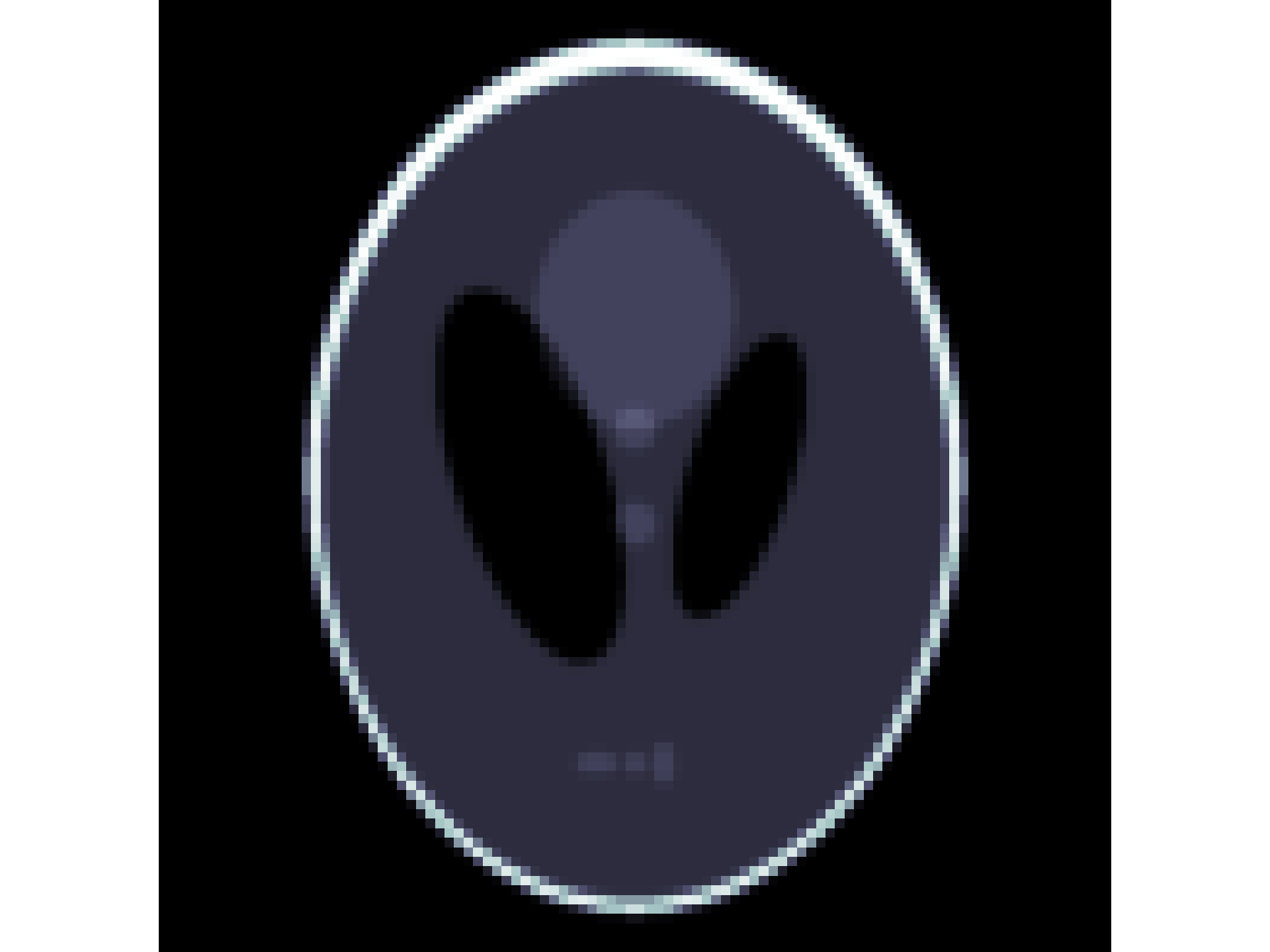}
\caption{Ground truth $u$.}\label{fig:Tika}
\end{subfigure}\hfill%
\begin{subfigure}{.32\textwidth}
\includegraphics[width=\textwidth]{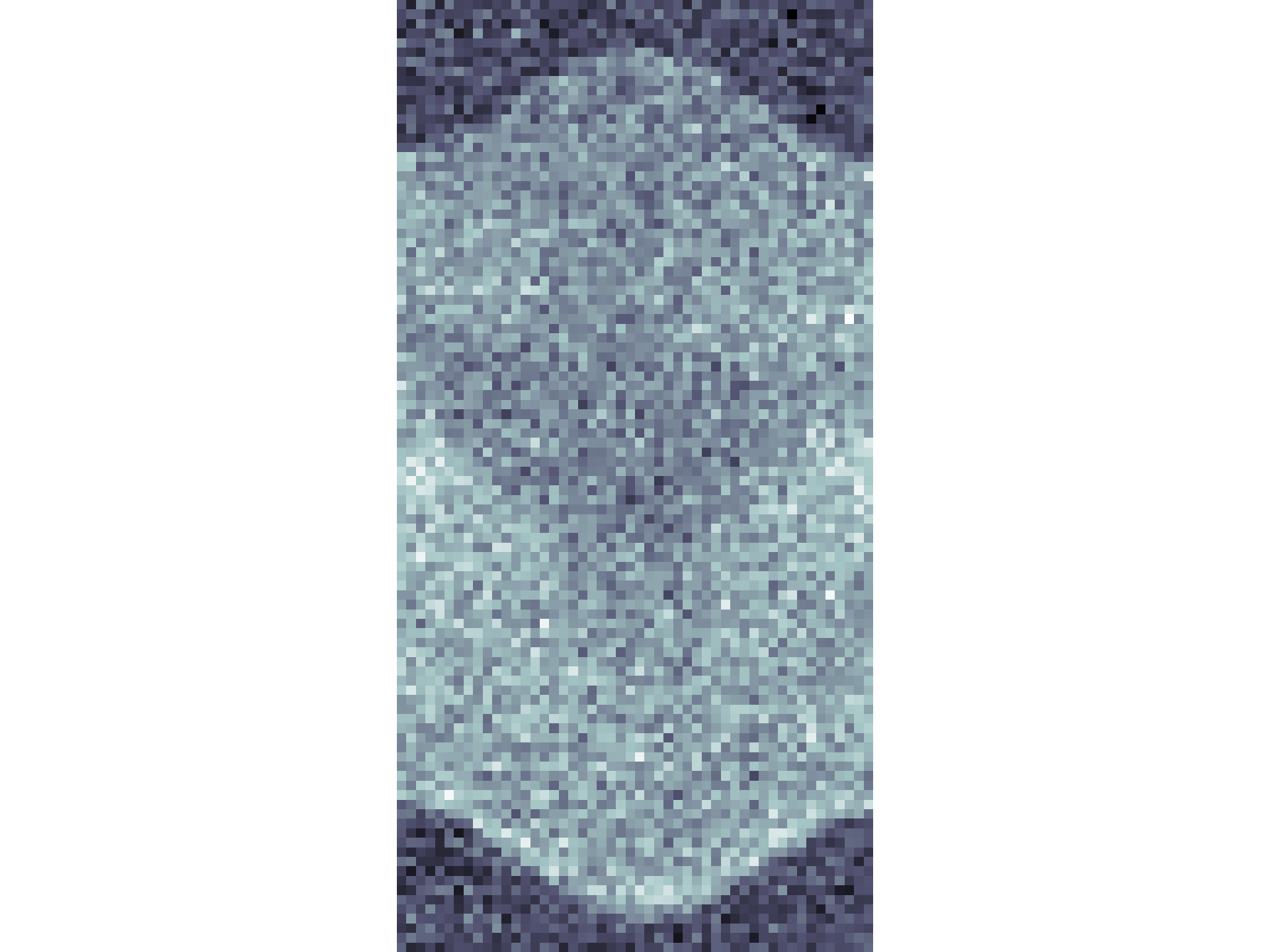}
\caption{Noisy sinogram data $\fd$.}\label{fig:Tikb}
\end{subfigure}\hfill%
\begin{subfigure}{.32\textwidth}
\includegraphics[width=\textwidth]{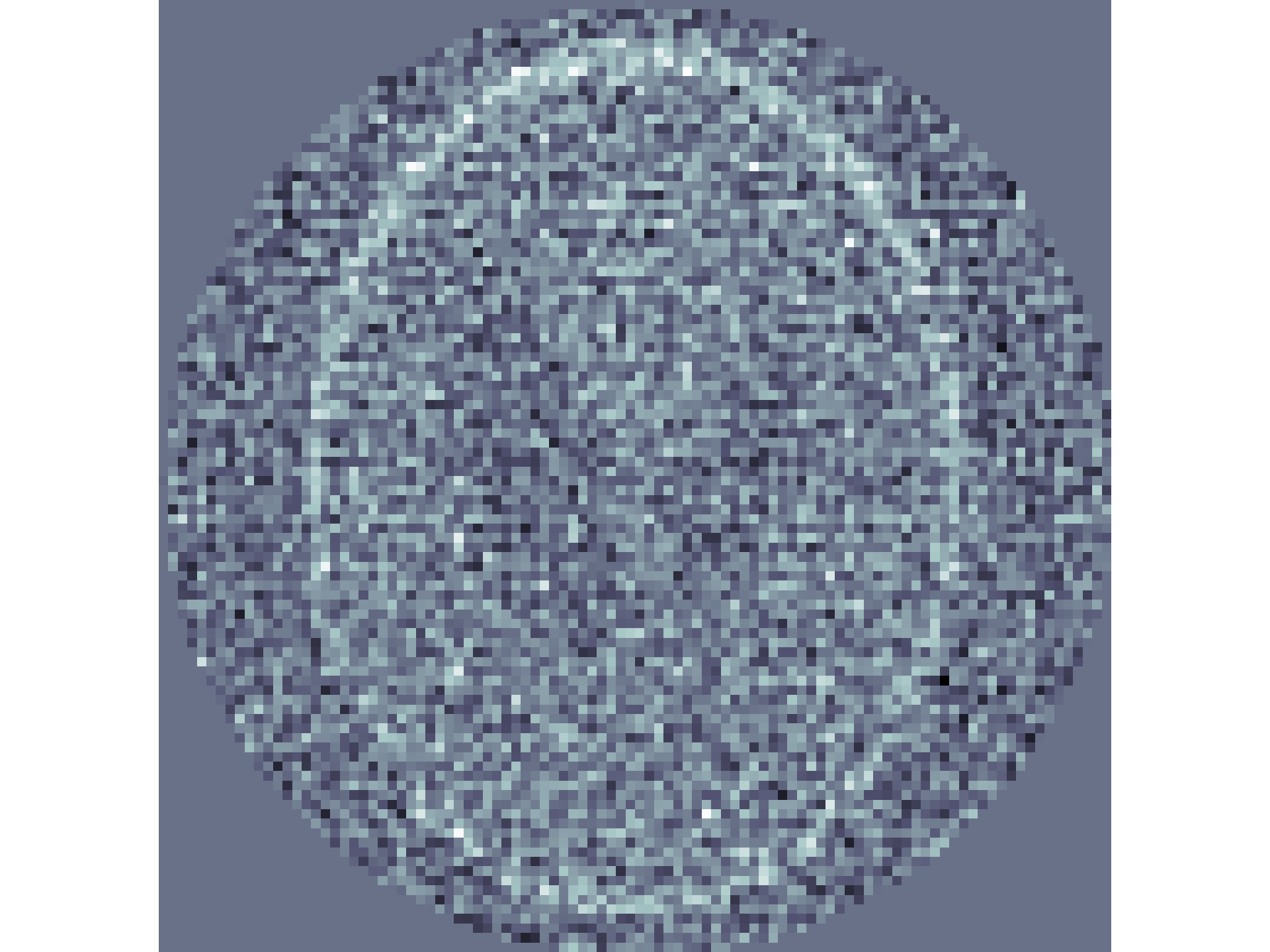}
\caption{Pseudo inverse $A^\dagger \fd$.}\label{fig:Tikc}
\end{subfigure}\\
\begin{subfigure}{.32\textwidth}
\includegraphics[width=\textwidth]{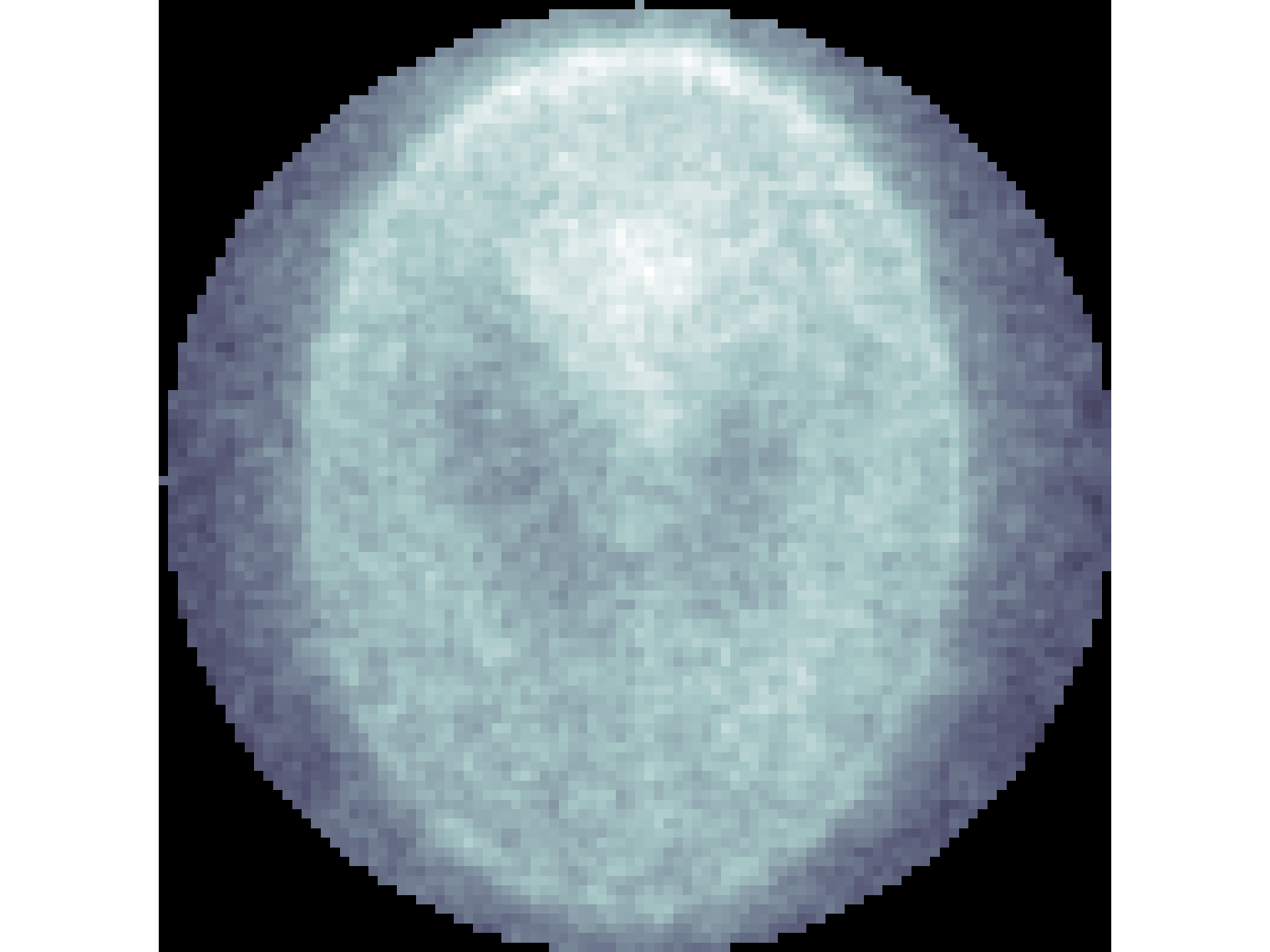}
\caption{$\ell^2$-regularization with $\regp=0$.}\label{fig:Tikd}
\end{subfigure}\hfill%
\begin{subfigure}{.32\textwidth}
\includegraphics[width=\textwidth]{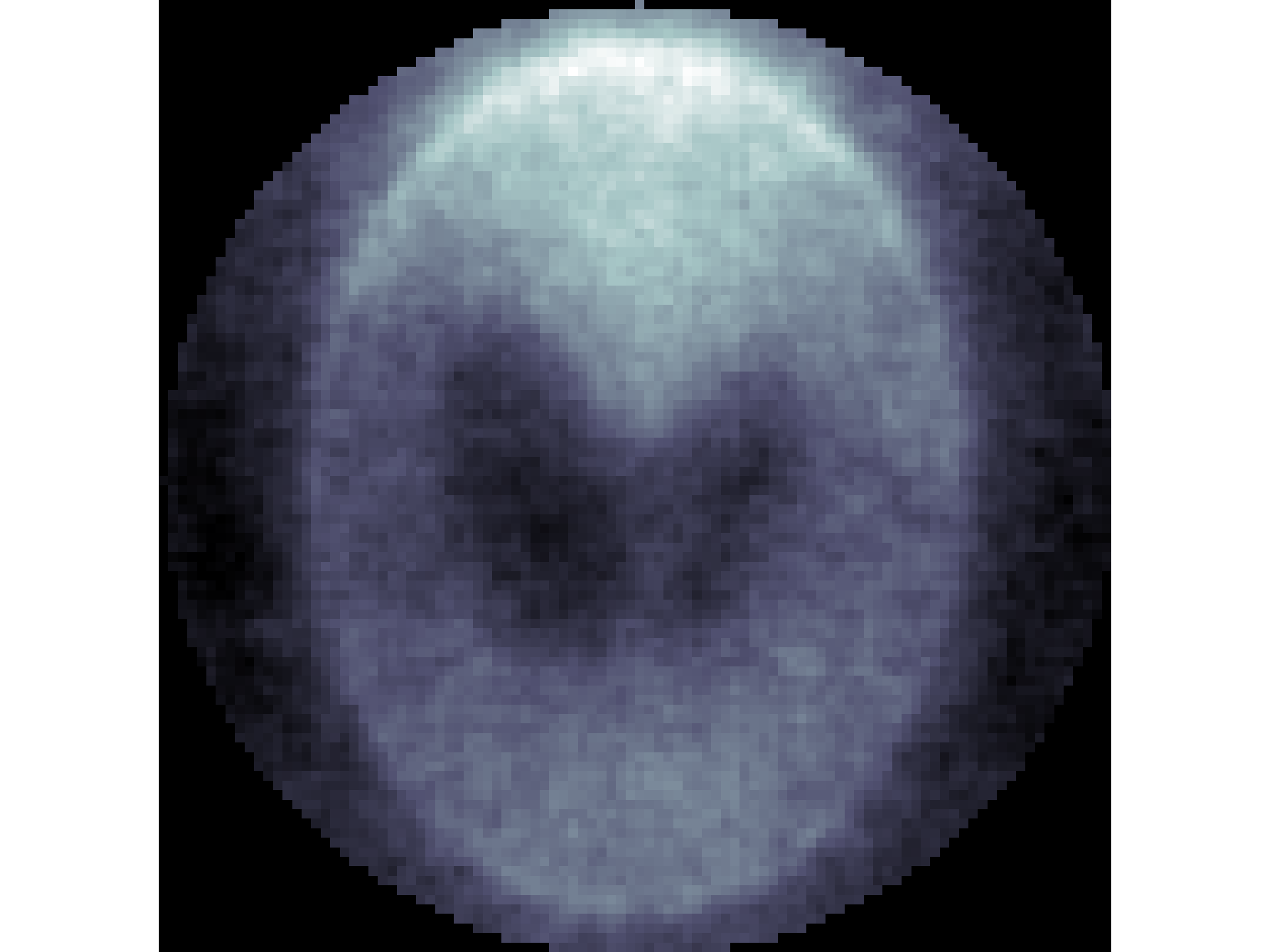}
\caption{$\ell^2$-regularization with $\regp=0.55$.}\label{fig:Tike}
\end{subfigure}\hfill%
\begin{subfigure}{.32\textwidth}
\includegraphics[width=\textwidth]{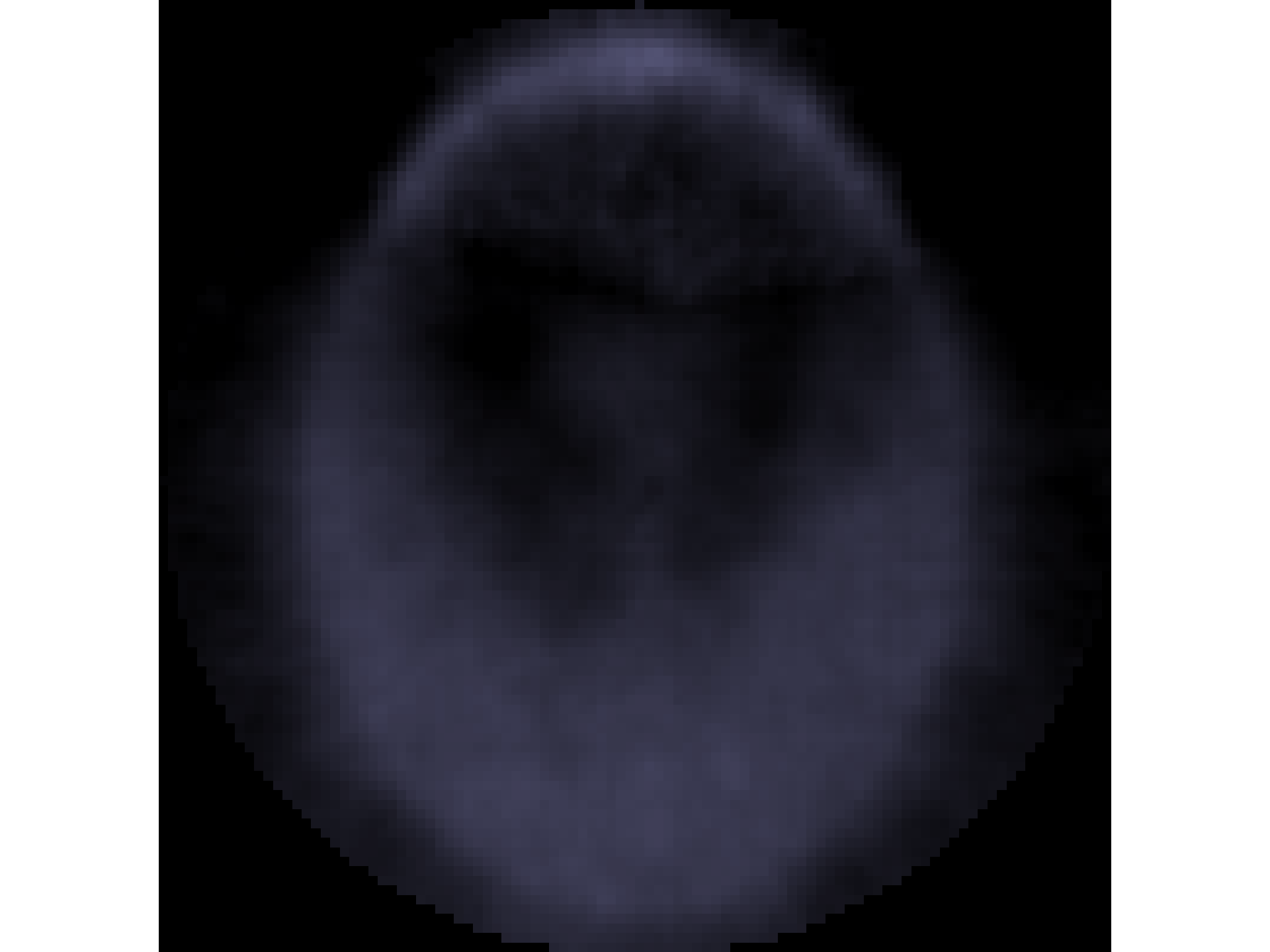}
\caption{$\ell^2$-regularization with $\regp=2$.}\label{fig:Tik}\label{fig:Tikf}
\end{subfigure}
\caption{We compare the reconstruction results of a computed tomography toy example (a)-(b) achieved using the pseudo-inverse (c) and GD applied on $\ell^2$-based Tikhonov regularization for different regularization parameters $\alpha$ (d)-(f). We can observe that the GD iteration does not recover $A^\dagger f^\delta$, although it is an admissible solution to the problem GD tries to solve.}\label{fig:Tik}
\end{figure}
In \cref{fig:Tik} we consider the example of computed tomography for an artificial object $u$ (\cref{fig:Tika}) and simulated noisy measurement $\fd$ (\cref{fig:Tikb})\footnote{The code for the experiments can be found here \href{https://github.com/TimRoith/LectureNotesIP/tree/main}{github.com/TimRoith/LectureNotesIP}}. We show the results of the reconstruction using the pseudo inverse (\cref{fig:Tikc}) as well as GD applied to \labelcref{eq:GNE_Tik} for different regularization parameters $\alpha\in\ls0, 0.55, 2\rs$. While we can suppress heavy artifacts and instabilities with Tikhonov regularization, the quality of the reconstruction performance depends on the choice of regularization parameter $\regp$. If there is no a-priori parameter choice rule available, one can employ the following a-posteriori rule, which is known as \alert{Morozov's discrepancy principle}. Assuming that we know the noise level $\noiselvl= \norm{\noise} = \norm{Au_{\text{true}} - f^\noiselvl}_2$, it is not meaningful to look for a reconstruction with $\norm{Au^* - f^\noiselvl}_2 < \delta$. This motivates the following definition.
\begin{definition}{Morozov's Discrepancy Principle}{}
Let $u_\alpha$ denote the solution of \labelcref{eq:GNE_Tik}. For a parameter $\mu>0$
Morozov's discrepancy principle is the a-posteriori choice rule defined as
\begin{align*}
\bs{\alpha}(\noiselvl, f^\noiselvl) :=
\sup\ls\alpha > 0 \, \middle| \, \norm{A u_\alpha - f^\noiselvl} \leq \mu\, \noiselvl\rs.
\end{align*}
\end{definition}

\begin{remark}{}{}
In practice, this can be realized via running the optimization multiple times, with different $\alpha$. Alternatively, one can update $\alpha$ adaptively, e.g., during the iteration of gradient descent.
\end{remark}
If we intend to solve a system of linear equations, a method that often outperforms plain gradient descent is the so-called \alert{conjugate gradient} method. Consider the matrix equation
\begin{align}
C u = b
\end{align}
for some symmetric positive definite matrix $C$. The idea is, that we consider the descent directions given by the functional
\begin{align*}
u\mapsto\langle C u, u\rangle + \langle b, u\rangle.
\end{align*}
Instead of using the gradient direction directly, we construct an orthonormal system of descent directions. I.e., at step $k$ we modify the gradient of the above functional, such that it is orthogonal w.r.t. all previous descent directions. In this scenario, we can also compute the optimal step size in each iteration. For a given start point $u^{(0)}$, initializing the so-called \alert{residual} $r^{(0)} = b - C u^{(0)}$ and the first descent direction as $p^{(0)} = r^{(0)}$, we have the following scheme:

\begin{align}
\alpha^\k &= \frac{\langle r^\k, r^\k\rangle}{\langle p^\k, C p^\k\rangle},  \\
u^\kk &= u^\k + \alpha^\k\, p^\k, \\
r^\kk &= r^\k - \alpha^\k\, C p^\k, \\
\beta^\k &= \frac{\langle r^\kk, r^\kk\rangle}{\langle r^\k, r^\k\rangle}, \\
p^\kk &= r^\kk + \beta^\k p^\k.
\end{align}
In our case, we can apply the CG iteration to the regularized normal equations in \labelcref{eq:GNE_Tik}, i.e., for $C=(A^*A + \alpha I)$ and $b=A^*\fd$.

\subsection{Sparsity-Based Regularization}

\begin{figure}[t]
\begin{subfigure}{.24\textwidth}
\includegraphics[width=\textwidth]{fig/sparsity/Deconv_Res_0.png}
\caption{Ground truth $u$.}
\end{subfigure}
\begin{subfigure}{.24\textwidth}
\includegraphics[width=\textwidth]{fig/sparsity/Deconv_Res_1.png}
\caption{Noisy data $\fd$.}
\end{subfigure}
\begin{subfigure}{.24\textwidth}
\includegraphics[width=\textwidth]{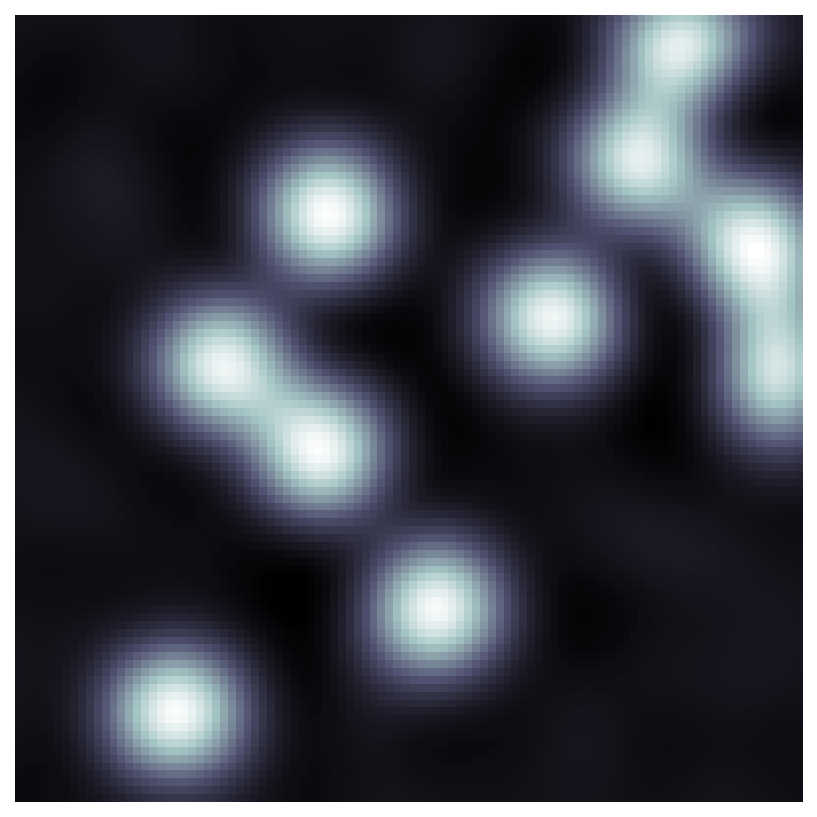}
\caption{$\ell^2$-regularization.}
\end{subfigure}
\begin{subfigure}{.24\textwidth}
\includegraphics[width=\textwidth]{fig/sparsity/Deconv_Res_3.png}
\caption{$\ell^1$-regularization.}
\end{subfigure}
\caption{We compare the reconstruction results of a deconvolution toy example (a)-(b) achieved using GD applied on $\ell^2$-based Tikhonov regularization (c) and PGD applied on $\ell^1$-based sparsity regularization, i.e., ISTA (d).
}\label{fig:ISTA}
\end{figure}

In many applications, we have prior information about our solution. As an introductory example, we consider a situation, where we know that our true image is sparse. This means that the value of the so-called $\ell^0$-pseudonorm
\begin{align}
\norm{u}_0 = \#\{ u_i \neq 0, \, i=1,\ldots, n\},
\end{align}
should be small. Assume that our forward  model is a convolution with a given kernel $\kappa$, i.e., 
\begin{align}
\fd = \kappa \ast u + \noise.
\end{align}
We observe the setup and the performance of $\ell^2$-regularization in \cref{fig:ISTA}\footnote{The code for the experiments can be found here \href{https://github.com/TimRoith/LectureNotesIP/tree/main}{github.com/TimRoith/LectureNotesIP}}. We see that while this regularization removes the noise, the reconstruction performance is not great. This indicates, that we need to choose a better regularization strategy that also incorporates our prior knowledge about the sparsity of $u$. A first idea could be to add an $\ell^0$-regularization term instead. However, this yields an unfavorable optimization strategy. Mainly due to the fact, that it is non-convex.
\begin{exercise}{}{}
Show, that the $\ell^0$-functional as defined above is a pseudonorm, but not a norm. Furthermore, show that it is non-convex.
\end{exercise}\noindent%
A popular alternative is to choose an $\ell^1$-based regularization instead, which has motivated the so-called basis pursuit \cite{chen2001atomic},
\begin{align}
\min_{u}\norm{u}_1 \, \text{ such that } \, Au = f.
\end{align}
In \cref{fig:l1vis} we observe the \alert{sparsity-promoting} property of the $\ell^1$-norm. Therefore, we consider the following regularization objective:
\beq
\min_{u} \frac{1}{2}\norm{A u - f}^2_2 + \regp \norm{u}_1.
\eeq
\begin{memo}{}{}
The $\ell^1$-norm is sparsity promoting.
\end{memo}
\begin{figure}%
\centering
\includegraphics[width=0.5\linewidth]{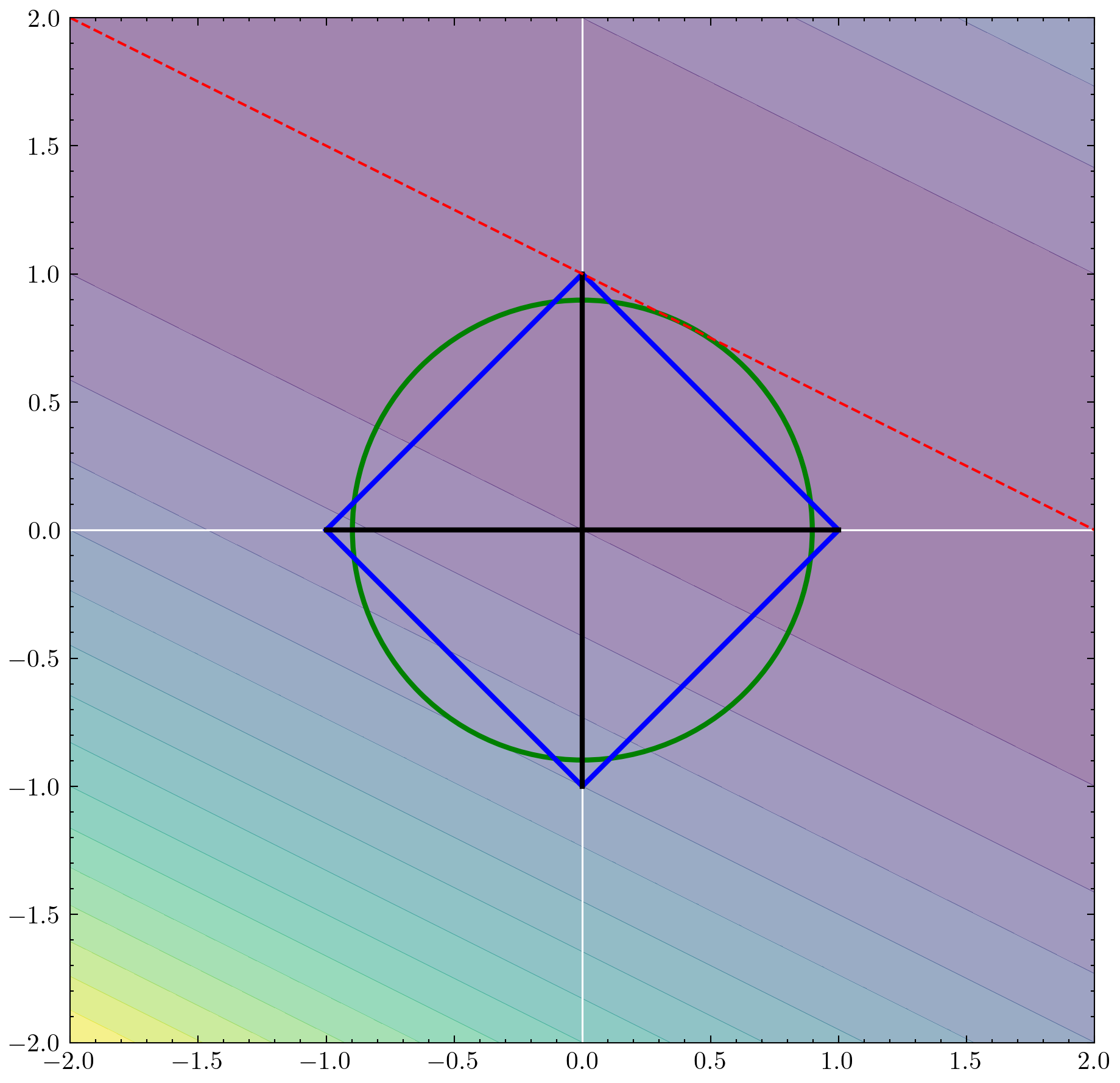}
\caption{Sparsity-promoting property of $\ell^1$-norm. The dotted red line visualizes the space of solutions to a linear problem $Au=f$. The solution with smallest $\ell^1$-norm (top of blue square) is sparse while, in comparison, the solution with the smallest $\ell^2$-norm (top-right of green circle) is not.}
\label{fig:l1vis}
\end{figure}\noindent%
While the $\ell^1$-norm is convex, it is unfortunately not differentiable; thus we can not apply the gradient descent algorithm. Let us go back to the initial motivation of gradient descent to motivate a  more general algorithm which can handle non-differentiability. In the second motivation \cref{it:M2}, we have defined the iterates via a linearization around the current iterate
\begin{align}
\argmin_u \obj(u^{(k)}) + \langle \nabla\obj(u^{k}), u - u^{(k)} \rangle + \frac{1}{2\tau}\norm{u - u^{(k)}}^2.
\end{align}
We want to keep this idea while still making it applicable in the non-differentiable case. To do so, we revert the linearization but keep the distance penalty to the previous iterate, i.e.,
\beq
u^{(k+1)} = \argmin_{u} \obj(u) + \frac{1}{2\tau}\norm{u - u^{(k)}}^2.
\eeq
In convex analysis, the above object is referred to as the \alert{proximal operator} or \alert{proximal point}.
\begin{definition}{Proximal Operator}{}
Given a proper, convex and lower-semicontinuous function $\obj:U\to(-\infty, +\infty]$, and a parameter $\tau> 0$,  we define the proximal operator as 
\beq
\prox_{\tau \obj}(\tilde{u}) := \argmin_{u} \obj(u) + \frac{1}{2\tau}\norm{u - \tilde{u}}^2.
\eeq
The functional $\obj^{\tiny \tilde{u}, \tau}(u)= \obj(u) + \frac{1}{2\tau}\norm{u - \tilde{u}}^2$ is called \alert{Moreau--Yosida regularization}\footnote{\href{https://en.wikipedia.org/wiki/Jean-Jacques_Moreau}{Jean Jacques Moreau} (1923--2014) was a French mathematician and mechanician. 
\href{https://en.wikipedia.org/wiki/K\%C5\%8Dsaku_Yosida}{Kōsaku Yosida} (1909--1990) was a Japanese mathematician.}.
\end{definition}
\begin{lemma}{Properties of the Proximal Operator}{}
Let $\obj$ be a proper, convex, lsc function.
\begin{enumerate}
    \item The proximal operator $\prox_{\tau\obj}$ is single-valued.
    \item The proximal operator is non-expansive, i.e., $1$-Lipschitz continuous.
\end{enumerate}%
\end{lemma}%
\begin{proof}%
Exercise.
\end{proof}%
\noindent%
Using this notion, we define the so-called \alert{proximal point algorithm}.
\begin{definition}{Proximal Point Algorithm}{}
Let $\obj$ be a proper, convex, lsc function and let the initial point $u^{(0)}$ as well as the step size $\tau$ be given. Then the iterates of the proximal point algorithm are defined as
\begin{align}
u^\kk_\PP = \prox_{\tau \obj}(u^\k_\PP).
\end{align}
\end{definition}%
\noindent%
When $\obj$ is differentiable, we can derive an expression for the proximal point using its gradient. In the non-differentiable convex case, we need a generalization of the gradient:
\begin{definition}{Subdifferential}{}
Let $\obj$ be a convex function. The \alert{subdifferential} of $\obj$ at a point $u$ is given as the set
\beq
\partial \obj(u):= 
\{
p\in U^*: \obj(\tilde{u}) - \obj(u) \geq \langle p, \tilde{u} - u\rangle\quad\forall \tilde{u}\in U\}
\eeq
and $p\in\partial \obj(u)$ is called \alert{subgradient} of $\obj$ at $u$.
\end{definition}
\begin{remark}{}{}
Subgradients are elements of the \alert{dual space}. In the finite-dimensional setting, i.e., $U=\R^d$, it holds $U^* = U$. Note that the definition also works in the infinite-dimensional setting; in this case we would have to consider the definition of the dual space more carefully as 
it does not coincide with the primal space in general.
\end{remark}
\begin{lemma}{Properties of the Subdifferential}{lem:propsubdiff}
Let $\obj$ be a proper, convex function.
\begin{enumerate}
\item $\partial \obj(u)$ is a closed convex set for every $u\in U$.
\item If $\obj$ is differentiable at $u$, then the subdifferential at $u$ is single-valued and $\partial\obj(u) = \{\nabla\obj(u)\}$. In this case, we also write $\partial \obj(u) = \nabla \obj(u)$.
\item For scalar $\alpha$, we have $\partial (\alpha \obj) = \alpha\, \partial \obj$.
\item Let $H$ be proper and convex . Then we have
\beq\partial H + \partial J\subset\partial (\obj + H).\eeq
If there exists $u^s\in\dom(\obj)\cap \dom(H)$ such that $\obj(u^s), H(u^s) < +\infty$ and at most one of $\obj$ and $H$ are discontinuous at $u^s$, 
then we have
\begin{align}
\partial (\obj + H) = \partial H + \partial J.
\end{align}
\item For any $\tau>0$ we have
\begin{align}
p\in \partial \obj(u) \Leftrightarrow 
u = \prox_{\tau \obj}(u+\tau p).
\end{align}
\end{enumerate}
\end{lemma}
\begin{proof}
For a proof of the fourth item, we refer to \cite[Thm. 47.B]{zeidler2013nonlinear}. The other statements are part of the exercises.
\end{proof}
For the sparsity promoting regularization we are interested in the subdifferential of the $\ell^1$-norm:
\begin{exercise}{}{}
The subdifferential of the $\ell^1$-norm is given by
\beq
p\in \partial\obj(u) \Leftrightarrow 
p_i \in 
\begin{cases}
\{\sign(u_i)\} &\text{if } u_i\neq 0,\\
[-1,1]&\text{if } u_i=0,
\end{cases}
\qquad \forall i=1,\ldots,d.
\eeq
\end{exercise}\noindent%
Using the notion of the subdifferential, we can generalize the optimality conditions which we have seen before.
\begin{lemma}{Optimality Condition}{}
Let $\obj$ be a proper, convex function, then $u^*$ is a global minimizer, iff
\beq
0\in \partial \obj(u^*).
\eeq
\end{lemma}
\begin{proof}
Exercise.
\end{proof}\noindent%
Using this fact together with the sum rule we obtain the following expression for the proximal point,
\begin{align}
u = \prox_{\tau\obj}(\tilde{u})
&\qquad\Leftrightarrow\qquad
0 \in \partial \lb\obj(u) + \frac{1}{2\tau}\norm{u - \tilde{u}}^2\rb \\
&\qquad\Leftrightarrow\qquad
0 \in \frac{1}{\tau} (u - \tilde{u}) + \partial \obj(u)\\
&\qquad\Leftrightarrow\qquad
u\in (I + \tau \partial \obj)^{-1}(\tilde{u}).
\end{align}
For any set-valued operator $A:U\to2^U$ the operator $(I+ A)^{-1}$ is called the \alert{resolvent}.
\begin{exercise}{}{}
Show that the proximal operator for the following functionals is given as specified below.
\begin{enumerate}
\item For the \alert{indicator} function of a convex set $C$
\beq
\chi_C(u) := \begin{cases}
    0 &\text{if } u \in C\\ +\infty &\text{else},
\end{cases}
\eeq
the proximal point is given by
\beq
\prox_{\chi_C}(u) = \Pi_C(u),
\eeq
where $\Pi$ denotes the projection operator.
\item For the $\ell^1$-norm the proximal point is given as the \alert{soft shrinkage} operator
\beq
\prox_{\tau \norm{\cdot}_1}(u)_i = 
\operatorname{shrink}_\tau(u)_i = 
\begin{cases}
u_i - \tau &\text{if } u_i> \tau,\\
u_i + \tau &\text{if } u_i< \tau,\\
0 &\text{else}.
\end{cases}
\eeq
See also \cref{fig:shrink}.
\end{enumerate}
\end{exercise}
\begin{figure}
\centering
\includegraphics[width=0.5\linewidth]{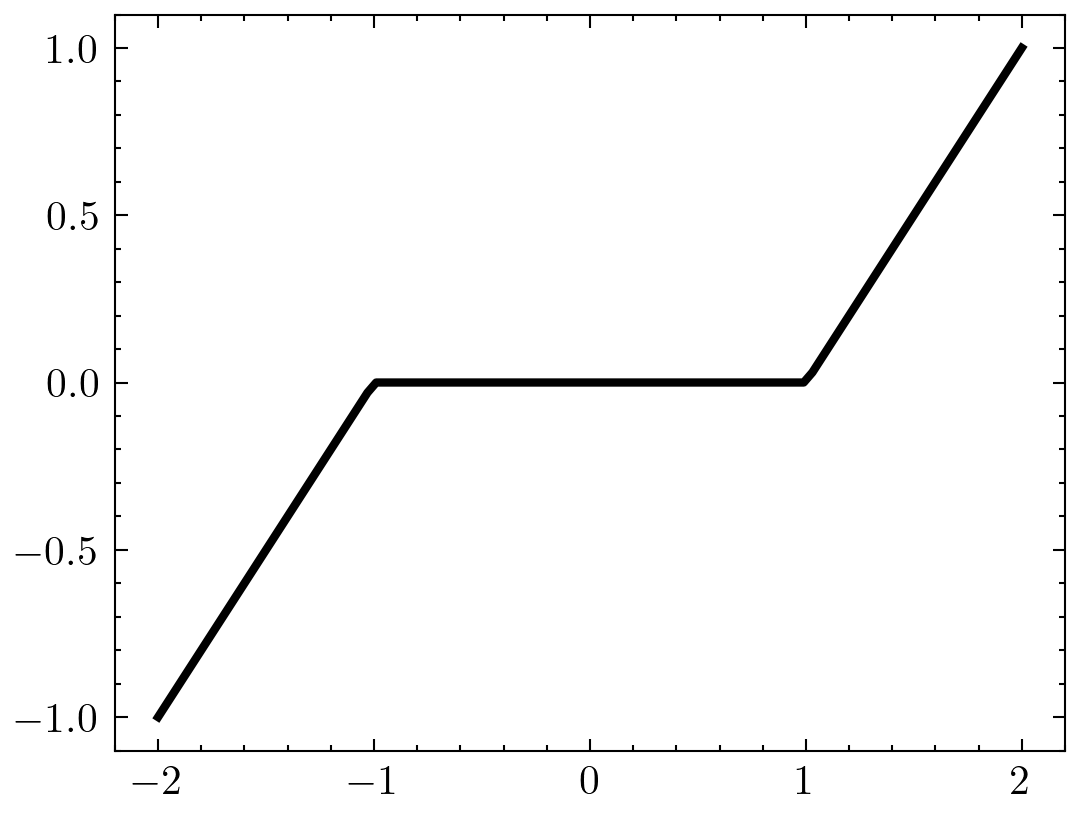}
\caption{The shrinkage operator for $\tau=1$.}
\label{fig:shrink}
\end{figure}
In the following, we give a convergence result for the proximal point algorithm.
\begin{theorem}{Convergence of the Proximal Point Method}{}
Let $\obj$ be a proper, convex, lsc functional with global minima in the set $S\neq\emptyset$.
\begin{enumerate}
\item For any $u^*\in S$ we have that 
\beq
\obj(u^\k_\PP) - \obj(u^*) \leq \frac{\norm{u^{(0)} - u^*}^2}{2\tau k}.
\eeq
\item If additionally $\obj$ is $\nu$-strongly convex, we have
\beq
\norm{u^\k - u^*} = (1+\nu \tau)^k \norm{u^{(0)} - u^*}.
\eeq
\end{enumerate}
\end{theorem}
\begin{proof}
\textbf{Ad 1.}: Directly from the definition, we have that
\beq
\frac{1}{2\tau}\norm{u^\kk_\PP - u^\k_\PP}^2 + \obj(u_\PP^\kk) \leq \frac{1}{2\tau}\norm{u_\PP^\k - u_\PP^\k}^2 + \obj(u^\k_\PP) = \obj(u^\k_\PP),
\eeq
which yields that $J(u^\k)$ is a non-increasing sequence. By the optimality condition we know that
\beq
0 \in \partial \obj^{u^\k_\PP,\tau}(u^\kk_\PP) 
\Leftrightarrow
-\frac{u^\kk_\PP - u^\k_\PP}{\tau} \in \partial \obj(u^\kk_\PP)
\eeq
and, therefore, using the definition of the subgradient, we get
\beq
\obj(u) \geq \obj(u^\kk_\PP) - 
\left\langle
\frac{u^\kk_\PP - u^\k_\PP}{\tau},
u - u^\kk_\PP
\right\rangle,\qquad \forall u \in U.
\eeq
This yields
\begin{align}
2\tau(\obj(u^\kk_\PP) - \obj(u))
&\leq 
2\langle
u^\kk_\PP - u^\k_\PP,
u - u^\kk_\PP
\rangle
\\
&=
\norm{u - u^\k_\PP}^2 - \norm{u - u^\kk_\PP}^2 - \norm{u^\kk_\PP - u^\k_\PP}^2\\ &\leq
\norm{u - u^\k_\PP}^2 - \norm{u - u^\kk_\PP}^2.
\end{align}
Summing the above inequality up to step $K$ we arrive at
\begin{align}
2\tau \sum_{k=0}^K (\obj(u^\kk_\PP) - \obj(u)) \leq \norm{u - u^{(0)}}^2
\end{align}
and using that $\obj(u^\k)$ is non-increasing we obtain
\beq
\obj(u^K_\PP) - \obj(u) \leq 
\frac{\norm{u - u^{(0)}}^2}{2\tau K}.
\eeq
Since $u$ was arbitrary, the statement follows.\\
\textbf{Ad 2.}: Let $u^*$ be the global minimizer with $0\in\partial \obj(u^*)$. From the strong convexity and the fact that $1/\tau(u^\kk_\PP-u^\k_\PP)\in\partial J(u^\kk_\PP$ we have that
\begin{align}\label{eq:strconv}
\langle 
u^\kk_\PP-u^\k_\PP - 0,
u^\kk_\PP - u^*
\rangle
\geq
\nu \tau 
\norm{u^\kk_\PP - u^*}^2,
\end{align}
which directly implies
\beq
\norm{u^\kk_\PP-u^\k_\PP} \geq 
\nu \tau 
\norm{u^\kk_\PP - u^*}.
\eeq
Using that $2\langle u, \tilde{u}\rangle = \norm{u+\tilde{u}}^2 - \norm{u}^2 - \norm{\tilde{u}}^2$, we obtain from \cref{eq:strconv}
\begin{align}
\norm{u^* - u^\k_\PP}^2  &\geq (1+ 2\nu \tau)\norm{u^\kk_\PP - u^*}^2
+
\norm{u^\kk_\PP-u^\k_\PP}^2\\
&\geq
(1+ 2\nu \tau + (\nu\tau)^2)\norm{u^\kk_\PP - u^*}^2\\
&=
(1+\nu \tau)^2\norm{u^\kk_\PP - u^*}^2,
\end{align}
which yields the desired inequality by induction.
\end{proof}\noindent%
Having introduced the abstract framework of convex optimization methods, we can go back to our actual problem, namely solving
\beq
\argmin_u \ls \frac{1}{2}\norm{Au - f}^2_2 + \regp \norm{u}_1 \rs.
\eeq
A possible approach would be to directly apply the proximal point method to the above functional. However, there is no guarantee that we can easily compute the prox of the sum. A better approach is to use so-called \alert{splitting} techniques. Assume that our objective is given as the sum
\begin{align*}
\obj(u) = H(u) + G(u),
\end{align*}
where $H$ is differentiable and $G$ is convex. We now present two strategies to derive the same splitting method, that yields so-called \alert{proximal gradient descent}:
\begin{enumerate}
\item While the prox of $H+G$ might not be directly computable, we can replace $H$ by a linearization around the current iterate $u\mapsto \langle \nabla H(u^\k), u\rangle$. Then we can compute
\begin{align}
&\argmin_{u} \ls H(u) + G(u) + \frac{1}{2\tau} \norm{u - u^\k}^2_2 \rs \\
\approx
&\argmin_{u} \ls G(u) + \ll \nabla H\lb u^\k\rb, u\rr + 
\frac{1}{2\tau} \norm{u}^2 - \frac{1}{\tau} \ll u, u^\k \rr 
+
\frac{1}{2\tau} \norm{u^\k}^2\rs \\
=
&\argmin_{u} \ls G(u) + 
\frac{1}{2\tau} \norm{u}^2 - \frac{1}{\tau} \ll u, u^\k -\tau \nabla H\lb u^\k\rb\rr 
+
\frac{1}{2\tau} \norm{u^\k}^2\rs \\
=&\argmin_{u} \ls G(u) + 
\frac{1}{2\tau} \norm{u}^2 
- 
\frac{1}{\tau} \ll u, u^\k -\tau \nabla H\lb u^\k\rb\rr 
+
\frac{1}{2\tau} \norm{u^\k -\tau \nabla H\lb u^\k\rb}^2\rs \\
=
&\argmin_{u} \ls G(u) + 
\frac{1}{2\tau} \norm{u}^2 
- 
\frac{1}{\tau} \ll u, u^\k -\tau \nabla H\lb u^\k\rb\rr 
+
\frac{1}{2\tau} \norm{u^\k}^2\rs \\
=&\argmin_{u} \ls G(u) + 
\frac{1}{2\tau} \norm{u - \lb u^\k -\tau \nabla H\lb u^\k\rb\rb}^2\rs \\
=&
\prox_{\tau G} (u^\k -\tau \nabla H(u^\k)).
\end{align}
\item We search for $u^*$ such that $0\in\partial \obj(u^*).$ Using the sum rule we can instead search for $u^*$ and $p^*$ such that
\begin{align}
p^*&\in \partial H(u^*)\qquad\text{and}\qquad -p^*\in\partial G(u^*)\\
\Leftrightarrow
p^*&\in \partial H(u^*)\qquad\text{and}\qquad u^* = \prox_{\tau G}(u^* - \tau p^*),
\end{align}
where we used \cref{lem:propsubdiff}. This yields the so-called \alert{forward-backward splitting} method. Note, that this derivation works also for non-differentiable $H$, where in each step we would have to choose a subgradient. Unfortunately, choosing appropriate subgradients can be computationally difficult and, therefore, we restrict ourselves to cases where $H$ is differentiable.
\end{enumerate}
\begin{definition}{Proximal Gradient Descent}{}
Given a differentiable functional $H$, a proper, convex, lsc functional $G$, a step size $\tau$ and an initial iterate $u^{(0)}_{\PGD}$, the iterates of proximal gradient descent are defined as
\beq
u^\kk_\PGD = \prox_{\tau G}\lb u_\PGD^\k - \tau \nabla H\lb u^\k_\PGD\rb\rb.
\eeq
\end{definition}
\begin{theorem}{Convergence of PGD}{thm:convpgd}
Let $H,G$ be proper, convex and lsc, and let $H$ be differentiable with $L$-Lipschitz continuous gradient. Then for a step size $\tau \leq 1/L$ and any global minimizer $u^*$ we have that
\begin{align}
\obj(u^k_\PGD) - \obj(u^*) \leq
\frac{\norm{u^{(0)} - u^*}^2}{2 \tau\, k}.
\end{align}
\end{theorem}
%
%
%
In our special case, where $H(u) = \frac{1}{2}\norm{Au -f}^2_2$ and $G(u)=\regp \norm{u}_1$, we obtain the so-called \textbf{I}terated \textbf{S}oft \textbf{T}hresholding \textbf{A}lgorithm (ISTA), see e.g. \cite{daubechies2004iterative,chambolle1998nonlinear}, where the iterates are given by
\beq
u^\kk_\ISTA = \operatorname{shrink}_{\tau\regp}(u^\k_\ISTA - \tau A^T(Au_\ISTA^\k - f)).
\eeq
In \cref{fig:ISTA} we use a deconvolution toy example to demonstrate the effect the choice of regularizer and optimization approach have on the result. While $\ell^2$-regularization leads to a noise-free but blurry result, $\ell^1$-regularization returns a noise-free but also much sparser result. 
\begin{remark}{}{}
The accelerated version of ISTA is called Fast ISTA or FISTA, see \cite{beck2009fast}.
\end{remark}
If it is reasonable to assume that our solution is sparse, applying ISTA directly can have a desirable regularization effect. However, in many scenarios, this prior belief might fail. If we consider the situation of the phantom in \cref{fig:Tik}, a pure sparsity prior does not seem sensible. However, often we can find another representation of our object in which it is sparse. One classical example for images are so-called \alert{wavelets}, which we briefly remark upon. We refer to \cite{daubechies1992ten} for a full introduction into the topic. From a high-level perspective, we consider a linear and orthogonal matrix $W$, such that we can make the assumption of $Wu$ being sparse. This yields the optimization problem
\begin{align}
\argmin_{u} \ls \frac{1}{2}\norm{Au - \fd}^2 + \regp \norm{Wu}_1 \rs.
\end{align}
In order to apply proximal gradient descent, we need to compute the prox of the regularizer, which is given as
\begin{align}
\prox_{\norm{W\cdot}_1}(\tilde{u}) 
&= 
\argmin_u \ls \frac{1}{2}\norm{u - \tilde{u}}^2 + \norm{Wu}_1 \rs
= W^T\left(\argmin_z \frac{1}{2}\norm{W^T z - \tilde{u}}^2 + \norm{z}_1\right)\\
&=
W^T\left(\argmin_z \ls \frac{1}{2}\norm{z - W\tilde{u}}^2 + \norm{z}_1 \rs \right)
=W^T\left( \operatorname{shrink}_1(W\tilde{u})\right),
\end{align}
where we used the substitution $z=Wu$. Equivalently, we can substitute the variable $z=Wu$ in the problem and consider
\begin{align}
W^T\left(\argmin_{z} \ls \frac{1}{2}\norm{AW^T z - \fd}^2 + \regp \norm{z}_1 \rs \right).
\end{align}
In this form, we can directly apply the ISTA algorithm. In \cref{fig:wavelet-tv}, we can observe the differences between plain $\ell^2$- and wavelet- regularization. The wavelet used in this example, belongs to the family of Daubechies wavelets\footnote{\href{https://en.wikipedia.org/wiki/Ingrid_Daubechies}{Ingrid Daubechies} (born in 1954) is a Belgian-American physicist and mathematician.}.
\begin{figure}
\begin{subfigure}{.32\textwidth}
\includegraphics[width=\textwidth]{fig/sparsity/Sino_Res_0-55_3.png}
\caption{$\ell^2$-regularization}
\end{subfigure}%
\begin{subfigure}{.32\textwidth}
\includegraphics[width=\textwidth]{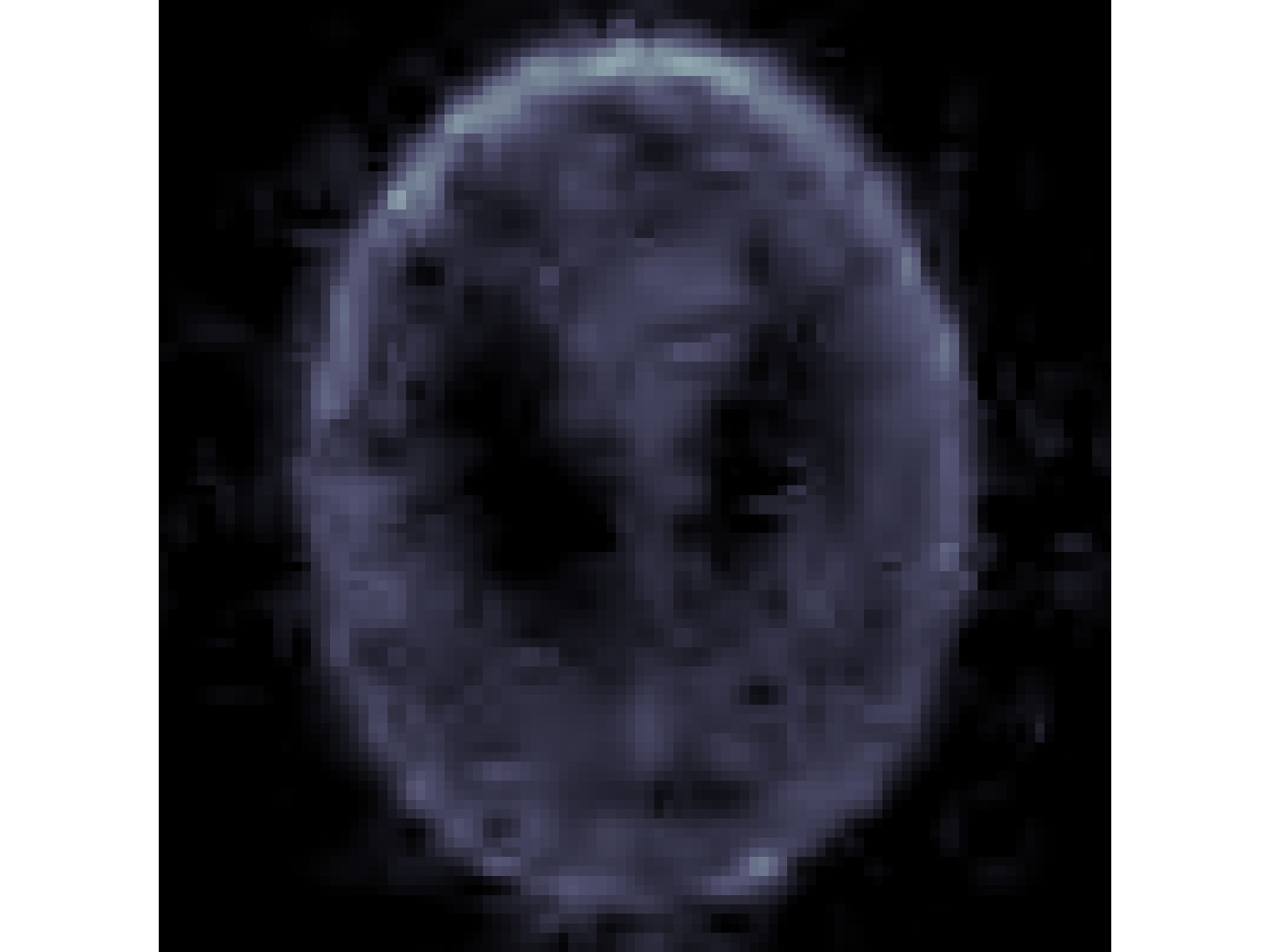}
\caption{wavelet-regularization}
\end{subfigure}%
\begin{subfigure}{.32\textwidth}
\includegraphics[width=\textwidth]{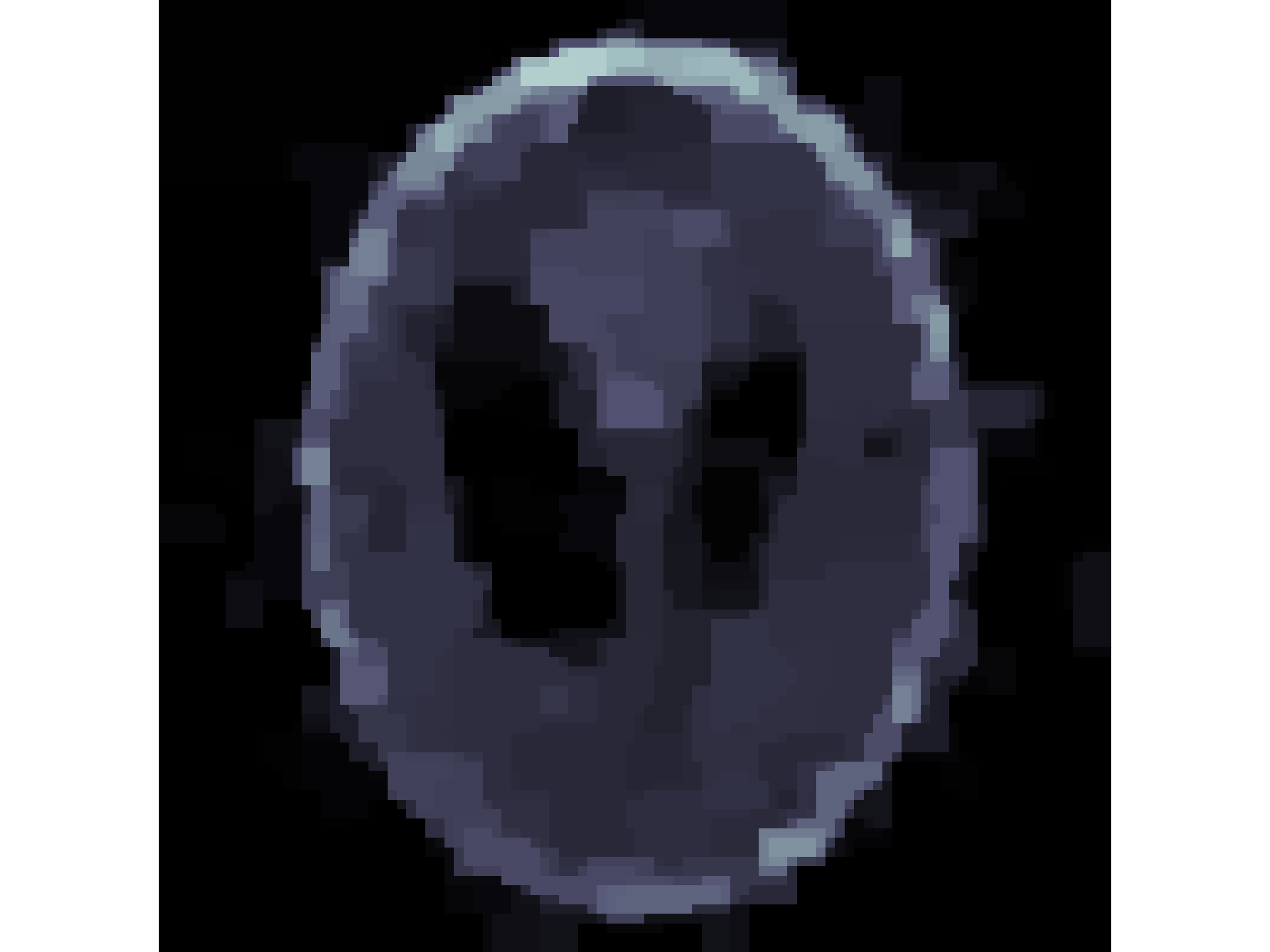}
\caption{TV-regularization}
\end{subfigure}%
\caption{We compare the reconstruction results of a computed tomography toy example \cref{fig:Tika}-\cref{fig:Tikb} achieved using GD applied on $\ell^2$-based Tikhonov regularization (a), PGD applied on $\ell^1$-based wavelet regularization (b) using wavelets belonging to the family of Daubechies wavelets and Chambolle--Pock applied to TV-regularization (c). }
\label{fig:wavelet-tv}
\end{figure}
\subsection{Total variation and the ROF model}
In the previous section, we have explored sparsity based regularization. In situations in which we could not assume sparsity of the image directly, we have considered sparsity in a different representation using the concepts of wavelets. We now want to further follow this approach and to consider a sparsity concept which is especially suited for images. Note that images are composed of different features: relative constant areas for different objects and surfaces and jumps around the edges. To construct a regularizer based on this insight, we introduce the gradient of an image $u\in\R^{m,n}$, 
\begin{gather}
\nabla: \R^{m,n} \to \R^{2,m,n}\\
(\nabla u)_{1, i, j} := u_{i+1,j} - u_{i, j}\quad\text{ for } 1\leq i < m, \quad (\nabla u)_{1, m, j} = 0\\ 
(\nabla u)_{2, i, j} := u_{i,j+1} - u_{i, j}\quad\text{ for } 1\leq j < n, \quad (\nabla u)_{2, i, n} = 0.
\end{gather}
%
%
Assuming piecewise constant areas in an image is equivalent to assuming sparsity in the image gradient. This leads us to the following regularization functional.
\begin{definition}{Total Variation}{}
For an image $u\in\R^{m,n}$ the discrete \alert{total variation} is defined as
\begin{align*}
\TV(u) := \norm{\nabla u}_1.
\end{align*}
\end{definition}
\begin{remark}{}{}
The TV functional is not only studied in the discrete setting. There exists a rich mathematical study, involving this functional applied to functions $u$ living in an infinite-dimensional space, see for example \cite{acar1994analysis}.
\end{remark}\noindent%
Employing this regularizer leads us to \alert{total variation} denoising, or for $A=Id$ the so-called \alert{Rudin--Osher--Fatemi (ROF)}\footnote{%
\href{}{Leonid Rudin} is an American computer scientist.
\href{https://en.wikipedia.org/wiki/Stanley_Osher}{Stanley Osher} (born in 1942) is an American mathematician.

} %
model \cite{rudin1992nonlinear},
\begin{align}\label{eq:ROF}\tag{ROF}
\argmin_u \frac{1}{2} \norm{A u - \fd}^2_2 
+ \regp \TV(u).
\end{align}
It turns out that computing the prox of the TV functional is a non-trivial task, and we do not have a closed-form expression. Therefore, we can not apply the same reasoning as for the wavelet case, to apply ISTA directly. In the following, we want to deduce an algorithm to solve this problem. To do so, we start with the abstract formulation 
\begin{align}
\argmin_u \ls G(u) + H(Au) \rs,
\label{eq:pproblem}\tag{P}
\end{align}
where $G,H$ are proper, convex functions, and $A$ is a linear operator.

To obtain a suitable algorithm for the solution of the above problem, we introduce the concept of \alert{duality}, which can be traced back to \cite{fenchel2013conjugate}\footnote{\alert{Werner Fenchel} (1905--1988) was a German-Danish mathematician.}.
\begin{definition}{Fenchel Conjugate}{}
Given a function $\obj:U\to\Rc$ defined on a topological vector space $U$, its \alert{convex conjugate} or also \alert{Fenchel conjugate}  $\obj^*:U^*\to\Rc$ is defined as
\begin{align*}
\obj^*(p) := 
\sup_{u} \ls \langle p, u \rangle - \obj(u) \rs.
\end{align*}
\comment{
Let $\obj:U\to(-\infty, +\infty]$ be a proper and convex functional, then we denote by
\begin{align*}
\obj^*(p) := 
\sup_{u} \langle p, u \rangle - \obj(u)
\end{align*}
the so-called \alert{Fenchel conjugate}.}
\end{definition}

\begin{example}{Fenchel Conjugates}{}
\begin{enumerate}
\item 
For any norm $\obj(u) = \norm{u}$ the Fenchel conjugate is a indicator function of the respective dual norm, i.e.,
\begin{align*}
\obj^* = \chi_{\norm{\cdot}_*\leq 1}.
\end{align*}
\item In case of the squared $\ell^2$-norm $\obj = \frac{1}{2}\norm{\cdot}^2$, it holds that
\begin{align*}
\obj^* = \obj.
\end{align*}
\end{enumerate}
\end{example}
In the following we list some important and well-known properties of the conjugate. The proof is left as an exercise.
\begin{lemma}{Properties of the Conjugate}{}
\begin{enumerate}
\item For any function $\obj:U\to(-\infty,\infty]$, the convex conjugate is convex.
\item It holds that $\obj(u) + \obj^*(p) \geq \langle p, u\rangle$ for any $p\in U^*, u \in U$.
\item For scalar $\regp>0$, it holds that $(\regp \obj)^*(p) = \regp\obj(p/\regp)$.
\item For a vector $\tilde{u}\in U$ we have that $[u\mapsto \obj(u-\tilde{u})]^*(p) = \obj^*(p) + \langle p, \tilde{u}\rangle.$
\item We have that $p\in\partial \obj(u)$ is equivalent to $\obj(u) + \obj^*(p) = \langle p, u \rangle$.
\end{enumerate}
\end{lemma}
\begin{proof}
See exercise.
\end{proof}
A natural question is how multiple applications of convex conjugation affect the function. The function $(\obj^*)^* = \obj^{**}$ is called the \alert{bi-conjugate} of $\obj$. By definition we obtain the inequality
\begin{align*}
J^{**}(u)\leq J(u).
\end{align*}
The following theorem, which is due to Moreau and Fenchel, shows that equality is attained iff the function is proper, convex and lsc. 
\begin{theorem}{Fenchel--Moreau}{}
A function $\obj$ is proper, convex and lsc iff $\obj=\obj^{**}$.
\end{theorem}
\begin{proof}
See \cite[Sec. 12]{rockafellar2015convex}.
\end{proof}
Another important concept is \alert{convex inversion}: It tells us that we can interpret the subgradient of the convex conjugate as the inverse of the gradient of the function $\obj$.%
\begin{theorem}{Convex Inversion}{thm:convinv}
Let $\obj$ be proper, convex and lsc, then we have that
\begin{align*}
p \in \partial \obj(u) \Leftrightarrow
u\in \partial\obj^*(p)
\end{align*}
\end{theorem}
\begin{proof}
See \cite[Thm. 23.5]{rockafellar2015convex}.
\end{proof}
With this terminology, we can now introduce the concept of primal and dual problems. We start with the problem in \labelcref{eq:pproblem} and first notice
\begin{align}
\inf_u \ls G(u) + H(Au) \rs 
&= 
\inf_u \ls G(u) + H^{**}(Au) \rs = 
\inf_u \ls G(u) + \sup_p \ls -H^*(p) + \langle p, Au \rangle \rs \rs\\
&= 
\inf_u\sup_p \ls G(u) -  H^*(p) + \langle p, Au \rangle \rs. \label{eq:pproblemproblem}\tag{PD}
\end{align}
The latter is a \alert{primal-dual} formulation of the original problem: we have dualized the term $H(Au)$ and kept the first term unchanged in its primal form. Assuming that we can interchange the inf and sup operation, we further compute
\begin{align}
\sup_p \inf_u \ls G(u) -  H^*(p) + \langle p, Au \rangle \rs &
= 
\sup_p \ls -H^*(p) + \inf_u \ls G(u) + \langle A^* p, u \rangle \rs\rs
\\&= 
\sup_p \ls -H^*(p) - \sup_u \ls\langle -A^* p, u \rangle -   G(u) \rs\rs
\\&=
\sup_p \ls -H^*(p) - G^*(-A^* p) \rs \label{eq:dproblem}\tag{D}
\end{align}
which is called the \alert{dual} problem, since both terms have now been dualized. If both $G$ and $H$ are convex, one can ensure that the primal and dual problems correspond to each other, see, e.g., \cite[Theorem 4.15]{Brandt22} or \cite[Part II]{ekeland1999convex}. The following theorem gives us an equivalent characterization of the solution of the primal and dual problems.
\begin{theorem}{Fenchel--Rockafellar}{}
Let $G,H$ be proper, convex and lsc, then the following are equivalent for points $u^*, p^*$:
\begin{enumerate}
\item The points $u^*, p^*$ respectively solve the primal \labelcref{eq:pproblem} and dual \labelcref{eq:dproblem} problems.
\item It holds that
\beq
-A^* p^* \in \partial G(u^*) \qquad p^*\in \partial H(Au^*).
\eeq
\end{enumerate}
\end{theorem}
\begin{proof}
See \cite[Sec. 31]{rockafellar2015convex}.
\end{proof}

\noindent%
The intuition behind the second identity is that under the above assumptions, we have that%
\begin{align}
\partial (u\mapsto H(Au)) = 
A^*\partial H(u).
\end{align}
Using the theorem of convex inversion, this condition can be reformulated to
\beq
Au^*\in\partial H^*(-p^*).
\eeq
Using \cref{lem:propsubdiff} this is equivalent to 
\begin{align}
u^* &= \prox_{\tau G}(u^* - \tau A^* p^*),\\
p^* &= \prox_{\sigma H^*}(p^* +\sigma A u).
\end{align}
This leads us to the so called \alert{Chambolle--Pock} \cite{chambolle2011first} or \alert{primal-dual hybrid gradient} (PDHG) method.
\begin{definition}{Chambolle-Pock (CP)}{}
Given proper and convex functions $G, H$, a matrix $A$, a scalar $\theta$ and step sizes $\tau,\sigma > 0$. The iterates of the Chambolle--Pock method are given as
\begin{align}
u_\CP^{(k+1)} &= \prox_{\tau G}(u_\CP^{(k)} - \tau A^* p_\CP^{(k)})\\
v_\CP^{(k+1)} &= u_\CP^{(k+1)} + \theta (u_\CP^{(k+1)} - u_\CP^{(k)})\\
p^{(k+1)}_\CP &=
\prox_{\sigma H^*}(p_\CP^{(k)} +\sigma A v_\CP^{(k+1)}).
\end{align}
\end{definition}
Above, $\theta$ is referred to as overrelaxation parameter and is often chosen as $\theta=1$. For the operators
\beq
M = 
\begin{pmatrix}
\tau^{-1} Id & -A^*\\
A & \sigma^{-1} Id
\end{pmatrix}
\qquad
T = 
\begin{pmatrix}
\partial F & A^* \\
-A &\partial G^*
\end{pmatrix}
\eeq
we can write the update in terms of the variable $v_\CP^{(k)}$ as
\beq
v^{(k+1)}_\CP = (M+T)^{-1} M v^{(k)}_\CP.
\eeq
\begin{lemma}{}{}
The operator $M$ is bounded and self-adjoint, furthermore, if $\tau\sigma\norm{A}^2 <1$ we know that $M$ is bounded and invertible.
\end{lemma}
\begin{proof}
See \cite[Lemma 5.32]{Brandt22}.
\end{proof}
\noindent%
This allows us to write $M=(M^{-1})^{-1}$ and thus
\beq
v^{(k+1)}_\CP = (Id+M^{-1}T)^{-1}v^{(k)}_\CP = 
\prox_{M^{-1} T} (v^{(k)}_\CP).
\eeq
Thus the Chambolle--Pock algorithm is a proximal point method. So far, we have only considered proximal mappings for convex functions, now we explicitly deal with the resolvent of an operator. Here, we need the notion of a \alert{maximal monotone} operator. For a set valued operator $M:U\to2^U$, sometimes denoted as $M:U\rightrightarrows U$ we denote the graph of $M$ as 
\beq
\operatorname{graph}(M):=\{(u,v): u\in U, v\in Mu\}.
\eeq%
\begin{definition}{}{Maximal monotone operator}
\begin{enumerate}
\item An operator $M$ is called \alert{monotone} if 
\begin{align*}
\langle v - \tilde{v}, u - \tilde{u}\rangle\geq 0  \qquad\forall 
(u,v),(\tilde{u},\tilde{v})\in\operatorname{graph}(M).
\end{align*}
\item An operator is called \alert{maximal monotone} if for $u,v\in U$
\begin{align*}
\langle v - \tilde{v}, u - \tilde{u}\rangle\geq 0 \qquad\forall(\tilde{u},\tilde{v})\in\operatorname{graph}(M)
\end{align*}
implies that $(u,v)\in \operatorname{graph}(M)$.
\end{enumerate}
\end{definition}\noindent%
Maximal monotone operators can be characterized via the following theorem, which was proposed by Minty\footnote{\href{https://en.wikipedia.org/wiki/George_J._Minty}{George James Minty Jr.} (1929--1986) was an American mathematician.}.
\begin{theorem}{Minty's Theorem}{thm:minty}
A monotone operator $M:U\to 2^U$ is \alert{maximal monotone} iff $I + M$ is surjective. 
\end{theorem}
\begin{proof}
See, e.g., \cite[Thm. 21.1]{Bauschke2017}.
\end{proof}
\noindent%
An important example of a maximal monotone operator is the subdifferential.
\begin{lemma}{}{}
The subdifferential of a proper, convex, lsc function is maximal monotone.
\end{lemma}
\begin{proof}
From the definition of the subdifferential we see that it is a monotone operator. From \cref{thm:convinv} and the fifth item in \cref{lem:propsubdiff} we have that $p-u\in\partial J(u) \Leftrightarrow u=\prox_{J}(u + (p-u)) = \prox_J(p)$. Since the prox operator is defined for every $p\in U$ we obtain that for every $p$ there exists $u$ such that $p\in u + \partial J(u)$.
\end{proof}
The proximal operator was defined as the resolvent of the subdifferential. Minty's theorem allows us to define the resolvent for general maximal monotone operators.
\begin{lemma}{}{}
For a maximal monotone operator $M$, its resolvent $(I + M)^{-1}$ is well-defined and single-valued.
\end{lemma}
\begin{proof}
Follows from Minty's theorem \cref{thm:minty}.
\end{proof}
In order to apply the convergence theorem of the proximal point methods, we need to show that $M^{-1} T$ is maximal monotone. This is provided by the following lemma.
\begin{lemma}{}{}
If $\tau\sigma\norm{A}^2 <1$ we have that $M^{-1}T$ is maximal monotone.
\end{lemma}
\begin{proof}
See \cite[Lem. 5.35]{Brandt22}.
\end{proof}
In order to apply the primal-dual algorithm to the TV denoising problem, we must choose the roles of $H,G$ and $A$. One such choice is 
\beq
G = 0\qquad 
A = 
\begin{pmatrix}
\tilde{A}\\ \nabla
\end{pmatrix}
\qquad 
H(u_1, u_2) = 
\frac{1}{2}\norm{u_1 - \fd}_2^2 + 
\regp \norm{u_2}_1.
\eeq
In order to compute the conjugate of the $H$, we use that for \alert{separable} functions we know that
\begin{align*}
H^*(z_1, z_2) = 
\left(\frac{1}{2}\norm{\cdot-\fd}_2^2\right)^*(z_1) + 
(\regp \norm{\cdot}_1)^* (z_2) = 
\frac{1}{2}\norm{z_1 }_2^2 + \ll \fd, z_1 \rr + \regp \chi_{\norm{\cdot}_\infty\leq \regp}(z_2).
\end{align*}
Furthermore, the prox for separable functions can be computed as
\begin{align*}
\prox_{\sigma H^*}(z_1, z_2) = 
\begin{pmatrix}
\prox_{\sigma  \lb \frac{1}{2} \norm{\cdot }_2^2 + \ll \fd, \cdot \rr \rb}(z_1)\\
\prox_{\sigma\alpha\chi_{\norm{\cdot}_\infty\leq \regp}}(z_2)
\end{pmatrix}
=
\begin{pmatrix}
(z_1-\sigma\fd)/(\sigma+1)\\
\operatorname{Proj}_{B_\alpha^\infty}(z_2).
\end{pmatrix}
\end{align*}
This yields an implementable algorithm.
\paragraph{ADMM} A further variant to solve the TV denoising problem is the \alert{Alternating Directions Method of Multipliers} (ADMM), see e.g. \cite{parikh2014proximal,boyd2011distributed}. The idea here is to reformulate the initial problem as
\begin{align*}
\min_{u, v} \ls G(u) + H(v) \rs\quad \text{such that}\quad Au - v = 0.
\end{align*}
The \alert{augmented Lagrangian} then adds the constraint to the functional in the following way
\begin{align*}
L(u,v, p) :=
G(u) + H(v) + 
\langle p, Au - v\rangle 
+
\frac{\mu}{2}
\norm{Au - v}^2_2.
\end{align*}
In each step, The ADMM algorithm then minimizes the function $L$ in one of its variables. The minimum in $p$ can be expressed in closed form.
\begin{definition}{ADMM}{}
The iterates of the ADMM algorithm are given as
\begin{align}
u^{(k+1)}_\ADMM &= 
\argmin_{u} \ls G(u) + \frac{\mu}{2}
\norm{Au - v^{(k)}_\ADMM + q_\ADMM^{(k)}}_2^2 \rs\\
v^{(k+1)}_\ADMM &= 
\argmin_{v} \ls F(v) + \frac{\mu}{2}
\norm{Au^{(k+1)}_\ADMM - v + q_\ADMM^{(k)}}_2^2 \rs\\
q_\ADMM^{(k+1)} &= 
q_\ADMM^{(k)} + Au^{(k+1)}_\ADMM - v^{(k+1)}_\ADMM.
\end{align}
\end{definition}
\chapter{Learning for Inverse Problems}

So far, we have considered the classical setting of regularization theory. The distinctive feature of the methods which we have considered so far is that the regularizers are \enquote{hand-crafted}. This means that some intuition, a calculation, or some other idea has lead us to writing down the concrete form of the regularization functional. In this chapter, we assume that there is some distribution $\mathbb{P}_U$ which describes how our desired quantities are distributed. Assuming that this distribution is available to us in a certain form, we want to incorporate this knowledge into our regularizer. The concepts of \alert{learning} will provide the tools we need for that.

\section{Introduction to Machine Learning}
We introduce the basic concepts of machine learning. Parts of this exposition are taken from \cite{roith2024consistency}. In the following, our goal is to \enquote{learn} a function $\net:\Inp\to\Oup$ from given data. We begin with two examples that introduce typical learning scenarios.
\begin{itemize}
\item \textbf{Classification}: The function $\net$ assigns to each $\inp\in 
\Inp$ a label $\oup$ out of $C\in\N$ possible classes, i.e., $\Oup=\{1,\ldots,C\}$. Typically, as an intermediate step, one infers a probability vector 
\begin{align*}
\bs{u}\in \Delta^C := \left\{\bs{u}\in[0,1]^C: \sum_{i=1}^C \bs{u}_i = 1\right\}.
\end{align*}
This can be enforced via the softmax function $\operatorname{soft max}:\R^C\to\R^C$
\begin{align}\label{eq:softmax}\tag{Softmax}
\operatorname{soft max}(z)_i := \frac{\exp{z_i}}{\sum_{j=1}^C \exp{z_j}}.
\end{align}
This allows the interpretation that the $i$th entry of $\net(\inp)\in\Delta^C$ models the probability that $\inp$ belongs to class $i$. In order to obtain a label, one can simply choose the maximum entry, i.e.  $\argmax_{i=1,\ldots,C} \net(\inp)_i$.
\item \textbf{Image denoising}: The function $\net$ outputs a denoised version of an input image. Here we have, e.g., $\Inp = \Oup=\R^{d_c\times d_h\times d_w}$, where
%
$d_c\in\N$ is the number of color channels,
and $d_h,d_w\in\N$ denote the height and width of the image.
\end{itemize}

In order to \alert{learn} a function $\net:\Inp\to\Oup$, we typically need so-called training data. Here, we want to introduce three possible scenarios:
\begin{enumerate}
\item \alert{Unsupervised learning}: As training data, we are given a set of inputs $\mathcal{T}=\{\inp_1, \ldots, \inp_T\} \subset \Inp$. The goal here can be to find patterns in the data (clustering) or to learn the distribution of the data.
\item \alert{Semi-supervised learning}: As before, we have a dataset $\{\inp_1, \ldots, \inp_T\} \subset \Inp$. Additionally, we now have outputs, or labels $\{\oup_1,\ldots, \oup_t\}$ for $t\leq T$ of the data points. I.e. we have pairs $(\inp_1, \oup_1),\ldots, (\inp_t,\oup_t)$.
A typical goal is to extend the labeling to the whole dataset, i.e., to also obtain pairs $(\inp_{t+1}, \oup_{t+1}), \ldots, (\inp_T,\oup_T)$.
\item \alert{Supervised learning}: Here, we are given a full set of input-output pairs 
\beq
\mathcal{T} = \{(\inp_1,\oup_1),\ldots, (\inp_T,\oup_T)\} \subset\Inp\times\Oup.
\eeq 
In the setting of image classification one then wants to infer a function $\net:\Inp\to\Oup$ acting on the whole space $\Inp$. However, there are many settings beyond classification, that employ supervised learning.
\end{enumerate}
\begin{figure}
\begin{subfigure}{.32\textwidth}%
\includegraphics[width=\textwidth]{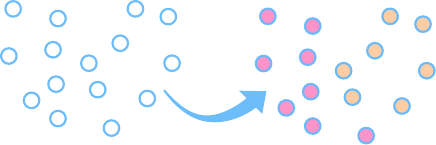}%
\caption{Unsupervised learning.}%
\end{subfigure}\hfill%
\begin{subfigure}{.32\textwidth}%
\includegraphics[width=\textwidth]{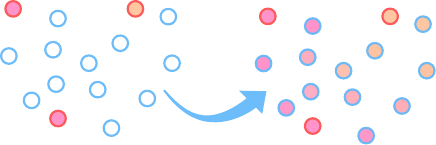}%
\caption{Semi-supervised learning.}%
\end{subfigure}\hfill%
\begin{subfigure}{.32\textwidth}%
\includegraphics[width=\textwidth]{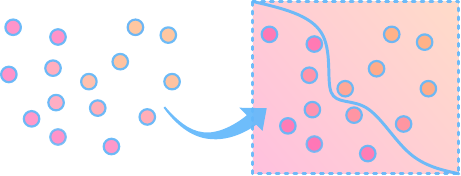}%
\caption{Supervised learning.}%
\end{subfigure}%
\end{figure}
For our introduction, we first focus on the supervised setting. Given the training dataset $\mathcal{T}$, we want to find the function $\net$ that matches this data best. This means that given $(\inp,\oup)\in\mathcal{T}$, we would like to find $\net$ such that $\net(\inp)=\oup$. To this end, it is useful to consider a function $\ell:\Oup\times\Oup\to \R^+_0$ which measures the difference between two outputs. In this interpretation we then would like to obtain $\ell(\net(\inp), \oup) = 0$ or at least as small as possible. Training consists of solving the minimization problem
\begin{align*}
\min_{\net\in\mathcal{H}} \ls\mathcal{L}(\net)\rs := \min_{\net\in\mathcal{H}} \ls 
\sum_{(\inp,\oup)\in\mathcal{T}} \ell(\net(\inp), \oup)\rs,
\end{align*}
where $\mathcal{H}$ denotes the class of possible functions we consider. We proceed by introducing two important examples of such classes.

\subsection{Spectral Architecture}

Let $\mcA:\mcU\to\mcV$ denote a compact linear operator. From \cref{thm:svd} we know that using the SVD $(U, \Sigma, V)$, we have that 
\begin{align}
\mcA u = \sum_{i=1}^\infty \sigma_i \langle u, u_i\rangle v_i.
\end{align}
In particular, the pseudo inverse is given as
\beq
\mcA^\dagger f = \sum_{i=1}^\infty \frac{1}{\sigma_i} \langle f, v_i\rangle u_i.
\eeq
We have already introduced the concept of obtaining a regularization, by filtering the coefficients $\frac{1}{\sigma_i}$. We will employ this strategy in the learning setting, where we make the following definition, which was for example employed in \cite{burger2023learned}.
\begin{definition}{Spectral Architecture}{}
For a compact linear operator $\mcA$ with SVD $(U,\Sigma, V)$ the spectral architecture is a reconstruction operator defined as $H:V\to U$,
\beq
\net(f; \Theta) := \sum_{i=1}^\infty \theta_n \langle f, v_i\rangle u_i.
\eeq
\end{definition}
In the above setting, the coefficients $\theta_n$ can be seen as the learnable parameters of this model. We will see later, that this definition allows us to extend the concept of convergent regularizers to the learning setting.

\subsection{Neural Networks}
A popular choice for a family of functions $\mathcal{H}$ are so-called \alert{neural networks}. Here, we mainly consider \alert{feedforward} neural networks, where we make the following definition, which is motivated by Rosenblatt's perceptron \cite{rosenblatt1958perceptron}\footnote{Frank Rosenblatt (1928 -- 1971) was an American psychologist.}.
\begin{definition}{Fully Connected Layer}{}
We denote by $W\in\R^{\tilde{d},d}$ the so-called \alert{weight matrix}, by $b\in\R^{\tilde{d}}$ the so-called \alert{bias vector} and call $\sigma:\R^{\tilde{d}}\to\R$ an \alert{activation} function. Then the mapping
\beq
\Lambda^\text{FC}(z; W, b, \sigma) := \sigma(Wz + b)
\eeq
is called \alert{fully-connected layer} with parameters $\theta=(W,b)$.
\end{definition}
The motivation of the above layer comes from a simplified model of how a brain processes information. In this interpretation, each entry of a vector $z$ is seen as a \alert{neuron}. Its value represents the signal it has received from a previous time step. In the next time step, the $i$th neuron passes its signal to all neurons it connects to; $W_{ji}$ gives a weight to the connection from $i$ to $j$. The activation then modifies this signal, e.g., by only passing the signal on if it is above a certain threshold. Here, a bias vector can help us specify this threshold individually for each neuron $j$. In the following section we will explore different types of \enquote{layers}. In general these are mappings $\Lambda:\R^d\to\R^{\tilde{d}}$ parameterized by some parameters $\theta$.
Often, the activation function is defined as a map from $\R$ to $\R$ and, with a slight abuse of notation, we write for $z\in\R^d$
\beq
\sigma(z)_i := \sigma(z_i).
\eeq
Typically, the activation function is set a-priori and is therefore part of the architecture. However, there also exist parametrized activation functions.
\begin{example}{Activation functions}{}
Here, we collect typical examples of activation functions used in literature.
\begin{enumerate}
\item \alert{ReLU}: the so-called rectified linear unit is defined as \cite{fukushima2007visual}
\beq\tag{ReLU}
\sigma(z) := \max\{0,z\}.
\eeq
\item \alert{PReLU}: the so-called parametric ReLU function is a defined as 
\beq\tag{PReLU}
\sigma(z):=
\begin{cases}
z &\text{if } z\geq 0,\\
az &\text{else},
\end{cases}
\eeq
where $a\geq 0$ is a learnable parameter.
\item \alert{Sigmoid}: the sigmoid function is defined as 
\beq\tag{Sigmoid}
\sigma(z) := \frac{1}{1+ \exp{-z}}
\eeq
\item \alert{Softmax}: the softmax function is defined in \labelcref{eq:softmax}.
\end{enumerate}
\end{example}
A feedforward neural network is then defined as the composition of multiple layers. Stacking multiple layers behind each other also motivates the name \alert{deep} learning.
\begin{definition}{Neural Networks}{}
A composition of layers
\begin{align*}
H_\theta = \Lambda^L(\cdot; \theta^L)\circ\cdots\circ\Lambda^1(\cdot; \theta^1)
\end{align*}
is referred to as a \alert{neural network}. We denote by $\theta=(\theta^L, \ldots, \theta^1)$ the collection of parameters and by $\Theta\ni\theta$ the space of all possible parameters.
\end{definition}
A typical question that arises with the above definition, is how expressive the family $\mathcal{H}(\Theta):=\{\net_\theta:\theta\in\Theta\}$ is. In other words, which functions are contained in $\mathcal{H}$ and which can be at least approximated arbitrarily well, i.e., are contained in $\overline{\mathcal{H}(\Theta)}$. Results that try to answer this type of question are known as \alert{universal approximation theorems}. Here, we only state the most popular result, which is attributed to Cybenko\footnote{George Cybenko is a mathematician.}.
\begin{theorem}{Universal Approximation}{}
Consider the family of neural networks defined as
\begin{align}
\mathcal{H}(\Theta) &:= \{H(\inp)=
W^2\sigma(W^1\inp + b^1): (W^2, W^1, b^1) \in \Theta\},\\
\Theta&:= 
\left\{(W^2, W^1, b^1): W^1\in \R^{d,1}, b\in\R^{d}, W^2\in\R^{1,d} \text{ for some } d\in\N
\right\}.
\end{align}
Iff $\sigma$ is not a polynomial, then $\mathcal{H}(\Theta)$ is dense in the space of continuous functions $C(\R)$ w.r.t.~ to the supremum norm.
\end{theorem}
\begin{proof}
See, e.g., \cite{cybenko1989approximation}.
\end{proof}
In the above statement, we allow varying the width $d$ of the network. This is why the above statement is also referred to as arbitrary width limit.
\paragraph{Convolutional layers}
Another important type of layers is induced by convolutions, which we define in the following.
\begin{definition}{Convolutional Layer}{}
A convolutional layer \cite{fukushima2007visual} is defined as
\beq
\Lambda^\text{Conv}(z; \kappa, b, \sigma) := \sigma(z \ast \kappa + b)
\eeq
with parameters $\theta=(\kappa, b)$.  
\end{definition}
This type of layer is especially attractive in imaging applications. Compared to a fully-connected layer, this still leads to a linear operation in the input variable. However, the number of learnable parameters is drastically reduced. Recall the definition of the convolution operation given in \cref{def:convolution}. In practice, the quantities $z,\kappa$ are only defined on a finite grid, or equivalently, have finite support and can be written as matrices. In the following, we assume that the kernel has an odd width and height, i.e., $\kappa\in\R^{2M+1, 2M+1}$. The number of trainable parameters is then given by $(2M+1)^2$. Notably, it does not depend on the size of $z\in\R^{d_h, d_w}$, where $d_h,d_w$ denote the height and width of the image. Since kernels are usually much smaller then the images, convolutional layers require much fewer parameters than fully connected layers. 

In matrix notation, the convolution can then be written as
\beq\label{eq:conv}
(\kappa \ast z)_{n,m} = 
\sum_{i=-M}^M \sum_{j=-M}^M \kappa_{i,j}\cdot z_{n-i, m-j}.
\eeq
In this discrete setting, we need to carefully consider the boundary cases of $z$. There are different ways to treat the boundary:
\begin{itemize}
\item \textbf{Padding}: We can extend the image to the size ${(d_h+2M)\times(d_w + 2M)}$ such that the sum in \cref{eq:conv} is well-defined at each boundary. Popular choices of padding methods are zero- or circular padding.
\item \textbf{No-padding}: If we do not perform the padding operation, the output of the convolution has a reduced dimension of $(d_h-M)\times(d_w-M)$. As a linear operator, the convolution is then a mapping $K:\R^{d_h\times d_w}\to \R^{(d_h-M)\times(d_w-M)}$.
\end{itemize}
Furthermore, we typically have multiple channels $d_c$, i.e. $z\in\R^{d_c, d_h, d_w}$. The convolution operation can be adapted to
\beq
z\mapsto \sigma\left(\sum_{c=1}^{d_c} \kappa^c \ast z + b\right)
\eeq
where every channel $c$ now has its own kernel. The expressiveness of a convolutional layer can be drastically improved by increasing the number of channels in one layer, such that we in fact have the layer
\begin{gather}
\Lambda^\text{Conv}:\R^{d_c,d_h,d_w}\to \R^{\tilde{d_c},d_h,d_w}\\
\Lambda^\text{Conv}(z)_{\tilde{c},:,:} := 
\sigma\left(\sum_{c=1}^{d_c} \kappa^{\tilde{c},c} \ast z + b^{\tilde{c}}\right).
\end{gather}
\paragraph{Dimension reduction}
To account for this increase in dimension, it is desirable to decrease the height and width of the images with each layer. Here, we list the following popular choices:
\begin{itemize}
\item Choosing a \enquote{no-padding} convolutional layer, the image size is decreased by half the kernel size~$M$.
\item One often introduces so-called \alert{pooling} layers which reduce the resolution of the image. A basic example is a nearest neighbor downsampling which -- under the assumption that  $d_h,d_w$ are even -- can be written as 
\begin{align}
P:\R^{d_h,d_w}\to \R^{d_h/2, d_w/2}, \quad
(P(z))_{i,j} := z_{2i, 2j}.
\end{align}
This can often be efficiently implemented as a \enquote{strided convolution}.
\end{itemize}
\paragraph{Dimension increase} We saw above that the \enquote{no-padding} convolution effectively reduces the dimension of the output. Interpreting the convolution as a linear operator $\Lambda^{\text{Conv}}:\R^{d_h, d_w} \to \R^{\tilde{d}_h, \tilde{d}_w}$, with $\tilde{d}_h\leq d_h, \tilde{d}_w\leq d_w$, we observe that its transpose can be used as an upsampling operation, since $\left(\Lambda^{\text{Conv}}\right)^T:\R^{\tilde{d}_h, \tilde{d}_w}\to\R^{d_h, d_w}$. This is often referred to as a \alert{transposed convolution}.
\paragraph{Residual layers}
A further enhancement is to transform any layer into its residual variant. In addition to the actual layer evaluation, one adds the original input to the output.
\begin{definition}{Residual layer}{}
Given a layer $\Lambda(\cdot ;\theta)$, then its residual version is defined as
\beq
\Lambda^\text{Res}(z;\theta) := z+
\Lambda(z;\theta).
\eeq
\end{definition}
\begin{example}{U-Nets}{}
One specific architecture, which is especially relevant for inverse problems, is the so-called \enquote{U-Net}. It was introduced in \cite{ronneberger2015u} and employs multiple convolutional and residual layers. Together with pooling and upsampling operations, the resulting architecture forms a U-shape which motivates the name. A visual explanation desribing this architecture can be found in the original paper \cite[Figure 1]{ronneberger2015u}.
\end{example}

\subsection{Training Neural Networks}

In this section, we derive the necessary ingredients to train a neural network. Recall that training means solving the following problem
\begin{align*}
\min_{\theta\in\Theta} \mathcal{L}(\theta) = 
\min_{\theta\in\Theta} \frac{1}{\abs{\mathcal{T}}}
\sum_{(v,u)\in\mathcal{T}} \ell(H_\theta(v), u),
\end{align*}
where $H_\theta$ is a neural network parameterized by $\theta\in\Theta$. In \cref{ch:reg}, we have introduced different optimization strategies. Ignoring (for now) the question of differentiability,  we focus in the following on gradient descent-based optimizers. In order to apply this algorithm, we need to evaluate the gradient
\beq
\nabla \mathcal{L}(\theta),
\eeq
which we deal with in the following. 

\paragraph{Backpropagation and automatic differentiation}

Neural networks can be understood as the composition of multiple different operations. Typically, for each of these operations, we can easily compute an analytic gradient. Therefore, the difficulty that remains is to calculate the gradient of the composition. In a more general context, that is not limited to neural networks, this leads to the field of \alert{automatic differentiation}. We focus only on a very simplified model, and refer to \cite{baydin2018automatic} for a more in-depth study of this topic. From a mathematical perspective, backpropagation, which we describe in the following, is an application of the chain rule. From this point of view it could either be attributed to Leibniz \cite{leibniz2007early}, L’Hospital \cite{lhopital}, Lagrange \cite{lagrange1813theorie} or Cauchy \cite{cauchy1823resume}\footnote{See for example \cite{boyer1959history} for the history of calculus.}. In the context of training neural networks a first appearance of the algorithm can be found in \cite{rosenblatt1958perceptron}. It was later popularized in \cite{rumelhart1986learning}. The following presentation is partially motivated by \cite{parr2018matrix}.

We consider a simple fully-connected neural network without bias vectors:
\beq
H_\theta(v) = \Lambda^L(\cdot, \theta^L)\circ\cdots\circ \Lambda^1(\cdot; \theta^1)(v) = 
\sigma(W^L\, \sigma(\ldots \sigma(W^1 v))).
\eeq
In order to compute the gradient of this function, we can employ the chain rule. Here, it is convenient to think in terms of the total derivative $D_a f$ of a function $f$ evaluated at a point $a$. Namely, consider the function $p(z):= \ell(z, u)$ where $(v,u)$ is a given data point. Then we have
\begin{align}
D_v \lb p \circ H_\theta \rb = D_{H_\theta(v)} p \circ D_v H_\theta.
\end{align}
We can introduce some notation that makes this clearer:
\begin{align}
a^0 := v, \qquad  w^l := W^l a^{l-1}, \qquad a^l := \sigma(w^l).
\end{align}
In particular, it thus holds $a^L = \sigma(w^L) = H_\theta(v)$.
We want to compute the gradient of the loss function w.r.t. all parameters. We can iteratively derive a rule by first considering
\begin{align}
\delta^L = \frac{\partial \ell(H_\theta(v), u)}{\partial w^L} = 
(D_{w^L} \ell(\sigma(\cdot), u))^T = 
(D_{a^L} \ell(\cdot, u) \circ D_{w^L}^T \sigma)^T = 
\sigma'(w^L) \odot D_{a^L} \ell(\cdot, u).
\end{align}
For the next layer, we have that
\begin{align}
\delta^{L-1} &= (D_{w^{L-1}} \ell(\sigma(W^L\sigma(\cdot)), u))^T = 
(D_{w^{L}} \ell(\sigma(\cdot), u) \circ D_{w^{L-1}} (W^L\sigma(\cdot)))^T\\
&= (\delta^L \circ W^L \odot \sigma'(w^{L-1}))^T
= (W^L)^T \delta^L \odot \sigma'(w^{L-1}).
\end{align}
Analogously, we can derive
\begin{align}
\delta^l = (W^{l+1})^T \delta^{l+1} \odot \sigma'(w^{l})
\end{align}
for any $l=1, \ldots, L-1$, which motivates the name backpropagation. The gradient w.r.t. the weight matrices can then further be computed as 
\beq
\frac{\partial\ell}{{\partial W^l}} = \delta^l a^{l-1}.
\eeq

\paragraph{Stochastic Gradient Descent}
Our loss function does not only consist of $\ell(H_\theta(v), u)$ at a single data point $(v,u)$ but rather of the sum of all data points. Since we typically have a large number of data points, it is computationally unfeasible to store them all in memory at the same time. To overcome this issue, we consider so-called \alert{minibatching}. This consists of randomly sub-sampling the data, i.e., extracting a smaller subset $B\subset\mathcal{T}$ and then only considering the loss
\beq
\mathcal{L}(\theta; B) := \frac{1}{\abs{B}}\sum_{(v,u)\in B} \ell(H_\theta(v), u).
\eeq
This yields the following algorithm, which is known as \alert{(minibatching) Stochastic Gradient Descent (SGD)}. For the first appearance of SGD in the context of neural network training, we again refer to \cite{rosenblatt1958perceptron}.
\begin{definition}{SGD}{}
Let the initial point $\theta^{(0)}$ and step size $\tau$ be given, then the iterates of the minibatching stochastic gradient descent algorithm are given as
\beq%
\begin{gathered}
\text{Sample}\quad B\subset \mathcal{T}\\
\theta^\kk_\SGD = \theta^\k_\SGD - \tau\, \nabla \mathcal{L}(\theta^\k_\SGD; B).
\end{gathered}
\eeq%
\end{definition}
The main difference to standard gradient descent is that we now do not evaluate the true gradient, but rather a stochastic estimator. An important property is the unbiasedness, namely
\beq\label{eq:unbias}
\E{}{\nabla \mathcal{L}(\theta; B)} = \nabla \mathcal{L}(\theta).
\eeq
\begin{exercise}{}{}
Show \cref{eq:unbias}.
\end{exercise}

\section{Convergent Data-Driven Regularization}

In \cref{sec:reg} we have introduced the terms \emph{regularization}, \emph{parameter choice rule} and \emph{convergent regularization} as they are defined in the classic setting; we now want to re-consider and re-define these terms in a data-driven setting. This chapter is based on \cite{burger2023learned,kabri2024convergent}.

Previously, regularization described a family $\ls \mfRa \rs_{\alpha\in I }$ of continuous linear operators parametrized by some real variable $\alpha$. The optimal choice of  $\alpha:\R^+\times \mcV \to I$ was formulated either in dependence of only the assumed noise level $\delta\in\R^+$ or additionally in dependence of a given noisy measurement $\fd\in\mcV$. 
One specific example we have considered is the spectral regularization
\beq
    \mfRa f := \sum_{i=1}^\infty r_\alpha(\sigma_i)\ll f, v_i\rr_\mcV u_i,
\eeq
which uses the SVD $\ls(U, \Sigma, V)\rs$ of the forward operator $\mcA$. Previously, we have handcrafted $r_\alpha$.

\newcommand{\priordis}{\pi}
\newcommand{\priorfam}{\tilde{\bs{\pi}}}

\newcommand{\inoisedis}{\nu}
\newcommand{\inoisefam}{\tilde{\bs{\nu}}}

\newcommand{\tnoisedis}{\mu}
\newcommand{\tnoisefam}{\tilde{\bs{\mu}}}

Now, in the data-driven setting, we want to consider the spectral architecture
\beq
    H(f; \theta) := \sum_{i=1}^\infty \theta_i \ll f, v_i \rr u_i
\eeq
and determine the parameters $\theta = \ls \theta_i\rs_{i\in\mathbb{N}}$ with $\theta_i := r_\alpha(\sigma_i)$ by minimizing a risk function which can depend on the distribution of the measurements and (for supervised approaches) also on the distribution of the ground truth data. E.g., for $\ell_\mcU:\mcU\times\mcU\to\R$ denoting some loss function
\beq
    \theta^* \in \argmin_{\theta \in \ls\theta_i\rs_{i\in\N}} \E{u\sim\priordis, \noise\sim\tnoisedis}{\ell_\mcU(u, H_\theta(\mcA u + \noise))}.
\eeq
Here, the parameters are learned from pairs $(u, Au+\varepsilon)$ where $u\sim\priordis$ (the prior $\priordis$ comes from a family of distributions $\priorfam$) and $\varepsilon\sim\tnoisedis$ (the training noise $\tnoisedis$ comes from a family of distributions $\tnoisefam$). Note, that $\tnoisedis$ might differ from the ideal noise model $\inoisedis$ which might even come from a different family of distributions $\inoisefam$! For a fixed risk functional and fixed method of obtaining the minimizer to this functional the optimal parameters are greatly influenced by the distributions $\pi$ and $\mu$ -- or in the empirical setting by the choice of training samples which reflect said distributions. 
 
 In the following, we mark the dependence on the choice of training noise $\tnoisedis$ and prior $\priordis$ by using the notation $\ddreg$ for the respective spectral architectures. While the resulting regularization operator may be non-unique in practice due to local minima in the training process, stochastic optimization or initialization, we here discard this non-uniqueness.

\begin{definition}{}{}
    A family of operators $\ddreg:\mcV\to\mcU$ is called a \alert{(data-driven) family of regularization operators} if $\ddreg$ is continuous for each training noise model $\tnoisedis\in\tnoisefam$ and prior model $\priordis \in \priorfam$.
\end{definition}

The notion of convergence now can be defined in a statistical sense, where we use the following definition for the statistical noise level of a distribution:
\beq
    \boldsymbol{\delta}(\tnoisedis) := \sqrt{\sup_{n\in\mathbb{N}} \E{\varepsilon\sim\tnoisedis}{\ll \varepsilon, v_n\rr^2}}.
\eeq

\begin{definition}{}{}
    A family of regularization operators $\ddreg$ is called \alert{convergent} if there exists a parameter choice rule $\tnoisedis(\delta, \nu)$ and $\priordis(\delta, \nu)$ such that
    \beq
        \lim_{\delta \to 0} \sup_{\boldsymbol{\delta}(\nu^\delta)  \leq \delta} \ls  \E{\varepsilon\sim \nu^\delta}{\| \ddregdep (\mcA u+\varepsilon) - \mcAdag \mcA u \|_\mcU } \rs = 0
    \eeq
    for each $u\in \mcU$. We call the regularization operators \alert{convergent over a prior distribution} $\priordis^*$ if 
    \beq
        \lim_{\delta \to 0} \sup_{\boldsymbol{\delta}(\nu^\delta) \leq \delta)}\ls  \E{u\sim\priordis^*, \varepsilon\sim \nu^\delta}{ \| \ddregdep (\mcA u+\varepsilon) - \mcAdag \mcA u \|_\mcU } \rs = 0.
    \eeq    
\end{definition}

\paragraph{Supervised Learning of Spectral Regularization.} Let us now look at a concrete example. In a supervised setting, the optimal parameters can be learned such that they minimize the mean squared error, i.e.,
\beq
    \theta^{mse}(\tnoisedis, \priordis) = \argmin_{\theta = \ls \theta_i \rs_{i\in\N} } \ls \E{u\sim\priordis, \varepsilon\sim\tnoisedis }{ \| \ddregdep(\mcA u + \varepsilon; \theta) - \mcAdag \mcA u \|^2 }\rs.
    \label{eq:thetamse}
\eeq
For noise which is uncorrelated to the data one can show that the optimal parameters are given by
\beq
    \theta_i^{mse}(\tnoisedis, \priordis) = \frac{\sigma_i}{\sigma_i^2 + \frac{ \Delta_i(\tnoisedis)}{\Pi_i(\priordis)}}, 
    \label{eq:thetaimse}
\eeq
where $\Delta_i(\tnoisedis) := \E{\varepsilon \sim \tnoisedis}{\ll \varepsilon, v_i\rr^2}$ and $\Pi_i(\priordis) := \E{u\sim\priordis}{\ll u, u_i\rr^2}$.

\begin{exercise}{}{}
    Show that the parameters given in \cref{eq:thetaimse} are indeed optimal for \cref{eq:thetamse} under the assumption that noise and data are uncorrelated.
\end{exercise}

Under the assumption that the optimal prior distribution is used, one can show that the associated reconstruction operators are continuous under reasonable conditions on noise and data.

\begin{lemma}{Continuity \cite[Lemma 1]{burger2023learned}}{}
    The reconstruction operator $\ddregdepef^{mse}: \mcV \to \mcU$, with
    \beq
        \ddregdepef^{mse}(f) = \sum_{i=1}^\infty \theta_i^{mse}(\tnoisedis, \priordis^*) \ll f, v_i \rr_\mcV u_i
    \eeq
    is continuous if and only if there exists a constant $c> 0$ such that
    \beq
        \Delta_n(\tnoisedis) \geq c \sigma_n \Pi^*_n
    \eeq
    for almost any $n\in\mathbb{N}$. In particular, this condition is fulfilled if there exists $c> 0$ and $n_0\in\mathbb{N}$ such that
    \beq
        \Delta_n(\tnoisedis) \geq c \Pi^*_n
    \eeq
    for all $n\geq n_0$.
\end{lemma}

\begin{proof}
    See \cite[Lemma 1]{burger2023learned}.
\end{proof}

Thus the approach defines a data-driven family of regularization operators. Furthermore, one can show convergence:

\begin{theorem}{Convergence \cite[Theorem 1]{burger2023learned}}{}
    The family of learned spectral reconstruction operators $\ls \ddregdepef^{mse}\rs_{\tnoisedis\in\tnoisefam}$ with
    \beq
        \tnoisefam = \ls \tnoisedis \, \middle |\, \ddregdepef^{mse}:\mcV \to \mcU \text{ is continuous }\rs
    \eeq
    is a convergent data-driven regularization for $\mcA^\dagger$ and in particular it holds for any $u\in\kernel(\mcA)^\perp$ that
    \beq
        \sup_{\inoisedis^\delta\in\inoisefam, \boldsymbol{\delta}(\inoisedis^\delta)\leq\delta} \E{\noise\sim\inoisedis^\delta}{\| H_{\inoisedis^\delta, \priordis(\delta,\inoisedis^\delta)}(\mcA u+\varepsilon) - u \|^2 } \to 0
    \eeq
    as $\delta\to 0$.
\end{theorem}

\begin{proof}
    See \cite[Theorem 1]{burger2023learned}.
\end{proof}



\section{Different Machine Learning Approaches for Inverse Problems}
\label{sec:mlforip}

In the context of inverse problems, machine learning approaches come in a variety of different flavors. Some approaches aim to directly map either the (noisy) measurement or a model-based noisy reconstruction of that measurement to the corresponding true quantity of interest (end-to-end approaches). Other approaches seek to find a data-adaptive regularizer instead of handcrafting it (learned regularization methods). Some of the approaches build on regularization and others on Bayesian inversion. For a detailed overview and introduction to the here discussed and also to other methods, we refer to \cite{arridge2019solving}, which we also lean on for the following exposition.

\paragraph{Learning in statistical regularization.} Some learning approaches are directly linked to the Bayesian approach of solving inverse problems. Let us quickly recall the Bayesian approach and fix some notation: Both $\urv$ and $\frv$ are considered to be random variables. We assume that there exists an $\mcU\times\mcV$--valued random variable $(\bs{u}, \bs{f})$ which follows an unknown joint distribution $\duf$.
The goal of the Bayesian approach is to approximate the posterior $\posterior$ which is a distribution of the $\mcU$--valued random variable $(\urv | \frv = f)$. The reconstruction in this setting then can be derived via different estimations of the posterior distribution. E.g., the typical \alert{Bayes estimator} for $\urv$ conditioned on $\frv$ amounts to finding a reconstruction operator $\mfR : \mcV \to \mcU$ which solves
\beq
    \mfR \in \argmin_{\mfR:\mcV\to\mcU} \E{(\bs{u}, \bs{f})\sim \duf}{\ell_\mcU (\bs{u}, \mfR (\bs{f}))},
    \label{eq:bayesian_estimator}
\eeq
for some fixed loss function $\ell_\mcU:\mcU\times\mcU\to\R$. By the law of total probability we have
\beq
    \duf (u, f) = \prior(u) \otimes \data(f).
\eeq
Since the data likelihood $\data$ is a distribution of $(\frv | \urv = u)$ and known for any $u\in\mcU$ via
\beq
    \frv = \mcA (\urv) + \erv
\eeq
with $\erv \sim \dnoise$ the joint distribution is known as soon as a prior $\prior$ and a noise model $\dnoise$ have been selected. Even with this, calculating \cref{eq:bayesian_estimator} comes with the following challenges:
\begin{itemize}
    \item The choice of an appropriate prior is challenging and hand-crafted choices often do not capture available a-priori information to full extent.
    \item Calculating the expectation over $\mcU\times\mcV$ is infeasible.
    \item Minimizing over the set $\mfR:\mcV\to\mcU$ is infeasible.
\end{itemize}
These issues can be addressed using supervised learning. First, we replace $\duf$ with supervised training data 
\beq
    \mcT := \ls (f_i, u_i) \,  \middle | \, (u_i, f_i) \text{ drawn from } (\urv, \frv) \sim \duf \text{ for } i=1,...,T \rs \subset \mcU\times\mcV. 
\eeq
Second, we replace $\mfR$ with a family of reconstruction methods $\ls H_\theta:\mcV\to\mcU \rs_{\theta\in\Theta}$ parametrized by a deep neural architecture. Then we can approximate $\mfR$ with $H_{\hat{\theta}}$ where 
\beq
    \hat{\theta} \in \argmin_{\theta \in \Theta} \frac{1}{T} \sum_{i=1}^T\ell_\mcU (u_i, H_{\theta}(f_i)).
    \label{eq:bayes_estimator_network}
\eeq
This, however, opens a plethora of new problems as \cref{eq:bayes_estimator_network} is in it-self ill-posed and solving it requires incorporation of a regularization, careful choice of network architecture, training algorithm and initialization, etc. In the following, we consider two approaches:

\begin{itemize}
    \item Fully learned Bayes estimation refers to approaches where one assumes there is enough supervised training data to learn the joint law. Any explicit knowledge about the data likelihood is disregarded.
    \item  Post-processing methods apply an initial reconstruction operator that maps noisy measurements to noisy reconstructions and afterwards approximate an estimator for the noise-free reconstruction conditioned on the initial noisy reconstruction.
\end{itemize}

\subsection{Fully Learned Bayes Estimation}

In these kind of approaches one assumes that enough supervised training data
\beq
    \mcT := \ls (f_i, u_i) \,  \middle | \, (u_i, f_i) \text{ drawn from } (\urv, \frv) \sim \duf \text{ for } i=1,...,T \rs
\eeq
is available to learn the joint distribution $\duf$. It disregards available information on $\data$, i.e., does not incorporate knowledge about the the forward process $\mcA$. Using this approach we learn an approximation
\beq
    \mfR_{\text{FLBE}}:\mcV\to\mcU, \quad \mfR_{\text{FLBE}}(\fd) := H_{\hat{\theta}}(\fd),
\eeq
of \cref{eq:bayesian_estimator} with 
\beq
    \hat{\theta} \in \argmin_{\theta \in \Theta} \frac{1}{T} \sum_{i=1}^T \ell_\mcU (u_i, H_{\theta}(f_i)).
\eeq
Typically, employed networks in this setting are composed of two parts, i.e.,
\beq
    H_{\theta} := H^{nn}_{\theta_1} \circ H^{fc}_{\theta_2}, \quad \theta= (\theta_1, \theta_2).
\eeq
$H^{fc}_{\theta_2} : \mcV \to \mcU$ has one or multiple fully connected layers and represents the pseudo-inverse operation. $H^{nn}_{\theta_1}: \mcU \to \mcU$ is, e.g., a deep CNN or U-Net and represents the denoising operation. A prominent example is the \alert{automated transform by manifold approximation} (AutoMap) \cite{zhu2018image}.

\paragraph{Pros and Cons.} AutoMap and other fully learned Bayes estimation approaches are attractive due to their simplicity; They bypass the need for an explicit forward model or data likelihood. Also, performing the inference step (the actual reconstruction using the trained model) is rather fast. The simplicity, however, comes with a cost. In order to be able to learn the reconstruction, one has to use neural networks with connected layers which always scale badly. While one can try to overcome this using adapted network architectures for a specific inverse problem, the second problem, namely the demand for a large amount of supervised training data, is more difficult to fix. Also, the networks require retraining as soon as, e.g., the acquisition protocol or instruments used for taking the measurements change. Altogether, these challenges make these method rather impractical.

\subsection{Post-Processing}

Now one might ask, if it is truly necessary to throw away available knowledge of the forward model and to have a network learn the inverse process instead. If the forward-model is well-known, one can obtain a partially-learned regularization which incorporates the domain knowledge by first attempting a naive reconstruction approach, which is then corrected by applying a learned denoiser. This approach is referred to as \emph{post-processing}. The first naive reconstruction step can, e.g., be the application of the pseudo-inverse $\mcAdag$ which maps the (noisy) measurement to a (noisy) reconstruction.  In a second step, a data-driven approach is then used to remove the noise from this initial reconstruction. This can be formulated as the reconstruction process
\beq
    \mfR_{\text{PP}}:\mcV\to\mcU, \quad \mfR_{\text{PP}} := \mcH_{\text{PP}} \circ \mcAdag.
\eeq
Let us take a closer look at how such an $\mcH_{\text{PP}}$ can be trained. If one has supervised training data $\mcH_{\text{PP}}$ is typically chosen as the Bayes estimator of $\bs{u}$ conditioned on the initial reconstruction $\mcAdag(\bs{f})$. If such data is not available, one has to use other non-Bayesian estimators which are less well understood.

\paragraph{Supervised Setting.} The Bayes estimator $\mcH_{\text{PP}}$ for $\bs{u}$ conditioned on the initial reconstruction $\mcAdag(\bs{f})$ is given as 
\beq
    \mcH_{\text{PP}} \in \argmin_{\mcH:\mcU\to\mcU} \E{(\bs{u},\bs{f})\sim\Pi}{\ell_\mcU (\bs{u}, \mcH (\mcAdag(\bs{f})))}
\eeq
for some fixed loss function $\ell_\mcU:\mcU\times\mcU\to\R$. In supervised post-processing we learn an approximation $H_{\hat{\theta}}:\mcU\to\mcU$ of $\mcH_{\text{PP}}$ from  supervised training data
\beq
    \mcT := \ls (\mcAdag f_i, u_i) \,  \middle | \, (u_i, f_i) \text{ drawn from } (\urv, \frv) \sim \duf \text{ for } i=1,...,T \rs
\eeq
via 
\beq
    \hat{\theta} \in \argmin_{\theta \in \Theta} \frac{1}{T} \sum_{i=1}^T \ell_\mcU (u_i, H_{\theta}(\mcAdag f_i)).
\eeq
Popular architectures for $H_\theta$ are CNNs and U-Nets.

Now one might ask if there is a connection between the fully learned Bayesian estimator $\mfR_{\text{FLBE}}$ and $\mfR_{\text{PP}}$ in the supervised setting. As stated in \cite[p.98]{arridge2019solving}: These two indeed coincide if the considered loss $\ell_\mcU$ is the squared $L^2$--norm and if the considered reconstruction method is a so-called \emph{linear sufficient statistic}, i.e., if it satisfies
\begin{equation}
    \E{}{\bs{u} | \bs{u}^\dagger = u^\dagger} = \E{}{\bs{u} | \mcAdag(\bs{\fd}) = \mcAdag(\fd)}. 
\end{equation}
This equivalence, however, only holds in the limit of infinite amount of training data and infinite model capacity. In fact, when applied to finite number of training data and finite model capacity, learned post-processing differs from deep direct Bayes estimations ... so basically in all real-life applications.

\paragraph{Semi-Supervised Setting.} Semi-supervised learning does not exclusively depend on samples from the joint distribution $\Pi$ -- which in practice one might not have enough of -- but also allows for unpaired samples from the respective marginals $\Pi_{\bs{u}} (= \prior)$ and  $\Pi_{\bs{f}}$. Estimators for the post-processing application can be formulated as 
\begin{align}
\small
\mcH \in\argmin_{\mcH:\mcU\to\mcU}  \Big\{ \E{(\bs{u}, \bs{f}) \sim \Pi_{\bs{u}} \otimes \Pi_{\bs{f}}} { 
\ell_{\mcU}   (\mcH(\mcAdag (\bs{f})), \bs{u}) 
+ \ell_{\mcV} (\mcA (\mcH(\mcAdag (\bs{f}))), \bs{f}) 
+ \lambda \ell_{\mathcal{P}_{\mcU}} ((\mcH\circ\mcAdag)_{\#}(\Pi_{\bs{f}}), \Pi_{\bs{u}})
} \Big\},
\end{align}
where
\begin{itemize}
    \item $\ell_\mcU:\mcU\times\mcU\to\R$ and $\ell_\mcV:\mcV\times\mcV\to\R$ are  loss functions on $\mcU$ and $\mcV$ respectively; Note that $\ell_\mcU$ may later be defined such that it is only evaluated on the labeled samples $(f_i, u_i)$ drawn from $(\bs{f},\bs{u})\sim\Pi$.
    \item $\ell_{\mathcal{P}_{\mcU}}:\mathcal{P}_{\mcU}\times\mathcal{P}_{\mcU}\to\R$ is a distance notion between two probability distributions on $\mcU$;
    \item $(\mcH\circ\mcAdag)_{\#}(\Pi_{\bs{f}})$ is the pushforward of $\Pi_{\bs{f}}$ by $\mcH\circ\mcAdag$, i.e.
    \beq
        \mfR_{\#}(\Pi_{\bs{f}}) (U) =\Pi_{\bs{f}}\left( \mfR^{-1} (U) \right), \quad \forall \, U\subset\mcU.
    \eeq
\end{itemize} 
Let us take a closer look at the term $\ell_{\mathcal{P}_{\mcU}}$. A popular choice is the \alert{Wasserstein-1 distance} \cite{villani2008optimal}:
\beq
    \ell_{\mathcal{P}_{\mcU}} (\mfR_{\#}(\Pi_{\bs{f}}), \Pi_{\bs{u}}) = \sup_{\substack{d:\mcU\to\R \\ d \text{ is } 1\text{-Lipschitz}}} \E{\bs{u}\sim\mfR_{\#}(\Pi_{\bs{f}})}{d(\bs{u})} - \E{\bs{u}\sim\Pi_{\bs{u}}}{d(\bs{u})}. 
    \label{eq:wasserstein}
\eeq
Since this is not a closed-form expression, evaluating it during training is not straightforward and it is common  to borrow techniques from \alert{generative adversarial networks (GANs)} \cite{goodfellow2014generative} to do so. The latter describes a competitive learning framework in which two networks are trained against each other. In the usual use case, one network (the generator) tries to create new images which are similar to images in a given training set, while the other network (the discriminator/critic) tries to discriminate these newly generated images from original (non-generated) images. To this end, the discriminator is trained to evaluate the distance between the probability distribution of the true images and the one of the generated images. Similar to the GAN setting, we can also train two networks. Given some semi-supervised training data
\beq
    \mcT := \ls u_i \, \middle| \, u_i \text{ drawn from } \bs{u}\sim\Pi_{\bs{u}} \rs \cup  \ls f_i \, \middle| \, f_i \text{ drawn from } \bs{f}\sim \Pi_{\bs{f}}  \rs.
\eeq
we can learn an approximation $\mcH_{\hat{\theta_1}}:\mcV\to\mcU$ with  $\mcH_{\hat{\theta_1}}:= H_{\hat{\theta_1}} \circ \mcAdag $ of $\mfR_{\text{PP}}$ and also an approximation $D_{\hat{\theta_2}}:\mcU\to\R$ of $d$ in \cref{eq:wasserstein}. $\mcH_{\hat{\theta_1}}$ takes the role of the \emph{generator} and $D_{\theta_2}$ the one of the \emph{discriminator}. The discriminator acts as a critic as to how likely it is that a sample comes from the prior $\Pi_{\bs{u}}$ or the distribution $(\mcH\circ\mcAdag)_{\#}(\Pi_{\bs{f}})$; the latter can be thought of as the distribution of the random variable $\tilde{\bs{u}} := \mcH_{\hat{\theta_1}}(\bs{f})$. The parameters $ \theta = (\theta_1, \theta_2)$ are determined such that 
\begin{align}
    \hat{\theta} \in \argmin_{\theta_1\in\Theta_1} \ls \frac{1}{T}\sum_{i=1}^T  
    \ell_{\mcU}   (H_{\theta_1} \circ \mcAdag (A(u_i)), u_i) 
    + \ell_{\mcV} (\mcA (H_{\theta_1} \circ \mcAdag (f_i)), f_i) \right.\\
    \left. + \lambda \argmax_{\theta_2\in\Theta_2} \ls 
    D_{\theta_2}(H_{\theta_1} \circ \mcAdag (f_i)) - D_{\theta_2}(u_i)
    \rs \rs.
\end{align}
In practice, the parameters are updated alternatively; usually for each update step of $\theta_1$ multiple update steps are taken for $\theta_2$.

\paragraph{Pros and Cons.} Post-processing methods are attractive, because they can be easily applied on top of traditional reconstruction methods. 

A potential drawback are the missing guarantees on the reconstruction quality. Artifacts introduced by the initial reconstruction (e.g. $\mcA^\dagger$) might not be suppressed appropriately, especially if they did not appear during training. This is especially severe in cases where $\mcH_{\text{PP}}$ is trained without knowledge of $\mcAdag$. For example in \cite{burger2023learned}, the continuity of the post-processing approach can only be ensured for restrictive noise assumptions; Namely when the training noise is bigger than the one used for unseen data. Apart from this theoretic insight, \cite{burger2024learning} also provides a practical example, where worst case perturbations lead to a complete failure of post-processing.

\subsection{Plug-and-Play}

In this section we consider so-called plug-and-play regularization\footnote{The exposition here is inspired by the following presentations \cite{LeClaire,Repetti2024}.}. This approach can be derived from the optimization approaches in \cref{ch:reg}. There, we considered the minimization problem
\beq\label{eq:optreg}
\min_u G(u) + H(u)
\eeq
where both $H$ and $G$ are proper, convex and lsc functionals. Typically, $H$ is a differentiable functional that relates to the data fidelity or -- from the Bayesian view point -- to the negative log likelihood, i.e., $H(u)= -\log p(f^\delta|u)$. In \cref{ch:reg}, $G$ was a \enquote{hand-crafted} regularizer that incorporated prior information about the data distribution, e.g., a TV denoising term which expects sharp edges in images. In order to solve \cref{eq:optreg}, we considered various different prox-based algorithms such as proximal gradient descent
\begin{align}
u^\kk_\PGD = \prox_{\tau G}(u^\k_\PGD - \tau \nabla H(u^\k_\PGD)).
\end{align}
Recall the definition of the prox operator:
\beq
\prox_{\tau G}(v) = \argmin_{z} \frac{1}{2} \norm{v - z}^2_2 + \tau G(z).
\eeq
We observe that the application of the prox operator can be interpreted as the solution of a denoising problem:
\begin{align*}
\fd = u + \noise, \qquad \noise\sim \mathcal{N}(0, \tilde{\delta}^2 I), \qquad u\sim\prior.
\end{align*}
This gives the main idea of plug-and-play, which can be summarized as follows:
\begin{center}
\itshape%
Replace the prox-operator in our algorithm by any denoiser $D$.    
\end{center}
For proximal gradient descent the iteration would thus read
\begin{align}
u^\kk_\PNPPGD = D(u^\k_\PNPPGD - \tau \nabla H(u^\k_\PNPPGD)).
\end{align}
Here, $D$ can be any denoising operation, e.g., a classical algorithm like BM3D \cite{dabov2007image} or a trained neural network.

However given this scheme the following questions arise:
\begin{itemize}
\item Is there a function $G$ such that $D = \prox_{\tau G}$?
\item Does the PnP-PGD scheme converge?
\end{itemize}
As seen in \cref{sec:opt} the proximal operator of $G$ can be expressed as the resolvent of its subdifferential,
\begin{align}
\prox_G = (I + G)^{-1}.
\end{align}
Furthermore, most of the convergence proofs we have considered are based on the theory of resolvents for maximally monotone operators. Therefore, instead of asking whether $D$ is the proximal operator of a functional $G$ we can instead ask if it is the resolvent of a maximally monotone operator. Here, we use the following lemma, which is a central ingredient for the plug-and-play schemes in \cite{pesquet2021learning}, see Proposition 2.1 therein.
\begin{lemma}{\cite[Proposition 2.1]{pesquet2021learning}}{}
An operator $A$ is maximally monotone iff there exists a non-expansive operator $Q$ (i.e. Lipschitz constant less than or equal to 1) such that $(I+A)^{-1} = \frac{1}{2}(I + Q)$.   
\end{lemma}
\begin{proof}
Exercise.
\end{proof}
Consequently, as in \cite{pesquet2021learning}, we could choose an architecture 
\begin{align*}
D_\theta = \frac{1}{2}(I + Q_\theta),
\end{align*}
where $Q$ is for example a U-Net with a constrained Lipschitz constant. The training objective can be formulated as
\beq
\min_\theta \sum_{v \in\mathcal{T}} \norm{D_\theta(v + \tilde{\delta}\, \xi) - v}^2, \text{ s.t. } Q_\theta\text{ has Lipschitz constant bounded by } 1
\eeq
where $\mathcal{T} = \{v_1, \ldots, v_T\}$ and $\xi\sim\mathcal{N}(0, I)$. The constraint that $Q_\theta$ should be non-expansive cannot be trivially enforced. If one doesn't want to restrict the architecture of $Q_\theta$ one could instead add a gradient penalty of the form 
\beq
\min_\theta \sum_{v \in\mathcal{T}} \norm{D_\theta(v + \tilde{\delta}\, \xi) - v}^2 + \lambda \max\{\norm{\nabla Q_\theta(\tilde{v})}^2, 1- \iota\},
\eeq
where the point $\tilde{v}$ is chosen depending on the data point $v$ and $\iota\in(0,1)$ is an additional parameter. For more details on the algorithm, we refer to \cite{pesquet2021learning}.

\begin{remark}{}{}
A different approach can be derived with so-called \alert{gradient-step denoisers}
\beq
D_\theta = I - \nabla q_\theta,
\eeq
where $q_\theta(v):=\frac{1}{2}\norm{v - Q_\theta(v)}^2$. We refer to \cite{hurault2021gradient} for more details on this approach.
\end{remark}

\ifpub%
\else%

\fi

\printbibliography

\end{document}